\documentclass[a4paper, 11pt]{amsart}
\usepackage[english]{babel}
\usepackage{csquotes}
\usepackage{amsmath,amssymb,dsfont,mathrsfs,anysize,fancyhdr,epsfig, comment, hyperref}

\usepackage{enumerate}
\usepackage{xcolor}
\usepackage{cleveref}
\usepackage{tikz}
\usepackage{caption}
\usepackage{subfig}

\newcommand{\NN}{\mathbb{N}}
\newcommand{\ZZ}{\mathbb{Z}}

\newcommand{\RR}{\mathbb{R}}

\newcommand{\CC}{\mathbb{C}}

\newcommand{\oper}[1]{\operatorname{#1}}

\newcommand{\calA}{\mathcal{A}}

\newcommand{\calI}{\mathcal{I}}
\newcommand{\calJ}{\mathcal{J}}

\newcommand{\calL}{\mathcal{L}}
\newcommand{\calN}{\mathcal{N}}

\newcommand{\calS}{\mathcal{S}}
\newcommand{\cS}{\mathcal{S}}

\newcommand{\xx}{\mathbf{x}}
\newcommand{\nn}{\mathbf{n}}
\newcommand{\rr}{\mathbf{r}}

\newcommand{\vv}{\mathbf{v}}
\newcommand{\zz}{\mathbf{z}}
\newcommand{\uu}{\mathbf{u}}
\newcommand{\ww}{\mathbf{w}}

\DeclareMathOperator{\Dom}{Dom}

\DeclareMathOperator{\supp}{supp}
\DeclareMathOperator{\Span}{Span}

\newcommand{\Cliq}{\mathsf{Cliq}}
\newcommand{\Graph}{\mathsf{Graph}}

\newcommand{\floor}[1]{\lfloor #1 \rfloor}
\newcommand{\ceil}[1]{\lceil #1 \rceil}

\newtheorem{theorem}{Theorem}[section]
\newtheorem{lemma}[theorem]{Lemma}
\newtheorem{corollary}[theorem]{Corollary}

\newtheorem{definition}[theorem]{Definition}
\newtheorem{proposition}[theorem]{Proposition}
\newtheorem{assumption}[theorem]{Assumption}

\theoremstyle{remark}

\newtheorem{remark}[theorem]{Remark}

\newtheorem{problem}[theorem]{Problem}

\title[Bimodule coefficients, Riesz transforms on Coxeter groups and strong solidity]{Bimodule coefficients, Riesz transforms on Coxeter groups and strong solidity}

\date{\noindent \today.  MSC2010 keywords: 46L10, 20F55.  MC is supported by the NWO Vidi grant `Non-commutative harmonic analysis and rigidity of operator algebras', VI.Vidi.192.018. MW was supported by the Research Foundation – Flanders (FWO) through a Postdoctoral
Fellowship and by long term structural funding - Methusalem grant of the Flemish Government. }

\author[Matthijs Borst]{Matthijs Borst}
\author[Martijn Caspers]{Martijn Caspers}
\address{TU Delft, EWI/DIAM,
	P.O.Box 5031,
	2600 GA Delft,
	The Netherlands}
\email{m.j.borst@outlook.com}
\email{m.p.t.caspers@tudelft.nl}

\author[Mateusz Wasilewski]{Mateusz Wasilewski}
\address{KU Leuven,
Celestijnenlaan 200B,
3001 Leuven,
Belgium}
\email{mateusz.wasilewski@kuleuven.be}

\begin{document}
	
	\begin{abstract}
In deformation-rigidity theory it is often important to know whether certain  bimodules are weakly contained in the coarse bimodule. Consider a bimodule $H$ over the group algebra $\mathbb{C}[\Gamma]$, with $\Gamma$ a discrete group. The starting point of this paper is that if a dense set of the so-called coefficients of $H$  is contained in the Schatten $\calS_p$ class $p \in [2, \infty)$ then the $n$-fold tensor power $H^{\otimes n}_\Gamma$ for $n \geq p/2$ is quasi-contained in the coarse bimodule.  We apply this to gradient bimodules associated with the  carr\'e du champ  of a symmetric quantum Markov semi-group.

 For Coxeter groups we give a number of characterizations of having coefficients in $\calS_p$ for the gradient bimodule constructed from the word length function. We get equivalence of: (1) the gradient-$\calS_p$ property introduced by the second named author, (2)  smallness at infinity of a natural compactification of the Coxeter group, and for a large class of Coxeter groups: (3) walks in the Coxeter diagram called parity paths.

We derive several strong solidity results. In particular, we extend current strong solidity results for right-angled Hecke von Neumann algebras beyond right-angled Coxeter groups that are small at infinity. Our general methods also yield a concise proof of a result by T. Sinclair for discrete groups admitting a proper cocycle into a $p$-integrable representation.

	\end{abstract}

	\maketitle


	\section{Introduction}

This paper establishes bridges between the Riesz transform in modern harmonic analysis  and  von Neumann algebra theory. The original Riesz transform   can be defined as follows. Consider the   positive unbounded Laplace operator $\Delta$ and the directional gradient $\nabla_j$ on $L_2(\mathbb{R}^n)$  given by
\[
\Delta = - \sum_{j=1}^n \frac{\partial^2}{ \partial x_j^2}, \qquad  \nabla_j = \frac{\partial}{\partial x_j}, \qquad 1 \leq j \leq n.
\]
     Then the Riesz transform $R_j = \nabla_j \circ \Delta^{- \frac{1}{2}}$ for $1 \leq j \leq n$ is an isometry on $L_2(\mathbb{R}^n)$ that has been studied extensively in classical harmonic analysis in the context of Fourier multipliers, singular integral operators and  Calder\'on-Zygmund theory.

     Riesz transforms can be defined abstractly for any $C_0$-semi-group of positive measure preserving unital contractions on $L_\infty(X, \mu)$, with $(X, \mu)$ a finite Borel measure space. Such semi-groups admit a generator $\Delta$ and a natural replacement of the gradient $\nabla$ known as the {\it carr\'e du champ}. The Riesz transform is then defined as $\nabla \circ \Delta^{-\frac{1}{2}}$. These Riesz transforms were studied by Meyer  \cite{Meyer} for (commutative) Gaussian algebras and their study was continued by Bakry \cite{Bakry1}, \cite{Bakry2}, Gundy \cite{Gundy}, Pisier \cite{Pisier}, amongst others. This in particular involves an analysis of diffusion semi-groups on compact Riemannian manifolds with lower bounds on the Ricci curvature \cite{Bakry2}. In the non-commutative situation Clifford algebras were considered by Lust-Piquard \cite{LustPiquard1}, \cite{LustPiquard2}.  Also recently the Riesz transform was studied on general groups \cite{JMP} using certain multipliers associated with cocycles.

     In this paper we study Riesz transforms associated with non-commutative generalizations of diffusion semi-groups: (symmetric) quantum Markov semi-groups. Let $M$ be a finite von Neumann algebra and $\Phi = (\Phi_t)_{t \geq 0}$ a point-strongly continuous semi-group of trace preserving unital completely positive maps.  Such a semi-group comes with a generator $\Delta$. The proper replacement of the gradient is played by a bilinear form that is a non-commutative version of the carr\'e du champ. For simplicity we consider mostly quantum Markov semi-groups of Fourier multipliers associated with a discrete group $\Gamma$, acting on the group algebra $\mathbb{C}[\Gamma]$. Then the carr\'e du champ allows the construction of a  $\mathbb{C}[\Gamma]$ bimodule $H_\nabla$ and a derivation, i.e. a map satisfying the Leibniz rule, $\nabla: \mathbb{C}[\Gamma] \rightarrow H_\nabla$ such that (here formally) $\Delta = \nabla^\ast \nabla$. So $\nabla$ is a root of $\Delta$ just as in the case of the Laplace operator and the gradient. We refer to Cipriani-Sauvageot \cite{CiprianiSauvageot} where also the analytical framework is established.   Then there is an isometry $\nabla \circ \Delta^{- \frac{1}{2}}: \ell_2(\Gamma) \rightarrow H_\nabla$ called the {\it Riesz transform}. This Riesz transform was studied in the context of $q$-Gaussian algebras \cite{caspersL2CohomologyDerivationsQuantum2021}, \cite{LustPiquard1}, \cite{LustPiquard2} and compact quantum groups \cite{caspersRieszTransformCompact2021}, \cite{caspersGradientFormsStrong2021}.

     In the current paper we are interested in applications of the Riesz transform to group von Neumann algebras of discrete groups; we focus on Coxeter groups but we also obtain results for other groups.

\vspace{0.3cm}

      Recall that to a discrete group $\Gamma$ we may associate the group von Neumann algebra $\calL(\Gamma)$ which is the von Neumann algebra generated by the left regular representation.
          Let $\mathbb{F}_2$ be the free group with two generators.
     In his fundamental papers on free probability Voiculescu \cite{Voiculescu} showed  that $\calL(\mathbb{F}_2)$ does not possess a Cartan subalgebra, meaning that there does not exist a maximal abelian subalgebra (MASA) of $\calL(\mathbb{F}_2)$ whose normalizer generates $\calL(\mathbb{F}_2)$. An important consequence is that  $\calL(\mathbb{F}_2)$ does not non-trivially decompose as a crossed product and cannot be constructed from an equivalence relation with a cocycle as was shown by Feldman and Moore \cite{FeldmanMoore1}, \cite{FeldmanMoore2}. In \cite{OzawaPopaAnnals} Ozawa and Popa gave an alternative proof of the Voiculescu's result. They showed that $\calL(\mathbb{F}_2)$ is strongly solid. This means that the normalizer of any diffuse amenable von Neumann subalgebra of $\calL(\mathbb{F}_2)$ generates a von Neumann algebra that is amenable again. Since $\calL(\mathbb{F}_2)$  is nonamenable and since MASA's are diffuse it automatically follows that $\calL(\mathbb{F}_2)$ does not possess a Cartan subalgebra. After  \cite{OzawaPopaAnnals} many von Neumann algebras were proven to be strongly solid, see e.g.  \cite{isonoExamplesFactorsWhich2015}, \cite{OzawaPopaAJM}, \cite{PopaVaesCrelle} and references given there.
          As a consequence of the methods in this paper we are able to prove such strong solidity results as well.

\vspace{0.3cm}

To motivate the first part of this paper we recall the following theorem from \cite{caspersL2CohomologyDerivationsQuantum2021}. We do not explain for now the technical terms that occur in this theorem but in the subsequent paragraph we explain what the crucial part is.  Theorem \ref{Thm=RieszImpliesAOIntro} itself is actually not that hard to prove; however its consequences (see \cite{PopaVaesCrelle}, \cite{isonoExamplesFactorsWhich2015}) and proving that its assumptions hold in examples is rather intricate.

\begin{theorem}[Proposition 5.2 of \cite{caspersL2CohomologyDerivationsQuantum2021}]  \label{Thm=RieszImpliesAOIntro}
Let $H$ be a $\mathbb{C}[\Gamma]$ bimodule and let $V: \ell_2(\Gamma) \rightarrow H$ be bounded. Assume that $H$ is quasi-contained in the coarse bimodule of $\Gamma$,  that $V$ is almost bimodular and that $V^\ast V$ is Fredholm. Assume that $C_r^\ast(\Gamma)$ is locally reflexive.  Then $\calL(\Gamma)$ satisfies AO$^+$.
\end{theorem}


The Akemann-Ostrand property AO$^+$ (as in \cite{isonoExamplesFactorsWhich2015}) will be used frequently in this paper for which we refer to Definition \ref{Dfn=AO}. If $\Gamma$ is weakly amenable then AO$^+$ implies strong solidity \cite{PopaVaesCrelle}, \cite{isonoExamplesFactorsWhich2015}.  The Coxeter groups in this paper are weakly amenable \cite{fendlerWeakAmenabilityCoxeter2002}, \cite{Janus} as are all hyperbolic discrete groups \cite{OzawaHyperbolic}.

 In view of  Theorem \ref{Thm=RieszImpliesAOIntro} we are mostly still interested in two things: (1) constructing almost bimodular maps $V: \ell_2(\Gamma) \rightarrow H$   with  $H$ a $\mathbb{C}[\Gamma]$ bimodule;  (2) showing that the $\mathbb{C}[\Gamma]$ bimodule $H$ is quasi-contained in the coarse bimodule $\ell_2(\Gamma) \otimes \ell_2(\Gamma)$ of $\Gamma$.
   It turns out that very often the Riesz transform is an almost bimodular map. Further, we provide comprehensible conditions that show that the gradient bimodule is quasi-contained in the coarse bimodule. We will develop general theory for this as follows.

\vspace{0.3cm}

In the first part of this paper we study bimodules over $\mathbb{C}[\Gamma]$ and their {\it coefficients}. We define coefficients of a $\mathbb{C}[\Gamma]$ bimodule as a certain map $\mathbb{C}[\Gamma] \rightarrow \mathbb{C}[\Gamma]$. This notion occurs for instance in \cite[Section 13]{PopaDelaRocheBook} for von Neumann algebras; the more algebraic notion we present here is more convenient for our purposes.   Since $\mathbb{C}[\Gamma] \subseteq \ell_2(\Gamma)$ a coefficient determines a densely defined map $\ell_2(\Gamma) \rightarrow \ell_2(\Gamma)$. We study when these maps are contained in the Schatten von Neumann non-commutative $L_p$-space $\calS_p$.

For two  $\mathbb{C}[\Gamma]$ bimodules $H_1$ and $H_2$  we shall also show that $H_1 \otimes H_2$ has  a natural $\mathbb{C}[\Gamma]$ bimodule structure and we denote this bimodule by $H_1 \otimes_\Gamma H_2$. As a Hilbert space $H_1 \otimes_\Gamma H_2 = H_1 \otimes H_2$. Recall that the coarse bimodule of $\Gamma$ is given by $\ell_2(\Gamma) \otimes \ell_2(\Gamma)$ where the left action of $\mathbb{C}[\Gamma]$ is on the first tensor leg and the right action on the second tensor leg.
 In Section \ref{Sect=Coefficients} we prove the following, amongst other results (except for part {\it (4)}, which is proved in Section \ref{Sect=RieszEtc}, see Corollary \ref{Cor=RieszBimodular}).

\begin{theorem}\label{Theorem=IntroA}
Let $H, H_1$ and $H_2$ be $\mathbb{C}[\Gamma]$ bimodules.
\begin{enumerate}
\item If a dense set of coefficients of $H$ are in $\cS_2$ then $H$ is a  $\calL(\Gamma)$ bimodule that is quasi-contained in the coarse bimodule of $\Gamma$.
\item If a dense set of coefficients of $H_i, i=1,2$ is contained in $\cS_{p_i}, p_i \in [1, \infty)$ then a dense set of coefficients of $H_1 \otimes_\Gamma H_2$ is contained in  $\cS_{p}$ where $\frac{1}{p} = \frac{1}{p_1} + \frac{1}{p_2}$.
\item If $V_i: \ell_2(\Gamma) \rightarrow H_i, i=1,2$ is almost $\mathbb{C}[\Gamma]$ bimodular then so is $V_1 \ast V_2 := (V_1 \otimes V_2) \circ \Delta_\Gamma: \ell_2(\Gamma) \rightarrow H_1 \otimes_\Gamma H_2$ where $\Delta_\Gamma: \ell_2(\Gamma) \rightarrow \ell_2(\Gamma) \otimes \ell_2(\Gamma)$ is the comultiplication.
\item Consider a proper length function $\psi: \Gamma \rightarrow \mathbb{Z}_{\geq 0}$ that is conditionally of negative type, defined on a finitely generated group $\Gamma$. Then the associated Riesz transform $R: \ell_2(\Gamma) \rightarrow \ell_2(\Gamma)_\nabla$ is almost bimodular.
\end{enumerate}
\end{theorem}

Theorem \ref{Theorem=IntroA} provides a clear strategy towards obtaining the input of Theorem \ref{Thm=RieszImpliesAOIntro}. Namely we start with a proper length function $\psi: \Gamma \rightarrow \mathbb{R}$ that is conditionally of negative type. We construct the associated gradient bimodule $H_\nabla$ and show that its coefficients are in $\cS_p$ for some $p \in [1, \infty)$. By tensoring we obtain a bimodule $(H_{\nabla})_\Gamma^{\otimes n }, n \geq \lceil \frac{p}{2} \rceil$ and a map
\[
V^{\ast n }: \ell_2(\Gamma) \rightarrow (H_{\nabla})_\Gamma^{\otimes n },
\]
 with the desired properties of Theorem \ref{Thm=RieszImpliesAOIntro}. This is the rough idea of our strategy. We say `rough' since in all applications we need some suitably adapted variation of this idea.

\vspace{0.3cm}

In the second part of this paper we analyse when coefficients of a gradient bimodule $H_\nabla$ are in $\calS_p, p \in [1, \infty)$. In order to do so we recall the property {\it gradient}-$\cS_p$ for quantum Markov semi-groups from \cite{caspersGradientFormsStrong2021},  \cite{caspersL2CohomologyDerivationsQuantum2021}. If a quantum Markov semi-group has gradient-$\cS_p$ then a dense set of coefficients of $H_\nabla$ are in $\cS_p$; consequently $H_\nabla$  is quasi-contained in the coarse bimodule of $\Gamma$.

  We first show (Lemma \ref{lemma:reduction-to-set-generating-the-group2}) that if $\psi: \Gamma \rightarrow \mathbb{Z}$ is a proper    length function that is conditionally of negative type then gradient $\cS_p, p \in [1, \infty)$ for the associated quantum Markov semi-group is independent of $p$. Then we analyse when the word length function of a general (finite rank) Coxeter group is gradient-$\cS_p$. We find the following characterization.

\begin{theorem}\label{Thm=ReflexionsIntro}
	Let $W = \langle S|M\rangle$ be a finite rank Coxeter system. Fix $p\in [1,\infty]$. The following are equivalent:
	\begin{enumerate}
		\item The quantum Markov semi-group  associated with the word length is gradient-$\calS_p$.
        \item   For all $s,t \in S$ the set $\{ w \in W : ws = t w \}$ is finite.
		\item The Coxeter system $\langle S|M\rangle$ is small at infinity (as in \cite{klisseTopologicalBoundariesConnected2020}).
\end{enumerate}
	\end{theorem}

In particular for right-angled Coxeter groups these statements are equivalent to the Coxeter group being a free product of finite abelian Coxeter groups, see \cite{klisseTopologicalBoundariesConnected2020}. This shows that gradient-$\calS_p$ is rather rare. However with the right tensor techniques it can still be turned into a very useful property. We also provide an almost characterization of when the equivalent statements of  Theorem \ref{Thm=ReflexionsIntro} hold in the following theorem. For the definition of the graph $\Graph_S(W)$ we refer to Definition \ref{Dfn=Graph}. The definition of a parity path is given in Definition \ref{Dfn=Parity}.

	\begin{theorem}\label{Theorem=NoParityIntro}
	Let $W=\langle S|M\rangle$ be a Coxeter group. If there does not exist a cyclic parity path in $\Graph_S(W)$ then the semi-group $(\Phi_t)_{t\geq 0}$ associated to the word length $\psi_{S}$ is gradient-$\mathcal{S}_p$ for all $p\in[1,\infty]$. The converse holds true if $m_{i,j} \not = 2$ for all $i,j$.
	\end{theorem}

  Section 	\ref{section:semi-groups-word-length}   shows that it is usually easy to determine whether  $\Graph_S(W)$ has a parity path, see Corollaries \ref{corollary:gradient-Sp-Coxeter-groups} and \ref{corollary:no-labels-two-iff}.

\vspace{0.3cm}

 We now come to the applications.  The first one is essentially the main result of \cite{Sinclair}. Now this theorem  follows rather directly from our analysis of bimodule coefficients in Section \ref{Sect=Coefficients}. The theorem in particular applies to icc lattices in the groups $SO(n,1), n \geq 3$ and $SU(m,1), m \geq 2$.  

\begin{theorem}[Application A]\label{Thm=IntroPIntegrable}
Let $\Gamma$ be a discrete group admitting a proper cocycle into a $p$-integrable representation for some $p<\infty$. Assume $C_r^\ast(\Gamma)$ is locally reflexive.  Then $\calL(\Gamma)$ has   property AO$^{+}$.
\end{theorem}

\vspace{0.3cm}

Next we obtain strong solidity results for Hecke von Neumann algebas: $q$-deformations of Coxeter groups.  
 The following theorem extends \cite[Theorem 0.7]{klisseTopologicalBoundariesConnected2020} in the case of a right-angled Coxeter system. What is of particular interest is that our methods really improve on the approach based on compactifications and boundaries in \cite{klisseTopologicalBoundariesConnected2020}. More precisely, \cite{klisseTopologicalBoundariesConnected2020} shows that if the action of a right-angled Coxeter group on a natural boundary associated with it is small at infinity, then actually the   Coxeter group is a free product of finite (commutative) Coxeter groups. So the approach in \cite[Theorem 0.7]{klisseTopologicalBoundariesConnected2020} cannot be extended to the current generality.

\begin{theorem}[Application B]\label{Thm=IntroHecke}
Let $W = \langle S | M \rangle$ be a right-angled Coxeter group and let $q = (q_s)_{s\in S}$ with $q_s >0$. Assume that  all elements in
\[
\mathcal{I} := \{ r \in S : \exists s,t \in S \textrm{ such that } m_{r,s} = m_{r,t} = 2 \textrm{ and } m_{s,t} = \infty \}
\]
commute. Then the Hecke von Neumann algebra $\calN_q(W)$ satisfies AO$^+$ and  is strongly solid.
\end{theorem}

 We note that a large part of the analysis in proving Theorem \ref{Thm=IntroHecke} applies to general Hecke algebras. However, the strong solidity properties are still pending on whether certain semi-groups extend to quantum Markov semi-groups. In the final Section \ref{Sect=Problems} of this paper we summarize some problems that are open to the knowledge of the authors.

\vspace{0.3cm}

\noindent {\bf Structure of the paper.} Section \ref{Sect=Preliminaries} contains the preliminaries.  Section \ref{Sect=Coefficients} contains results on bimodules and their coefficients. We prove Theorem \ref{Theorem=IntroA}. We also  directly obtain the first strong solidity result, namely Theorem \ref{Thm=IntroPIntegrable}. Section \ref{Sect=RieszEtc} introduces quantum Markov semi-groups, the gradient bimodule and the Riesz transform. We also derive many of the basic properties. Section \ref{section:semi-groups-word-length} proves Theorem \ref{Thm=ReflexionsIntro} and Theorem  \ref{Theorem=NoParityIntro}. Note that here we also establish the Corollaries \ref{corollary:gradient-Sp-Coxeter-groups} and \ref{corollary:no-labels-two-iff} which make it easy to see if a Coxeter group is small at infinity. Section \ref{Sect=Weights} contains an analysis of quantum Markov semi-groups with weights on the generators. This applies mostly to right-angled Coxeter groups and it is crucial in the later sections.
Section \ref{Sect=SolidTensor} proves strong solidity results for Coxeter groups using tensor methods.  In Section \ref{Sect=HeckeStronglySolid} we prove Theorem \ref{Thm=IntroHecke}. We have included Section \ref{Sect=Problems} to list some problems that are left open.

\vspace{0.3cm}

\noindent {\bf Acknowledgements.} The authors wish to express their gratitude to Mario Klisse for several valuable comments and noting the connections obtained in Section  \ref{Sect=SmallAtInfinity}.

	\section{Preliminaries}\label{Sect=Preliminaries}

Inner products are linear on the left and anti-linear on the right.

\subsection{Von Neumann algebras} For standard theory of von Neumann algebras we refer to \cite{StratilaZsido}, \cite{PopaDelaRocheBook} and \cite{TakI}. Let $B(H)$ be the bounded operators on a Hilbert space $H$.  A von Neumann algebra $M$ is a unital $\ast$-subalgebra of $B(H)$ that is closed in the strong operator topology. A von Neumann algebra is {\it finite} if it admits a faithful normal tracial state $\tau: M \rightarrow \mathbb{C}$. We will say that the pair $(M, \tau)$ is a finite von Neumann algebra. We let $L_2(M)$ be the Hilbert space completion of $M$ with respect to the inner product $\langle x, y \rangle_\tau = \tau(y^\ast x)$. Note that we suppress $\tau$ in the notation of $L_2(M)$. In case $M$ is a group von Neumann algebra (see below) $\tau$ is understood as the trace defined by \eqref{Eqn=TauGroup}. We denote $\Omega_\tau \in L_2(M)$ for the element $1 \in M$ identified within $L_2(M)$. A map between von Neumann algebras is called {\it normal} if it is strongly continuous on the unit ball.

\subsection{Operator spaces} For operator spaces we refer to \cite{EffrosRuan}, \cite{PisierBook}. A map $\Phi: M \rightarrow M$ on a von Neumann algebra $M$ is called completely positive if for every $n \in \mathbb{N}$ the map ${\rm id}_n \otimes \Phi: M_n(\mathbb{C}) \otimes M \rightarrow  M_n(\mathbb{C}) \otimes M$ maps positive elements to positive elements.

\hyphenation{a-me-na-ble}
\hyphenation{boun-ded}

\subsection{Approximation properties} A von Neumann algebra $M$ has the weak-$\ast$ completely bounded approximation property if there exists a net of normal completely bounded finite rank maps $\Phi_i: M \rightarrow M, i \in I$ such that $\sup_i \Vert \Phi_i \Vert_{cb} < \infty$ and for every $x \in M$ we have $\Phi_i(x) \rightarrow x$ in the $\sigma$-weak topology. If the $\Phi_i$ can moreover be chosen to be unital and completely positive then $M$ is called {\it amenable}. We refer to \cite{brownAlgebrasFiniteDimensionalApproximations2008} for further equivalent notions of amenability.

\subsection{Bimodules and containment} A {\it bimodule} over an algebra $\calA$ is a Hilbert space $H$ with commuting actions of $\calA$ and the opposite algebra $\calA^{{\rm op}}$. For $x,y \in \calA, \xi \in H$ we denote $x \cdot \xi \cdot y$ or $x \xi y$ for the left action of $x$ and the right action of $y$ on the vector $\xi$.
  In case $\calA$ is also a C$^\ast$-algebra we require that both actions are continuous as maps $\calA \rightarrow B(H)$ (and therefore contractive). In case $\calA$ is a von Neumann algebra we require both actions to be normal. We refer to these bimodules as $\calA$ bimodules and it should be clear from the context whether this is a bimodule over a $\ast$-algebra, C$^\ast$-algebra or von Neumann algebra.

 We say that an $\calA$ bimodule $H$ is {\it contained} in an $\calA$ bimodule $K$ if $H$ is a Hilbert subspace of $K$ that is invariant under the left and right action of $\calA$. We say that $H$ is {\it quasi-contained} in $K$ if $H$ is contained in $\oplus_{i \in I} K$ for some index set $I$ (if  $H$ is separable we may choose $I = \mathbb{N}$). We say that $H$
 is {\it weakly contained} in $K$ if for every $\epsilon > 0$, every finite set $F \subseteq \calA$ and every $\xi \in H$ there exist finitely many $\eta_j \in K$ indexed by $j \in G$ such that for $x,y \in F$,
 \[
 \vert \langle x \xi y, \xi \rangle - \sum_{j \in G} \langle x \eta_j y , \eta_j \rangle \vert < \epsilon.
 \]
 Containment implies quasi-containment which implies weak containment. In this paper we mostly deal with quasi-containment though in most of our applications a weak containment would be sufficient.

Let $M$ be a finite von Neumann algebra. Then  $M$ acts on $L_2(M)$ by left and right multiplication. This turns $L_2(M)$ into an $M$ bimodule called the {\it trivial bimodule}. Similarly $L_2(M) \otimes L_2(M)$ has a bimodule structure by extending
\[
x (\xi \otimes \eta) y = x \xi \otimes \eta y, \qquad x,y \in M, \xi, \eta \in L_2(M).
\]
The $M$ bimodule thus obtained is called the {\it coarse bimodule}.

\subsection{Schatten classes} Let $H$ be a Hilbert space. Define $\calS_{\infty} = \calS_{\infty}(H)$ as the space of compact operators on $H$.  For $p \in (0, \infty)$ we define $\cS_p = \cS_p(H)$ as the space of all $x \in B(H)$ for which
\begin{equation}\label{Eqn=SpNorm}
\Vert x \Vert_p := {\rm Tr}(   \vert x \vert^p )^{\frac{1}{p}} = \left( \sum_{i \in I} \langle \vert x \vert^p e_i, e_i \rangle \right)^{\frac{1}{p}}
\end{equation}
is finite where $e_i, i\in I$, is any orthonormal basis of $H$.
If $p \in [1, \infty]$ then $\|\cdot\|_p$ defines a norm turning $\calS_p$ into a Banach space that is moreover a 2-sided ideal in $B(H)$.

\subsection{Group algebras} Let $\Gamma$ be a discrete group. We denote $e$ for the identity of $\Gamma$.  Let
\[
\Gamma \mapsto   B(   \ell_2(\Gamma)   ): s  \mapsto \lambda_s,
\]
 be the left regular representation where $\lambda_s \delta_t = \delta_{st}$ and where $\delta_t$ is the delta function at $t \in \Gamma$. The {\it group algebra}  $\mathbb{C}[\Gamma]$ is the $\ast$-algebra generated by $\lambda_s, s \in \Gamma$.
  The {\it reduced group C$^\ast$-algebra} $C_r^\ast(\Gamma)$ is the norm closure of $\mathbb{C}[\Gamma]$.   The group von Neumann algebra $\calL(\Gamma)$ is the strong operator topology closure of $\mathbb{C}[\Gamma]$.  $\calL(\Gamma)$ is finite with faithful normal tracial state
 \begin{equation}
 \label{Eqn=TauGroup}
 \tau(x) = \langle x \delta_e, \delta_e \rangle, \qquad x \in \calL(\Gamma).
 \end{equation}

 Note that we have an identification as Hilbert spaces $L_2(\calL(\Gamma)) \simeq \ell_2(\Gamma)$ by $x \mapsto x \delta_e$ with $x \in \mathbb{C}[\Gamma]$. Under this identification $\ell_2(\Gamma)$ is the trivial bimodule with actions given by the left and right regular representations $\lambda$ and $\rho$.  The coarse bimodule is then given by $\ell_2(\Gamma) \otimes \ell_2(\Gamma)$ with left and right actions given by
 \[
 x \cdot (\xi \otimes \eta) \cdot y = (x \xi) \otimes (\eta y), \xi, \eta \in \ell_2(\Gamma).
 \]
 We simply call $\ell_2(\Gamma) \otimes \ell_2(\Gamma)$ with these bimodule actions the coarse bimodule of $\Gamma$. We also summarize that
 \[
 \Gamma \subseteq \mathbb{C}[\Gamma] \subseteq C_r^\ast(\Gamma) \subseteq \calL(\Gamma) \subseteq \ell_2(\Gamma),
 \]
 where the first inclusion is given by $s \mapsto \lambda_s$ and the others were discussed above.

  \subsection{Hyperbolic groups}
  Let $(V, E)$ be a graph with vertex set $V$ and edge set $E$. For $v,w \in V$ a geodesic from $v$ to $w$ is a shortest path in the graph. For $G \subseteq V$ and $\delta > 0$ we define the $\delta$-neighbourhood of $G$ as all points in $V$ for which there exists a geodesic of length at most $\delta$ to a point in $G$.    $(V, E)$ is {\it hyperbolic} if there exists $\delta > 0$ such that for every 3 vertices $v,w,u \in V$ and for all geodesics $[v, w], [w, u]$ and $[u,v]$ between these vertices, we have that $[u,v]$ lies in the $\delta$-neighbourhood of $[v,w] \cup [w,u]$.

  Let $\Gamma$ be a  finitely generated (discrete) group. $\Gamma$ is  {\it hyperbolic} or {\it word hyperbolic}  if its Cayley graph is hyperbolic; this definition is independent of the finite generating set that is used to construct the Cayley graph.  We emphasize that in this paper hyperbolic and word hyperbolic mean the same thing. The terminology `word hyperbolic' is more common in the theory of Coxeter groups.

\subsection{Functions on groups} Let $\Gamma$ be a discrete group.
A {\it length function}  is a function $\psi: \Gamma \rightarrow \mathbb{R}_{\geq 0}$ satisfying $\psi(uv) \leq \psi(u) + \psi(v)$ for all $u,v \in \Gamma$. If $\Gamma$ is generated by a finite set  $S$ then a typical length function is defined by $\psi(w) = n$ where $w = s_1 \ldots s_n$ is the shortest way of writing $w$ as a product of generators $s_i \in S$. Note that $\psi(w)$ is the distance from $w$ to $e$ in the Cayley graph of $\Gamma$.
 A function $\psi: \Gamma \rightarrow \mathbb{R}$ is called {\it conditionally of negative type} (also known as {\it conditionally negative definite})  if $\psi(e) = 0$, $\psi(s) = \psi(s^{-1}), s \in \Gamma$ and for all $n \in \mathbb{N}$ and $s_1, \ldots, s_n \in \Gamma$ and real numbers $c_1, \ldots, c_n$ with $\sum_{i=1}^n c_i = 0$ we have
 \[
 \sum_{i=1}^n \sum_{j=1}^n  c_i c_j \psi(s_j^{-1} s_i) \leq 0.
 \]
 In this paper we shall frequently work with length functions that are conditionally of negative type. A function $\psi: \Gamma \rightarrow \mathbb{R}$ is called {\it proper} if the inverse image of a compact set is compact (hence finite as $\Gamma$ is discrete).

\subsection{Tensor products} With mild abuse of notation we use $\otimes$ for several different tensor products in this paper. If $V_1$ and $V_2$ are vector spaces then $V_1 \otimes V_2$ is the tensor product of these vector spaces. If $V_1$ and $V_2$ are algebras or $\ast$-algebras then we see $V_1 \otimes V_2$ as an algebra or $\ast$-algebra as well. When $V_1$ and $V_2$ are Hilbert spaces $V_1 \otimes V_2$ is the Hilbert space tensor product (closure of the vector space tensor product) and it should be understood from the context which tensor product is meant. We further use $\otimes$ to denote tensor products of linear maps or elements. In case $V_1$ and $V_2$ are C$^\ast$-algebras we will write $\otimes_{{\rm alg}}$ instead of $\otimes$ for their tensor product as a $\ast$-algebra and $\otimes_{{\rm min}}$ for their spatial tensor product; this is also the minimal tensor product by Takesaki's theorem. If $V_1$ and $V_2$ are von Neumann algebras we denote by $V_1 \overline{\otimes} V_2$ the von Neumann algebraic tensor product (strong operator topology closure of the spatial tensor product).

\section{Coefficients of bimodules} \label{Sect=Coefficients}

In this section we study bimodules over the group algebra of a discrete group and provide sufficient criteria for when such a bimodule is quasi-contained in the coarse bimodule. We also consider tensor products of such bimodules. We conclude this section with our first strong solidity result in  Section \ref{Sect=PInt}.

\subsection{Coefficients and quasi-containment}
Let  $\Gamma$ be a discrete group with group algebra $\CC[\Gamma]$, reduced group C$^\ast$-algebra $C_r^\ast(\Gamma)$ and group von Neumann algebra $\calL(\Gamma)$. They include naturally
\[
\mathbb{C}[\Gamma] \subseteq C_r^\ast(\Gamma) \subseteq \calL(\Gamma).
\]
 In turn $\calL(\Gamma) \subseteq \ell_2(\Gamma)$ by $x \mapsto x \delta_e$. Hence we may and will view $\mathbb{C}[\Gamma]$ as the subspace of $\ell_2(\Gamma)$ of functions with finite support. Now a $\CC[\Gamma]$ \emph{bimodule} will be a Hilbert space $H$ with commuting left and right actions of $\Gamma$ and thus of $\CC[\Gamma]$ by extending the actions linearly. 

\begin{definition}[Coefficients]\label{Dfn=Coefficients}
Let  $H$ be a $\CC[\Gamma]$ bimodule. Let $\xi,\eta\in H$ be such that there exists a map  $T_{\xi,\eta}: \CC[\Gamma] \to \CC[\Gamma]$ such that
\begin{equation}\label{Eqn=Coefficient}
\tau(T_{\xi,\eta}(x)y) = \langle x\xi y,\eta\rangle, \qquad x,y\in \CC[\Gamma].
\end{equation}
$T_{\xi, \eta}$ is called the \textit{coefficient} of $H$ at $\xi,\eta$. Set  $T_{\xi} := T_{\xi,\xi}$. The coefficient $T_{\xi,\eta}$ is in $\calS_p$ with $p\in[1,\infty]$ if $T_{\xi, \eta}$ exists and extends to a bounded operator $T_{\xi,\eta}: \ell_2(\Gamma) \to \ell_2(\Gamma)$ that is moreover in $\calS_p :=  \calS_p(\ell_2(\Gamma))$.
\end{definition}

 Note that if the map $T_{\xi, \eta}$ is existent then it is uniquely determined by \eqref{Eqn=Coefficient}. Indeed, if  $T_{\xi,\eta}'$ is another map with this property then $\tau((T_{\xi,\eta}- T_{\xi,\eta}')(x)y) = 0$ for all $x,y\in \CC[\Gamma]$ so that $T_{\xi,\eta} ' =  T_{\xi,\eta}$.

\begin{remark}
In \cite[Definition 13.1.6]{PopaDelaRocheBook} the notion of a coefficient of a von Neumann bimodule is defined.  Definition \ref{Dfn=Coefficients} is an   algebraic analogue which is more convenient for our purposes. The reason that we work in this algebraic setting is that the bimodules we consider in this paper are {\it a priori} not necessarily von Neumann bimodules. In fact for the gradient bimodules we consider in Section \ref{Sect=RieszEtc} this is not even true in general. However, under the conditions of Proposition \ref{proposition:quasicontainment} the normal extensions of the left and right actions automatically exist.
\end{remark}

\begin{proposition}[Quasi-containment]\label{proposition:quasicontainment}
	Let $H$ be a $\mathbb{C}[\Gamma]$ bimodule. Suppose that there exists a dense subset $H_{0} \subset H$ such that for any $\xi \in H_{0}$ the coefficient $T_{\xi}: \CC[\Gamma] \to \CC[\Gamma]$ is in $\calS_2$. Then the left and right $\mathbb{C}[\Gamma]$ actions on $H$ extend to (bounded) normal $\calL(\Gamma)$ actions and the $\calL(\Gamma)$ bimodule $H$ is quasi-contained in the coarse bimodule $\ell_2(\Gamma) \otimes \ell_2(\Gamma)$.
\end{proposition}
\begin{proof}
Take $\xi \in H_0$. Define the functional
\[
\rho: \CC[\Gamma] \otimes_{{\rm alg}} \CC[\Gamma]^{{\rm op}} \rightarrow \mathbb{C}: x \otimes y^{{\rm op}}  \mapsto \langle x \cdot \xi \cdot y, \xi \rangle.
\]
For $x,y \in \CC[\Gamma]$ by definition of $T_\xi$,
\[
\rho(x \otimes y^{{\rm op}} ) = \langle x \cdot \xi \cdot y, \xi \rangle =  \tau(  T_{\xi}(x) y ) =    \tau(  y T_{\xi}(x) )  = \langle T_{\xi}(x), y^\ast \rangle_\tau.
\]
Now as $T_\xi$ is Hilbert-Schmidt there exists a vector $\zeta_\xi \in \ell_2(\Gamma) \otimes \ell_2(\Gamma)$ such that
\[
\rho(x \otimes y^{{\rm op}} ) = \langle x \otimes  y^{ {\rm op} }   , \zeta_\xi \rangle = \langle (x \otimes  y^{ {\rm op} })  \cdot (1 \otimes 1),  \zeta_\xi \rangle.
\]
This shows that $\rho$ extends contractively to  $C_r^\ast(\Gamma) \otimes_{{\rm min}} C_r^\ast(\Gamma)$. Moreover, this shows that $\rho$ extends to a normal contractive map on the von Neumann algebraic tensor product $\calL(\Gamma) \overline{\otimes} \calL(\Gamma) \rightarrow \mathbb{C}$. By Kaplansky's density theorem this extension of $\rho$ is moreover positive. Since $\ell_2(\Gamma) \otimes \ell_2(\Gamma)$ is the standard form of $\calL(\Gamma) \overline{\otimes} \calL(\Gamma)^{{\rm op}}$ there exists $\eta \in \ell_2(\Gamma) \otimes \ell_2(\Gamma)$ such that
\[
\rho(x \otimes y^{{\rm op}} ) = \langle x \cdot \eta \cdot y, \eta \rangle, \qquad x,y \in \calL(\Gamma).
\]
This proves that the conditions of  \cite[Lemma 2.2]{caspersL2CohomologyDerivationsQuantum2021} are fulfilled and hence $H$ is quasi-contained in the coarse bimodule. We already observed in the preliminaries that this quasi-containment implies that the left and right actions extend to normal actions of $\calL(\Gamma)$.
\end{proof}

A subset $H_{00} \subseteq H$  of a $\mathbb{C}[\Gamma]$ bimodule $H$  is called {\it cyclic} if $H_0 :=  {\rm span}  \CC[\Gamma] H_{00} \CC[\Gamma]$ is dense in $H$. The following lemma tells us that we can reduce Proposition \ref{proposition:quasicontainment} to checking the property only for the coefficient in a cyclic subset.
\begin{lemma}[Reduction to cyclic subset]\label{lemma:reduction-to-cyclic-subset}
	Suppose that $H_{00} \subseteq H$ is a subset whose coefficients $T_{\xi, \eta}$ for  $\xi, \eta \in H_{00}$   are in $\calS_2$. Then the coefficients $T_{\xi, \eta}$ for  $\xi, \eta \in H_0 :=  {\rm span}  \CC[\Gamma] H_{00} \CC[\Gamma]$   are in $\calS_2$. Consequently, if $H_{00}$ is cyclic then $H$ is a $\calL(\Gamma)$ bimodule that is quasi-contained in the coarse bimodule $\ell_2(\Gamma) \otimes \ell_2(\Gamma)$.
\end{lemma}
\begin{proof}
	Let $\xi' = \lambda_g \xi \lambda_h$ and $\eta' =   \lambda_s \eta \lambda_t$ for some $g, h, s,  t \in \Gamma$ and $\xi,\eta \in H_{00}$.
	We have that
	\[
\begin{split}
	 & \tau(T_{\xi',\eta'}(x)y) = \langle x\xi'y,\eta'  \rangle = \langle x\lambda_g\xi \lambda_hy,\lambda_s\eta \lambda_t\rangle
	= \langle \lambda_{s^{-1}}x\lambda_g\xi\lambda_h y \lambda_{t^{-1}},\eta\rangle \\
	= & \tau(T_{\xi,\eta}(\lambda_{s^{-1}}x\lambda_g) \lambda_hy\lambda_{t^{-1}})
	= \tau(\lambda_{t^{-1}}T_{\xi,\eta}(\lambda_{s^{-1}}x\lambda_g)\lambda_h y).
\end{split}	
\]
 This shows that $T_{\xi',\eta'}(x) = \lambda_{t^{-1}}T_{\xi,\eta}(\lambda_{s^{-1}}x\lambda_g)\lambda_h$ and so $T_{\xi',\eta'}$ is in $\calS_2$. The first statement then follows by linearity. By Proposition \ref{proposition:quasicontainment} we find that $H$ is quasi-contained in the coarse bimodule $\ell_2(\Gamma)\otimes \ell_2(\Gamma)$ in case $H_{00}$ is cyclic.
\end{proof}

\subsection{Tensoring bimodules}\label{Sect=TensoringBimodules}

 If $H_1$ and $H_2$  are two $\mathbb{C}[\Gamma]$ bimodules then we can construct a bimodule $H_1 \otimes_{\Gamma} H_2$, which, as a Hilbert space, is the same as $H_1\otimes H_2$ and the actions are given by
 \[
s  \cdot (\xi\otimes \eta):= s   \xi\otimes s \eta \qquad \textrm{ and } \qquad (\xi\otimes \eta) s := \xi s \otimes \eta s, \qquad \xi \in H_1, \eta \in H_2, s \in \Gamma.
 \]
 The actions extend by linearity to actions of $\CC[\Gamma]$. If we take an $n$-fold tensor power of a given bimodule $H$, it will be denoted by $H^{\otimes n}_{\Gamma}$. For later use we also recall that the comultiplication
 \[
 \Delta_\Gamma \colon \CC[\Gamma] \to \CC[\Gamma] \otimes \CC[\Gamma]
 \]
  is given by the linear extension of the assignment $\gamma \mapsto \gamma \otimes \gamma, \gamma \in \Gamma$. Then $\Delta_\Gamma$ extends to an isometry $\ell_2(\Gamma) \rightarrow \ell_2(\Gamma) \otimes \ell_2(\Gamma)$ which we still denote by $\Delta_\Gamma$.

	\begin{lemma}\label{Lem=Fix}
		Let $1 \leq p,q,r \leq \infty$ with $\frac{1}{r} = \frac{1}{p} + \frac{1}{q}$.   There exists a bounded bilinear map
		\[
		\calS_p \times \calS_q \rightarrow \calS_r: (x,y) \mapsto \Delta_\Gamma^{\ast}(x \otimes y) \Delta_\Gamma.
		\]
	\end{lemma}
	\begin{proof}
		For $r = 1$ take $(x,y) \in \calS_p \times \calS_q$ both positive so that $\Delta_\Gamma^{\ast}(x \otimes y) \Delta_\Gamma \in \calS_1$ is positive. Then 
		\[
		\begin{split}
			\Vert  \Delta_\Gamma^{\ast}(x \otimes y) \Delta_\Gamma \Vert_r = & \tau(  \Delta_\Gamma^{\ast}(x \otimes y) \Delta_\Gamma ) = \sum_{g\in \Gamma} \langle x  g,  g\rangle \langle  y  g,  g\rangle \\
			\leq &
			\left( \sum_{g\in \Gamma} \langle x  g,  g\rangle^p \right)^{\frac{1}{p}} \left( \sum_{g\in \Gamma} \langle y  g,  g\rangle^q \right)^{\frac{1}{q}}
			= \Vert x \Vert_{p}\Vert y \Vert_{q}.
		\end{split}
		\]
		As every element in $\calS_p$ and $\calS_q$ can be written as a linear combination of 4 positive elements with smaller or equal norm the lemma follows for $r =1$.	Now, take $r = \infty$. Then also $p = q = \infty$ and for $(x,y)\in \calS_{p}\times \calS_q$ we see that $\Delta_{\Gamma}^*(x\otimes y)\Delta_{\Gamma}\in \calS_{r}$. Furthermore, we have the norm estimate
		$$\|\Delta_{\Gamma}^*(x\otimes y)\Delta_{\Gamma}\| \leq \|\Delta_{\Gamma}^*\|\cdot\|x\otimes y\|\cdot \|\Delta_{\Gamma}\| \leq \|x\|\cdot\|y\|.$$ 
		The lemma then follows from bilinear complex interpolation \cite[Theorem 4.4.1]{berghInterpolationSpacesIntroduction2012}
	\end{proof}

	\begin{lemma}\label{Lem:Younginequality}
		Let $1 \leq p,q,r \leq \infty$ with $\frac{1}{r} = \frac{1}{p} + \frac{1}{q}$.
		Let $H_1$ and $H_2$ be $\mathbb{C}[\Gamma]$ bimodules and let $\xi \in H_1$ and $\eta \in H_2$. Suppose that the coefficient $T_{\xi}$ is in $\calS_p$ and the the coefficient $T_{\eta}$ is in $\calS_q$.    Then the coefficient $T_{\xi\otimes \eta}$ of $H_1 \otimes_{\Gamma} H_2$ is in $\calS_r$.
	\end{lemma}
	\begin{proof}
		We have for $s,t \in \Gamma,$
		\begin{align*}
			\tau (T_{\xi\otimes \eta}( s )t) &= \langle s \xi t \otimes s \eta t, \xi \otimes \eta\rangle \\
			&= \langle s \xi t, \xi\rangle \langle s \eta t, \eta\rangle \\
			&= \tau(T_{\xi}(s ) t)  \tau(T_{\eta}( s ) t).
		\end{align*}
		It follows that $T_{\xi\otimes \eta} = \Delta_\Gamma^{\ast} (T_{\xi} \otimes T_{\eta}) \Delta_\Gamma$. We conclude the proof by Lemma \ref{Lem=Fix}.
	\end{proof}

\begin{proposition}\label{Prop:Spquasicontain}
Let $H$ be a $\mathbb{C}[\Gamma]$  bimodule  such that for a dense subset of $H$ the coefficients are in $\calS_{p}$. Then the bimodule $H^{\otimes n}_{\Gamma}$ is quasi-contained in the coarse bimodule for any $n\geqslant \frac{p}{2}$.
\end{proposition}
\begin{proof}
By Lemma \ref{Lem:Younginequality} (and   induction) we get that a dense subset of coefficients of $H^{\otimes n}_{\Gamma}$ is in $\calS_{\frac{p}{n}} \subset \calS_2$, so by Proposition \ref{proposition:quasicontainment}  we get the quasi-containment.
\end{proof}

\begin{definition}
Let $H$ and $K$ be $\mathbb{C}[\Gamma]$ bimodules. A linear map $V: H \rightarrow K$ is called {\it almost bimodular} if for every $x, y \in \mathbb{C}[\Gamma]$ the map
\[
H \rightarrow K: \xi \mapsto  x V(\xi) y - V(x \xi y),
\]
 is compact.
\end{definition}

\begin{lemma}\label{Lem:almostbimod}
Let $H_1$ and $H_2$ be bimodules over $\mathbb{C}[\Gamma]$. Suppose $V_1\colon \ell_2(\Gamma)\to H_1$ and $V_2\colon \ell_2(\Gamma)\to H_2$ are almost bimodular bounded linear maps. Then
\[
V_1\ast V_2:= (V_1\otimes V_2) \circ \Delta_{\Gamma} \colon \ell^2(\Gamma) \to H_1 \otimes_{\Gamma} H_2
\]
 is almost bimodular.
\end{lemma}
\begin{proof}
It suffices the check the almost bimodularity for $x= s $ and $y= t$, as the general case will follow by taking linear combinations. For a map $V: \ell_2(\Gamma) \rightarrow H$ with $H$ a $\mathbb{C}[\Gamma]$ bimodule we will write $(s V t)(\xi) = s V(\xi) t$ and  $V^{s, t}(\xi):= V(s \xi t)$ where $\xi \in \ell_2(\Gamma)$.
It follows from the definitions that
\[
s (V_1\ast V_2)  t = (s \otimes s) \cdot  ((V_1 \otimes V_2) \circ \Delta_\Gamma ) \cdot (t \otimes t) =   (s V_1 t \ast s V_2 t).
\]
Further, for $\xi \in \ell_2(\Gamma)$,
\[
\begin{split}
(V_1 \ast V_2)^{s, t}(\xi) = &    (V_1 \otimes  V_2 )   \Delta_\Gamma ( s \xi t)\\
  = & (V_1 \otimes  V_2 )  ( ( s \otimes s) \Delta_\Gamma(\xi) ( t \otimes t) )
   = (V_1^{s, t} \ast V_2^{s, t} )(\xi).
\end{split}
\]
Hence
\[
(V_1 \ast V_2)^{ s , t} = V_1^{ s, t} \ast V_2^{ s , t }.
\]
Therefore we have
\begin{equation}\label{Eq:almostbimod}
s (V_1 \ast V_2) t - (V_1\ast V_2)^{s, t} = (( s  V_1 t - V_1^{ s , t })\ast s V_2 t) + (V_1^{ s , t}\ast (s V_2 t - V_2^{ s, t })).
\end{equation}
By our assumption the operators $s V_1 t - V_1^{ s, t}$ and $s V_2 t - V_2^{ s , t }$ are compact. So it suffices to check that if $K$ is compact and $T$ is bounded then both $K\ast T$ and $T\ast K$ are compact. To check that, for every finite subset $F \subset \Gamma$ consider the corresponding finite rank orthogonal projection $P_F$ onto the linear span of $\delta_s \in \ell_2(\Gamma), s \in F$. We can easily check that $\Delta \circ P_F = (P_F \otimes \oper{Id}) \circ \Delta = (\oper{Id} \otimes P_F) \circ \Delta$. It follows that $(K \ast T)P_F = (KP_F \ast T)$, so $(K\ast T) P_F - K \ast T = (KP_F - K) \ast T$. Further,
\[
\Vert  (K\ast T) P_F - K \ast T \Vert \leq \Vert KP_F - K \Vert \Vert T \Vert.
\]
By compactness of $K$ we see that $\Vert KP_F - K \Vert$  goes to 0 as $F$ increases.  So $K\ast T$ can be approximated in norm by finite rank operators and thus is compact.
The proof for $T\ast K$ is the same. Hence the operator in  \eqref{Eq:almostbimod} is compact, i.e. $V_1\ast V_2$ is almost bimodular.
\end{proof}

\begin{lemma}\label{Lem=PartialIsoConvolution}
For $i=1,2$ suppose that $V_i: \ell_2(\Gamma) \rightarrow H_i$ is a partial isometry to a $\mathbb{C}[\Gamma]$ bimodule $H_i$ such that $\ker(V_i)$ is spanned linearly by a subset $F_i \subseteq \Gamma$. Then $V_1 \ast V_2$ is a partial isometry whose kernel is the linear span of $F_1 \cup F_2$.

\begin{proof}
	The comultiplication $\Delta_\Gamma$ is an isometry $\ell_2(\Gamma) \rightarrow \ell_2(\Gamma) \otimes \ell_2(\Gamma)$.  Clearly $\Delta_\Gamma(s) =  s \otimes s$ is contained in  $\ker(V_1 \otimes V_2)$ if $s$ is in $F_1 \cup F_2$. Further, $V_1 \otimes V_2$ is isometric on $\ker(V_1)^\perp \otimes \ker(V_2)^\perp$ and so it is certainly isometric on the linear span of $\Delta_\Gamma( s ) =  s \otimes s, s \in \Gamma \backslash (F_1 \cup F_2)$. These observations conclude the lemma.
\end{proof}

\end{lemma}

 \subsection{The Akemann-Ostrand property AO$^+$ and strong solidity}
This section serves as a blackbox that connects the theory that we develop in this paper to a central concept in deformation-rigidity theory: strong solidity.  Firstly we recall a version of the Akemann-Ostrand property that was introduced in \cite{isonoExamplesFactorsWhich2015}.

\begin{definition}\label{Dfn=AO}
A finite von Neumann algebra $M$ has property AO$^+$ if there exists a $\sigma$-weakly dense unital C$^\ast$-subalgebra $A \subseteq M$ such that:
\begin{enumerate}
\item $A$ is locally reflexive \cite[Section 9]{brownAlgebrasFiniteDimensionalApproximations2008};
\item There exists a unital completely positive map $\theta: A \otimes_{{\rm min}} A^{{\rm op}} \rightarrow B(L_2(M))$ such that $\theta( a \otimes b^{{\rm op} }) - a b^{{\rm op}}$ is compact for all $a,b \in A$.
\end{enumerate}
\end{definition}

The following theorem will be the main tool to prove that certain von Neumann algebras have AO$^+$ using the Riesz transforms in this paper.

\begin{theorem}[Proposition 5.2 of \cite{caspersL2CohomologyDerivationsQuantum2021}]\label{Thm=RieszImpliesAO}
Let $H$ be a $\mathbb{C}[\Gamma]$ bimodule and let $V: \ell_2(\Gamma) \rightarrow H$ be a bounded linear map. Assume that $H$ is quasi-contained in the coarse bimodule of $\Gamma$,  that $V$ is almost bimodular and that $V^\ast V$ is Fredholm. Assume that $C_r^\ast(\Gamma)$ is locally reflexive.  Then $\calL(\Gamma)$ satisfies AO$^+$.
\end{theorem}

The following theorem in turn yields the strong solidity results from  AO$^+$. For the notion of weak amenability we refer to \cite[Section 12.3]{brownAlgebrasFiniteDimensionalApproximations2008}. If $\Gamma$ is a weakly amenable discrete group then $C_r^\ast(\Gamma)$ is automatically locally reflexive. All Coxeter groups are weakly amenable \cite{fendlerWeakAmenabilityCoxeter2002}, \cite{Janus} as well as simple Lie groups of real rank one \cite{CanniereHaagerup}, \cite{CowlingHaagerup}. We recall that amenability of a von Neumann algebra was defined in the introduction and preliminaries. We note that amenability and weak amenability shall not appear explicitly in the proofs of this paper. We recall that a von Neumann algebra is called {\it diffuse} if it does not contain minimal projections.

\begin{definition}\label{Dfn=StronglySolid}
A finite von Neumann algebra $M$ is called {\it strongly solid} if for every diffuse amenable von Neumann subalgebra $B \subseteq M$ we have that the normalizer
\[
\{ u \in M : u \textrm{ unitary and }  u B u^\ast = B \},
\]
generates a von Neumann algebra that is amenable again.
\end{definition}

\begin{theorem}[See \cite{PopaVaesCrelle} and \cite{isonoExamplesFactorsWhich2015}]\label{Thm=AOimpliesStronglySolid}
Let $\Gamma$ be a discrete weakly amenable group such that $\calL(\Gamma)$ satisfies AO$^+$. Then $\calL(\Gamma)$ is {\it strongly solid}.
\end{theorem}

\subsection{Application A: Proper cocycles into $p$-integrable representations} \label{Sect=PInt}  We are now able to harvest our first result.  We use a type of ad hoc Riesz transform which is slightly different from what we do in Section \ref{Sect=RieszEtc} but with similar fundamental properties.  We use tensoring of bimodules to establish the Akemann-Ostrand property. The method is exemplary for the rest of the paper.

\begin{definition}
Let $\Gamma$ be a discrete group. Suppose that $\pi\colon \Gamma \to B(H)$ is a unitary (or orthogonal) representation. We say that $\pi$ is $p$-integrable for some $p<\infty$ if there exists a dense subspace $H_0$ such that for any $v \in H_0$ the matrix coefficient $g\mapsto \langle \pi(g) v,  v\rangle$ is in $\ell_p(\Gamma)$.
\end{definition}

The following theorem is the main result of \cite{Sinclair}. The idea of the proof parallels \cite{Sinclair} but is somewhat cleaner and more conceptual we believe. 

We recall that a {\it derivation} $\partial: \mathbb{C}[\Gamma] \rightarrow K$ into a $\mathbb{C}[\Gamma]$ bimodule $K$ is a linear map that satisfies the Leibniz rule
\[
\partial(xy) = x \partial(y) + \partial(x) y, x,y \in \mathbb{C}[\Gamma].
\]

 A function $b: \Gamma \rightarrow H$ with $\pi: \Gamma \rightarrow B(H)$  a unitary representation on a Hilbert space $H$ is called a {\it cocycle} or {\it $1$-cocycle} if $b( st) = \pi(s) b(t) + b(s)$ for all $s,t \in \Gamma$. The following theorem is the only place where cocycles are used in this paper.

\begin{theorem}\label{Thm=Pint}
Let $\Gamma$ be a discrete group admitting a proper cocycle into a $p$-integrable representation for some $p<\infty$. Assume $C_r^\ast(\Gamma)$ is locally reflexive.  Then $\calL(\Gamma)$ satisfies $AO^{+}$.
\end{theorem}
\begin{proof}
Let $\pi\colon \Gamma \to B(H)$ be a $p$-integrable representation, $H_0\subset H$ a dense subspace with $p$-integrable coefficients, and let $c\colon \Gamma \to H$ be a proper $\pi$-cocycle.
 The map
 \[
 \lambda_g \mapsto \exp(-t\|c(g)\|^2) \lambda_g,  \qquad g \in \Gamma, t \geq 0,
  \]
  extend to a semi-group of normal unital completely positive maps on $\calL(\Gamma)$ \cite[Theorem C.11]{brownAlgebrasFiniteDimensionalApproximations2008}, i.e. a quantum Markov semi-group as will be defined in Section \ref{Sect=RieszEtc}.
 Set $\Delta: \mathbb{C}[\Gamma] \rightarrow \mathbb{C}[\Gamma]$ by  $\Delta(\lambda_g):= \|c(g)\|^2 \lambda_g, g \in \Gamma$.
 Set $K:=H \otimes \ell_2(\Gamma)$ with the left action given by $\pi\otimes \lambda$ and the right action given by ${\rm id} \otimes \rho$. Now define $\partial \colon \mathbb{C}[\Gamma] \to K$ by $\partial(\lambda(g)):= c(g)\otimes \delta_g$. As $c$ is a cocycle, $\partial$ is a derivation.
  We will check that the coefficients of this bimodule $K$ at vectors of the form $\xi:=v\otimes \delta_g$ are in $\cS_p$, where $v\in H_0, g \in \Gamma$. Let us start with the case $\xi=v\otimes \delta_e$. We claim that $T_{\xi}(\lambda_g)= \langle \pi(g)v, v\rangle \lambda_g$. Indeed, let us take $g_1, g_2 \in \Gamma$ and compute
\[
\langle g_1 \xi g_2, \xi\rangle = \langle \pi(g)v \otimes \delta_{g_1g_2}, v \otimes \delta_e\rangle = \langle \pi(g)v,v\rangle \delta_{g_1g_2=e}.
\]
It is clearly equal to $\tau(T_{\xi}(g_1) g_2)$ from the definition of the coefficient $T_{\xi}$. Therefore the coefficients are diagonal operators with $\ell_p$-coefficients, so they are in $\cS_p$. To handle vectors of the form $v\otimes \delta_g$ note that $T_{\xi\cdot g}= \lambda_g^{\ast} T_{\xi} \lambda_g$ and $(v\otimes \delta_e)\cdot g = v\otimes \delta_g$, so these coefficients are also in $\cS_p$. The Riesz transform defined as
\[
\partial \circ \Delta^{-\frac{1}{2}}: \mathbb{C}[\Gamma] \rightarrow K: \lambda_g \mapsto \|c(g)\|^{-1} c(g) \otimes \delta_g, \qquad g \in \Gamma.
\]
 extends to an almost bimodular isometry $\ell_2(\Gamma) \rightarrow K$  (see \cite[Proposition 5.3]{caspersL2CohomologyDerivationsQuantum2021}) whose kernel by definition equals the kernel of $\Delta$. This kernel is finite dimensional as $c$ is proper.

Now since $K$ has a dense set of coefficients in $\cS_p$ we have that the bimodule $K^{\otimes n}_\Gamma$ has a dense set of coefficients in $\cS_2$ for every $n \geq \frac{p}{2}$. Therefore, by Proposition \ref{proposition:quasicontainment} $K^{\otimes n}_\Gamma$ is quasi-contained in the coarse bimodule of $\Gamma$. From the previous paragraph  and Lemma \ref{Lem:almostbimod} we see that $(\partial \circ \Delta^{-\frac{1}{2}} )^{ \ast n }$ is an almost bimodular map. $(\partial \circ \Delta^{-\frac{1}{2}})^{\ast n }$ is moreover a partial isometry with a finite dimensional cokernel by Lemma \ref{Lem=PartialIsoConvolution}. Therefore the assumptions of Theorem \ref{Thm=RieszImpliesAO}  are satisfied and if $C_r^\ast(\Gamma)$ is locally reflexive we conclude that $\calL(\Gamma)$ has AO$^{+}$.
\end{proof}

\begin{remark}
In the proof of Theorem \ref{Thm=Pint} we may view $\partial$ as an unbounded densely defined operator $\ell_2(\Gamma) \rightarrow K$. It is not difficult to check that $(\partial^\ast \partial)(\lambda_g) = \Delta(\lambda_g), g \in \Gamma$ as for the derivations that occur in Section \ref{Sect=RieszEtc}.
\end{remark}

Recall that $\Gamma$ is called icc if all conjugacy classes except for the identity are infinite.

\begin{corollary}
Let $\Gamma$ be an icc lattice in either $SO(n,1), n \geq 3$ or $SU(m,1), m \geq 2$. Then $\calL(\Gamma)$ is strongly solid.
\end{corollary}
\begin{proof}
By \cite{CanniereHaagerup}, \cite{CowlingHaagerup} we have that $\Gamma$ is weakly amenable and in particular $C_r^\ast(\Gamma)$ is locally reflexive. By \cite[Theorem 1.9]{Shalom} $\Gamma$ admits a proper cocycle in a $p$-integrable representation for some $p \in [2, \infty)$. Therefore by Theorem \ref{Thm=Pint} $\calL(\Gamma)$ satisfies AO$^+$. We then conclude the proof by Theorem \ref{Thm=AOimpliesStronglySolid}.
\end{proof}

\section{Quantum Markov semi-groups, gradients and the Riesz transforms}\label{Sect=RieszEtc}
In this section we study quantum Markov semi-groups of Fourier multipliers on the group von Neumann algebra of a discrete group. We introduce the associated Riesz transform which takes values in a certain bimodule that we call the `gradient bimodule' or the bimodule associated with the `carr\'e du champ'. Our main goal is to analyze when the coefficients of this bimodule are in the Schatten $\cS_p$ space and consequently when this bimodule is quasi-contained in the coarse bimodule. We also show that under very natural conditions the Riesz transform is an almost bimodular map in the sense of Section \ref{Sect=Coefficients}.

\subsection{Quantum Markov semi-groups, the gradient bimodule and the Riesz transform} \label{Sect=Riesz}

We start defining the Riesz transform of a quantum Markov semi-group.

\begin{definition}
A {\it  quantum Markov semi-group} (QMS) on a finite von Neumann algebra $(M, \tau)$  is a semi-group $\Phi = (\Phi_t)_{t \geq 0}$ of  normal unital completely positive maps $\Phi_t: M \rightarrow M$ that are  trace preserving ($\tau \circ \Phi_t = \tau, t \geq 0$) and such that for every $x \in M$ the map $t \mapsto \Phi_t(x)$ is strongly continuous. We shall moreover assume that a quantum Markov semi-group is {\it symmetric} meaning that for every $x,y \in M$ and $t \geq 0$ we have $\tau(\Phi_t(x) y ) = \tau(x \Phi_t(y))$. So QMS always means symmetric QMS.
\end{definition}

  Fix a  QMS $\Phi = (\Phi_t)_{t \geq 0}$ on  a finite von Neumann algebra $M$ with a normal faithful tracial state $\tau$. By the Kadison-Schwarz inequality there exists a  semi-group of contractions $(\Phi_t^{(2)})_{t \geq 0}$ on $L_2(M) = L_2(M, \tau)$ by
\[
\Phi_t^{(2)}(  x \Omega_\tau) = \Phi_t(x) \Omega_\tau, \qquad x \in M.
\]
Here $\Omega_\tau = 1_M$ is the cyclic vector in $L_2(M)$.
The semi-group  $(\Phi_t^{(2)})_{t \geq 0}$ is moreover point-norm continuous, i.e. it is continuous for the strong topology on $B(L_2(M))$. By a special case of  the Hille-Yosida theorem there exists an unbounded positive self-adjoint operator $\Delta$ on $L_2(M)$ such that $\Phi_t^{(2)} =  \exp(-t \Delta)$. We will assume the existence of a $\sigma$-weakly dense $\ast$-subalgebra $\calA \subseteq M$ such that $\calA \Omega_\tau \subseteq \Dom(\Delta)$ and $\Delta( \calA \Omega_\tau  ) \subseteq \calA \Omega_\tau$. By identifying $a \in \calA$ with $a \Omega_\tau \in L_2(M)$ we may and will view $\Delta$ as a map $\calA \rightarrow \calA$. We now introduce the {\it carr\'e du champ} or {\it gradient}  as
\[
\Gamma: \calA \times \calA \rightarrow \calA: (a,b) \mapsto \frac{1}{2} (  \Delta(b^\ast) a +  b^\ast \Delta(a) - \Delta( b^\ast a )    ).
\]
Let $H$ be any $\calA$ bimodule, i.e. we recall $H$ is a Hilbert space with commuting left and right actions of $\calA$. For $a,b \in \calA, \xi, \eta \in H$ we set the possibly degenerate inner product on $\calA \otimes H$ (vector space tensor product) by
\[
\langle a \otimes \xi, b \otimes \eta \rangle = \langle \Gamma(a,b) \xi, \eta \rangle.
\]
The Hilbert space obtained by quotienting out the degenerate part of this inner product and taking the completion shall be denoted by $H_\nabla$. We denote by $a \otimes_\nabla \xi$ the element $a \otimes \xi$ identified in $H_\nabla$. For $x, y, a \in \calA$ and $\xi \in H$  we define commuting left and right actions by
\begin{equation}\label{Eqn=Actions}
x \cdot (a \otimes_\nabla \xi) = xa \otimes_\nabla \xi - x \otimes_\nabla  a\xi, \qquad (a \otimes_\nabla \xi) \cdot y = a \otimes_\nabla \xi y.
\end{equation}
In this paper we shall only deal with the case $H= L_2(M)$ with actions by left and right multiplication of $M$. In this case the actions \eqref{Eqn=Actions} extend to contractive actions on the norm closure of $\calA$. We do not say anything about whether the actions are normal at this point, but rather use Proposition \ref{proposition:quasicontainment} to show that they are normal in the cases that are relevant.
 We define a derivation
\[
\nabla: \calA \rightarrow L_2( M )_\nabla: a \mapsto a \otimes_\nabla \Omega_\tau.
\]
More precisely, $\nabla$ satisfies the Leibniz rule
\[
\nabla(xy) = x \nabla(y) + \nabla(x) y, \qquad  x,y \in \calA
\]
 with respect to the module actions \eqref{Eqn=Actions}. This fact uses that $\tau$ is tracial. Since $\Phi_t$ is $\tau$-preserving it follows that for $x \in \calA$ we have $\langle \Delta(x) \Omega_\tau, \Omega_\tau \rangle = \frac{d}{dt}\vert_{t = 0} \langle \Phi_t(x) \Omega_\tau, \Omega_\tau \rangle = 0$ (upper derivative).
Therefore, as $\Delta \geq 0$,
\[
\begin{split}
\Vert \nabla(a) \Vert^2 = & \langle \Gamma(a, a) \Omega_\tau, \Omega_\tau \rangle \\
= & \frac{1}{2} (\langle \Delta(a)  \Omega_\tau,  a \Omega_\tau \rangle +  \langle a \Omega_\tau, \Delta(a) \Omega_\tau \rangle  - \langle \Delta(a^\ast a) \Omega_\tau, \Omega_\tau \rangle ) \\
= & \frac{1}{2} (\langle \Delta^{\frac{1}{2}}(a)  \Omega_\tau,  \Delta^{\frac{1}{2}}( a) \Omega_\tau \rangle +  \langle \Delta^{\frac{1}{2}}(a) \Omega_\tau, \Delta^{\frac{1}{2}}(a) \Omega_\tau \rangle  - 0 ) \\
= & \Vert  \Delta^{ \frac{1}{2} } (a)  \Omega_\tau \Vert^2.
\end{split}
\]
It follows that we have an isometric map
\[
 \nabla \Delta^{-\frac{1}{2}}: \ker(\Delta)^\perp \rightarrow L_2(M)_\nabla.
\]
We extend this map  to a partial isometry
\[
R_\Phi: L_2(M) \rightarrow L_2(M)_\nabla
\]
 by defining it to have $\ker(\Delta)$ as its kernel. We call $R_\Phi$ the {\it Riesz transform}.

 \begin{remark}
 This Riesz transform was also used in \cite[Section 5]{caspersL2CohomologyDerivationsQuantum2021}. Note that mapping that was introduced in \cite[Section 5, Eqn. (5.1)]{caspersL2CohomologyDerivationsQuantum2021} differs from $R_\Phi$ only on $\ker(\Delta)$. If the kernel of $\Delta$ is finite dimensional then $R_\Phi$ agrees with  \cite[Eqn. (5.1)]{caspersL2CohomologyDerivationsQuantum2021} up to a finite rank perturbation. In particular this is the case if $\Delta \geq 0$ has a compact resolvent. The results of \cite[Section 5]{caspersL2CohomologyDerivationsQuantum2021} stay intact under this finite rank perturbation.
\end{remark}


\subsection{Coefficients of the gradient bimodule}
 We now start our analysis of coefficients of the gradient bimodule. The following definition of `gradient-$\cS_p$' that first occurred in \cite{caspersGradientFormsStrong2021}  (for $p=2$) and \cite{caspersL2CohomologyDerivationsQuantum2021} (for general $p$)  plays a central role in this paper. The definition may depend on the choice of the $\sigma$-weakly dense subalgebra $\calA$ of $M$ which we fixed before in our notation. This paper contains the first results for the gradient-$\cS_p$ property in the context of group algebras.

\begin{definition} \label{Dfn=GradientSp}
Let $p \in [1, \infty]$. Consider a QMS $\Phi$
 on a finite von Neumann algebra $(M, \tau)$
with generator $\Delta$ and a dense $\ast$-subalgebra $\calA \subseteq M$ as in Section \ref{Sect=Riesz}. $\Phi$ is called gradient-$\cS_p$ if for every $a,b \in \calA$ the map
\[
\Psi^{a,b}: \calA \rightarrow \calA: x \mapsto    \Delta(axb) +  a \Delta(x) b -  \Delta(ax) b -  a \Delta(xb ),
\]
extends as $x \Omega_\tau \mapsto \Psi^{a,b}(x) \Omega_\tau$ to a bounded map on $L_2(M)$ that is moreover in the Schatten $p$-class $\cS_p = \cS_p(L_2(M))$.
\end{definition}

\begin{remark}
Since $\Delta$ is self-adjoint we have for $a,b,x,y \in \calA$,
\[
\begin{split}
  \langle \Psi^{a,b}(x) \Omega_\tau, y \Omega_\tau \rangle =  &   \langle (\Delta(axb) + a \Delta(x) b - \Delta(ax) b - a \Delta(xb ) ) \Omega_\tau, y \Omega_\tau \rangle \\
= &    \langle x \Omega_\tau, (\Delta(a^\ast y b^\ast) + a^\ast \Delta(y) b^\ast - \Delta(a^\ast x) b^\ast - a^\ast \Delta(y b^\ast ) )  \Omega_\tau \rangle \\
= &  \langle x \Omega_\tau, \Psi^{a^\ast,b^\ast}(y) \Omega_\tau \rangle.
\end{split}
\]
So it follows that
\begin{equation}\label{calculation:adjoint}
  (\Psi^{a,b})^\ast =  \Psi^{a^\ast,b^\ast}, \qquad a,b \in \calA.
\end{equation}
\end{remark}

The following lemma simplifies verifying whether a QMS has the gradient-$\cS_p$ property.

	\begin{lemma}[Condition that implies Gradient-$\calS_p$ property]\label{lemma:reduction-gradient-Sp-to-generators}
  Let $p\in [1,\infty]$. Let $\mathcal{A}_0\subseteq \calA$  be a self-adjoint subset that generates $\mathcal{A}$ as a $\ast$-algebra.
  Then $(\Phi_t)_{t\geq 0}$ is gradient-$\mathcal{S}_p$ if and only if for all  $a, b \in \mathcal{A}_0$ we have that $\Psi^{a,b}$ is in $\mathcal{S}_p$.
 	\end{lemma}
		\begin{proof}
			The only if statement follows directly from the definition of gradient-$\mathcal{S}_p$. We will prove the other direction. We must prove that $\Psi^{a,b}$ is in $\cS_p$ for every $a,b \in \calA$. Since $\calA_0$ is self-adjoint $\calA$ is generated by $\calA_0$ as an algebra. So $\calA$ is spanned linearly by $(\calA_0)^n, n \in \mathbb{N}$. Note that the map $\Psi^{a,b}$ depends linearly on both $a$ and $b$. So in order to prove that $\Psi^{a,b}$ is in $\cS_p$ for all $a,b \in \calA$ it suffices to prove that $\Psi^{a,b}$ is in $\cS_p$ for all $a,b \in (\calA_0)^n$ for every $n \in \mathbb{N}_{\geq 1}$. We shall prove this latter statement by induction on $n$. The case $n = 1$  holds by assumption of the lemma. We now assume that we have proved the statement for $n$ and shall prove it for $n+1$.

First note that for   $u_1,u_2,v,w\in \mathcal{A}$ we have
			\begin{equation}
\begin{split}
			\Psi^{u_1u_2,w}(v) &= \Delta(u_1u_2vw) + u_1u_2\Delta(v)w - \Delta(u_1u_2v)w - u_1u_2\Delta(vw)\\
			&=\left(\Delta(u_1u_2vw) + u_1\Delta(u_2v)w - \Delta(u_1u_2v)w - u_1\Delta(u_2vw)\right)\\
			&+	u_1\left(\Delta(u_2vw) + u_2\Delta(v)w - \Delta(u_2v)w - u_2\Delta(vw)\right)\\
			&=\Psi^{u_1,w}(u_2v)+u_1\Psi^{u_2,w}(v),
\end{split}			
\end{equation}
			and likewise for $u,v,w_1,w_2\in \mathcal{A}$ we have
			\begin{equation}
			\Psi^{u,w_2w_1}(v) =\Psi^{u,w_1}(vw_2)+ \Psi^{u,w_2}(v)w_1.
			\end{equation}
			Combining these expressions we see that for $u = u_1 u_2$ and $w = w_2 w_1$ we have
			\begin{equation}\label{eq:Sp-property-generators}
\begin{split}
			\Psi^{u,w}(v) &= \Psi^{u_1u_2,w}(v)\\
			&= \Psi^{u_1,w}(u_2v) + u_1\Psi^{u_2,w}(v)\\
			&= \Psi^{u_1,w_2w_1}(u_2v) + u_1\Psi^{u_2,w_2w_1}(v)\\
			&= \left(\Psi^{u_1,w_1}(u_2vw_2) +\Psi^{u_1,w_2}(u_2v)w_1\right) +
			u_1\left(\Psi^{u_2,w_1}(vw_2) + \Psi^{u_2,w_2}(v)w_1\right).
\end{split}			
\end{equation}
			By the induction hypothesis we have that
			$\Psi^{u_1,w_1},\Psi^{u_1,w_2},\Psi^{u_2,w_1},\Psi^{u_2,w_2}$ are all in $\mathcal{S}_p$. Since the $\mathcal{S}_p$ class forms an ideal in $B(L_2(M,\tau))$  we have that the four operators in \eqref{eq:Sp-property-generators} are all in $\mathcal{S}_p$.  This finishes the induction and thus shows that the associated semi-group is gradient-$\mathcal{S}_p$.
		\end{proof}

\subsection{Almost bimodularity of the Riesz transform}
Next we analyze when the Riesz transform is almost bimodular. Therefore we introduce the following notions. We say that a QMS $\Phi$ on a finite von Neumann algebra is {\it filtered} if the generator $\Delta$ has a compact resolvent and for every eigenvalue $\lambda$ of $\Delta$ there exists a (necessarily finite dimensional) subspace $\calA(\lambda) \subseteq \calA$ such that $\calA(\lambda) \Omega_\tau$ equals the eigenspace of $\Delta$ at eigenvalue $\lambda$. Moreover, we assume that for an increasing enumeration $(\lambda_n)_{n \geq 0}$ of the eigenvalues of $\Delta$ we have for all $k,l \geq 0$ that
\[
\calA = \bigoplus_{n=0}^\infty \calA(\lambda_n), \qquad \calA(\lambda_l) \calA(\lambda_k) \subseteq \bigoplus_{n=0}^{l+k} \calA(\lambda_n).
\]
We will further say that $\Delta$  has {\it subexponential growth} if
\[
\lim_{ k \rightarrow \infty } \frac{ \lambda_{k+1}  }{  \lambda_k  } = 1.
\]
\begin{remark}
In \cite{caspersRieszTransformCompact2021} a more general notion of filtering and subexponential growth was considered for central Fourier multipliers on compact quantum groups. The current `linear'  type of definition suffices however for our purposes.
\end{remark}

 \begin{theorem}[Theorem 5.12 of \cite{caspersL2CohomologyDerivationsQuantum2021}]\label{Thm=AlmostBimodularRiesz}
 Suppose that a QMS $\Phi$ on a finite von Neumann algebra $M$ is filtered with subexponential growth. Then the Riesz transform $R_\Phi: L_2(M) \rightarrow L_2(M)_\nabla$ is almost bimodular.
 \end{theorem}

\subsection{Semi-groups of Fourier multipliers on group von Neumann algebras}
Now consider the case that $M$ is a group von Neumann algebra $\calL(\Gamma)$ of a discrete group $\Gamma$ and $\calA = \mathbb{C}[\Gamma]$. The following theorem is a version of  Sch\"onberg's theorem.

\begin{theorem}[See Appendix C of \cite{bekkaKazhdanProperty2008}]\label{Thm=Schoenberg}
Let $\psi: \Gamma \rightarrow \mathbb{R}$. The following are equivalent:
\begin{enumerate}
\item   $\psi$ is  conditionally of negative type.
\item   There exists a (recall: symmetric) QMS $\Phi = (\Phi_t)_{t \geq 0}$ on $M$ determined by
\[
\Phi_t(  \lambda_\gamma ) = \exp(-t \psi(\gamma))   \lambda_\gamma, \qquad \gamma \in \Gamma.
\]
\end{enumerate}
\end{theorem}

We will call a QMS $\Phi$ as in Theorem \ref{Thm=Schoenberg} a {\it QMS of Fourier multipliers} or a QMS associated with  $\psi: \Gamma \rightarrow \mathbb{R}$. Note that we assume such QMS's to be symmetric.  We view the generator of this semi-group as a map on   $\mathbb{C}[\Gamma]$ which is given by
\[
\Delta_\psi: \mathbb{C}[\Gamma] \rightarrow \mathbb{C}[\Gamma]:  \lambda_\gamma \mapsto \psi(\gamma) \lambda_\gamma.
\]
The following Theorem \ref{Thm=GradientCoefficients} connects Definition \ref{Dfn=GradientSp} to Section \ref{Sect=Coefficients}.

\begin{theorem}\label{Thm=GradientCoefficients}
Consider a QMS $\Phi = (\Phi_t)_{t \geq 0}$ of Fourier multipliers on $\calL(\Gamma)$. Let
\[
H_{00} = \{ a \otimes_\nabla c   \in \ell_2(\Gamma)_\nabla  : a, c \in \mathbb{C}[\Gamma] \} \subseteq \ell_2(\Gamma)_\nabla.
\]
 If $\Phi$ is gradient-$\cS_p$ with $p \in [1, \infty]$ then for every $\xi, \eta \in  {\rm span} \mathbb{C}[\Gamma] H_{00} \mathbb{C}[\Gamma]$ the coefficient $T_{\xi,\eta}$ is in $\cS_p$.
\end{theorem}
\begin{proof}
Let $a,b,c,d,x,y \in \CC[\Gamma]$ and let $\xi = a \otimes_\nabla c, \eta = b \otimes_\nabla d$ be elements of $H_{00}$.  We have
\[
\begin{split}
&   2  \langle x \cdot (a \otimes_\nabla c ) \cdot y, b \otimes_\nabla d   \rangle =
 2 \langle xa \otimes_\nabla c y   - x \otimes_\nabla a c y   , b \otimes_\nabla d   \rangle
 = 2 \langle \Gamma(xa, b  ) cy   - \Gamma( x, b ) a c y, d   \rangle_\tau \\
= &   \langle   (   b^\ast \Delta(x a)  + \Delta(b^\ast) x a   -  \Delta(b^\ast xa) - b^\ast \Delta(x) a -  \Delta(b^\ast) x a  + \Delta(b^\ast x) a    )  c y  , d   \rangle_\tau \\
 = &    \langle   (    \Delta(b^\ast x) a + b^\ast \Delta(x a) - \Delta(b^\ast xa) - b^\ast \Delta(x) a  )  c y  , d   \rangle_\tau
=   -   \langle   \Psi^{b^\ast, a}(x)  c y  , d  \rangle_\tau
=   -    \tau(   d^\ast \Psi^{b^\ast, a}(x)  c y ).
\end{split}
\]
We conclude that
\[
-2 T_{\xi,\eta}(x) = d^\ast \Psi^{b^\ast, a}(x)  c.
\]
In particular if  $\Psi^{b^\ast, a}$ is in $\calS_p$ then so is $T_{\xi, \eta}$. The statement now follows from Lemma \ref{lemma:reduction-to-cyclic-subset}.
\end{proof}

Let us now  show that in the case of semi-groups of Fourier multipliers, the case gradient-$\cS_p$ is conceptually much easier to understand.  Consider again a QMS $\Phi = (\Phi_t)_{t \geq 0}$ of Fourier multipliers associated with a function $\psi: \Gamma \rightarrow \mathbb{R}$ that is conditionally of negative type.  Let $\Delta_\psi: \mathbb{C}[\Gamma] \rightarrow  \mathbb{C}[\Gamma]$ be as before.
	For $u,w\in \Gamma$ we define a function $\gamma_{u,w}^{\psi}:\Gamma\to \RR$ as
\begin{equation}\label{Eqn=GammaFunction}
\gamma_{u,w}^{\psi }(v)= \psi(uvw)+\psi(v)-\psi(uv)-\psi(vw).
\end{equation}
  We have that the function $\gamma_{u,w}^{\psi}$ is related to the operator $\Psi^{\lambda_u,\lambda_w}$ associated with $\Delta_\psi$ as follows
\[
\begin{split}
	\Psi_{\Delta_{\psi}}^{\lambda_{u},\lambda_{v}}(\lambda_{v}) &=  \Delta_{\psi}(\lambda_{uvw}) + \lambda_{u}\Delta_{\psi}(\lambda_{v})\lambda_{w} - \Delta_{\psi}(\lambda_{uv})\lambda_{w} - \lambda_{u}\Delta_{\psi}(\lambda_{vw})
	 = \gamma_{u,w}^{\psi}(v)\lambda_{uvw}.
\end{split}
\]
	Now as by \eqref{calculation:adjoint} we have $(\Psi^{\lambda_u,\lambda_w})^*=\Psi^{\lambda_u^*,\lambda_w^*}=\Psi^{\lambda_{u^{-1}},\lambda_{w^{-1}}}$
	we obtain that
\begin{equation}\label{Eqn=FiniteRankish}
\begin{split}
	|\Psi^{\lambda_{u},\lambda_{w}}|^2(\lambda_{v}) &= \Psi^{\lambda_{u^{-1}},\lambda_{w^{-1}}}\Psi^{\lambda_{u},\lambda_{w}}(\lambda_{v})
	 = \gamma_{u^{-1},w^{-1}}^{\psi}(uvw)
	\gamma_{u,w}^{\psi}(v)\lambda_{v}
 = |\gamma_{u,w}^{\psi}(v)|^2\lambda_{v}.
\end{split}	
\end{equation}
	This then means that $|\Psi^{\lambda_u,\lambda_w}|^p(\lambda_v) = |\gamma_{u,w}(v)|^p\lambda_v$ and therefore, as $\{\lambda_{v}\}_{v\in \Gamma}$ forms an orthonormal basis, we have that
	\begin{equation}\label{Eqn=PsiComp}
	\|\Psi^{\lambda_u,\lambda_w}\|_{\mathcal{S}_p} = (\sum_{v\in \Gamma}\langle |\Psi^{\lambda_u,\lambda_w}|^p(\lambda_{v}),\lambda_{v}\rangle )^{\frac{1}{p}} =  \|\gamma_{u,w}^{\psi}\|_{\ell_p(\Gamma)}.
	\end{equation}
		Now for $p\in [1,\infty)$, in order to check whether $\Psi^{\lambda_{u},\lambda_{w}}$ is in $\calS_p$ we thus need to check whether $\gamma_{u,w}^{\psi}\in \ell_p(\Gamma)$. Moreover, for $p=\infty$, the condition that $\Psi^{\lambda_{u},\lambda_{w}}\in \calS_p$ means that $\Psi^{\lambda_{u},\lambda_{w}}$ is a compact operator, which is precisely the case when $\gamma_{u,w}^{\psi}\in c_0(\Gamma)$, i.e. when $\gamma_{u,w}^{\psi}$ vanishes at infinity.\\
		
		The above calculations, together with Lemma \ref{lemma:reduction-gradient-Sp-to-generators}, give us a simple condition to check for $p\in [1,\infty]$ whether the semi-group $(\Phi_t)_{t\geq 0}$ is gradient-$\calS_p$.
	
	\begin{lemma}\label{lemma:reduction-to-set-generating-the-group}
		Let  $p\in [1,\infty)$. Let $\Gamma_0 \subseteq \Gamma$ be a subset that generates a discrete group $\Gamma$  with  $\Gamma_0^{-1} = \Gamma_0$.   Let $\Phi = (\Phi_t)_{t\geq 0}$ be a  QMS   associated with a proper function $\psi: \Gamma \rightarrow \mathbb{R}$ that is conditionally of  negative type. If $\gamma_{u,w}^{\psi}\in \ell_p(\Gamma)$ for all $u,w\in \Gamma_0$ then the QMS $\Phi$ is gradient-$\calS_p$. The same holds true for $p=\infty$ when $\ell_p(\Gamma)$ is replaced with $c_0(\Gamma)$.
		\begin{proof}
			We denote $\mathcal{A}_0:=\{\lambda_{g}: g \in \Gamma_0 \}\subseteq \CC[\Gamma]$. Since $\Gamma_0^{-1} = \Gamma_0$ and $\Gamma_0$ generates $\Gamma$ we have that $\calA_0$ is self-adjoint and generates $\CC[\Gamma]$ as an algebra.  Now, if for   $u,w\in \Gamma_0$ we have that $\gamma_{u,w}^{\psi}\in \ell_p(\Gamma)$ then by \eqref{Eqn=PsiComp}  we have that $\Psi^{\lambda_u,\lambda_w} \in \calS_p$. Then Lemma  \ref{lemma:reduction-gradient-Sp-to-generators} shows  that $\Phi$ is gradient-$\calS_p$. The proof is similar for $p=\infty$.
		\end{proof}
	\end{lemma}

\begin{lemma}\label{lemma:reduction-to-set-generating-the-group2}
Let $\Phi = (\Phi_t)_{t \geq 0}$ be a QMS associated to a proper symmetric function $\psi: \Gamma \rightarrow \mathbb{Z}$ that is conditionally of negative type.  If $\Phi$ is gradient-$\cS_p$ for some $p \in [1, \infty]$ then for every $u,v \in \Gamma$ the function $\gamma^\psi_{u,v}: \Gamma \rightarrow \mathbb{Z}$  has compact support. In particular by \eqref{Eqn=FiniteRankish} we find that $\Psi^{\lambda_u, \lambda_v}$ is of finite rank and $\Phi$ is gradient-$\cS_p$ for all $p \in [1, \infty]$.
\end{lemma}
\begin{proof}
If $\psi$ takes integer values then so does $\gamma^\psi_{u,v}$ for all $u,v \in \Gamma$. Therefore $\gamma^\psi_{u,v}$ is contained in $\ell_p(\Gamma), p \in [1, \infty)$ or $c_0(\Gamma)$ if and only if $\gamma^\psi_{u,v}$ has compact support. The remainder of the lemma is directly clear.
\end{proof}

\subsection{Almost bimodularity of the Riesz transform for length functions}
We show that a QMS of Fourier multipliers associated with a $\mathbb{Z}_{\geq 0}$-valued length function automatically satisfies the  conditions of Theorem  \ref{Thm=AlmostBimodularRiesz}. Recall that $\psi: \Gamma \rightarrow  \mathbb{Z}_{\geq 0}$ is a length function if
\begin{equation}\label{Eqn=LengthFunction}
\psi(uw)\leq \psi(u)+\psi(w) \qquad \textrm{  for all } \qquad u,w \in \Gamma.
\end{equation}

\begin{theorem}\label{subsection:show-filtered-subexponential-growth}
	Let  $\psi: \Gamma \rightarrow \mathbb{Z}_{\geq 0}$ be a proper length function that is  conditionally of  negative type.   Then $\Delta_{\psi}$   is moreover filtered. If $\psi(\Gamma) = \mathbb{Z}_{\geq 0}$ or if  $\Gamma$ is finitely generated then $\Delta_\psi$ has subexponential growth.
  \end{theorem}
	\begin{proof}
		First of all we have that $(1+\Delta_{\psi})^{-1}(\lambda_{v}) = (1 + \psi(v))^{-1}  \lambda_v$ for all  $v\in \Gamma$.
As $\psi$ is proper this shows that  $(1+\Delta_{\psi})^{-1}$ is a compact operator on $\ell_2(\Gamma)$.
		Consider the finite dimensional spaces
		\begin{align}
		\CC[\Gamma](l) := \Span\{ \lambda_v \in \CC[\Gamma] :  \psi(v) = l \} \quad \text{ for  } l \in \mathbb{Z}_{\geq 0}.
		\end{align}
		Then $\CC[\Gamma](l)\Omega_{\tau}$ equals the eigenspace of $\Delta_{\psi}$ at the eigenvalue $l$.
	We have
		\begin{align}
		\CC[\Gamma] = \bigoplus_{l\geq 0}\CC[\Gamma](l) & \quad &\CC[\Gamma](l)\CC[\Gamma](k)\subseteq \bigoplus_{j=0}^{l+k}\CC[\Gamma](j) \text{ for } l,k\geq 0
		\end{align}
where $\bigoplus$ denotes the algebraic direct sum. The first equality holds because
		$\psi$ only takes positive integer values and the second equality holds because $\psi$ is a length function, i.e. \eqref{Eqn=LengthFunction}. This shows that $\Delta_{\psi}$ is filtered.

That $\Delta_{\psi}$    has subexponential growth follows in the first case from the fact that $\ZZ_{\geq 0}$ is the set of eigenvalues and we have $(l+1)/l \rightarrow 1$ as $l \rightarrow \infty$. In case $\Gamma$ is generated by a finite set $\Gamma_0$ we set
$K := \{ \max \psi(u) : u \in \Gamma_0 \}$.
  Then \eqref{Eqn=LengthFunction} implies that  $\mathbb{Z}_{\geq 0} \backslash \psi(\Gamma)$ cannot contain an interval of length $K+1$. Hence if $\lambda_0 \leq \lambda_1 \leq \ldots$ is an increasing enumeration of $\psi(\Gamma)$ then $\lambda_{k+1} \leq \lambda_k + K$. Hence  $\lambda_{k+1}/\lambda_k \rightarrow 1$ as $k \rightarrow \infty$.
	\end{proof}

\begin{corollary}\label{Cor=RieszBimodular}
Assume that $\Gamma$ is finitely generated.
	Let  $\psi: \Gamma \rightarrow \mathbb{Z}_{\geq 0}$ be a proper length function that is  conditionally of  negative type.   Let $\Phi$ be the associated QMS of Fourier multipliers. Then the Riesz transform $R_\Phi: \ell_2(\Gamma) \rightarrow \ell_2(\Gamma)_\nabla$ is almost bimodular.
\end{corollary}
\begin{proof}
This follows from Theorem \ref{Thm=AlmostBimodularRiesz} and Theorem \ref{subsection:show-filtered-subexponential-growth}.
\end{proof}

\begin{theorem}\label{Thm:AO}
Assume that $\Gamma$ is finitely generated and that $C_r^\ast(\Gamma)$ is locally reflexive.
	If there exists a   proper length function   $\psi: \Gamma \rightarrow \mathbb{Z}_{\geq 0}$ that is  conditionally of  negative type such that the associated QMS is gradient-$\calS_p$ for some $p \in [1, \infty)$. Then $\calL(\Gamma)$ has AO$^{+}$.
\end{theorem}
\begin{proof}
Let $H_\nabla := \ell_2(\Gamma)_{\nabla}$ be the gradient bimodule. Let $n \geqslant \frac{p}{2}$. Then by Proposition \ref{Prop:Spquasicontain} the bimodule $(H_{\nabla})_{\Gamma}^{\otimes n}$ is quasi-contained in the coarse bimodule. Let $R_\Phi:  \ell_{2}(\Gamma) \to H_{\nabla}$ be the Riesz transform. The kernel of $R_\Phi$ is spanned by all $\delta_g$ with $\psi(g) = 0$. Since $\psi$ is proper $\ker(R_\Phi)$ is finite dimensional. By Corollary \ref{Cor=RieszBimodular} we see that $R_\Phi$ is almost bimodular. By Lemma \ref{Lem:almostbimod} and Lemma  \ref{Lem=PartialIsoConvolution} the convolved Riesz transform $R_\Phi^{\ast n}\colon \ell^2(\Gamma) \to (H_{\nabla})_{\Gamma}^{\otimes n}$ is an almost bimodular partial isometry. Therefore we obtain  AO$^{+}$ from Theorem \ref{Thm=RieszImpliesAO}.

Note that in fact we could have avoided the tensor products in this proof by using Lemma \ref{lemma:reduction-to-set-generating-the-group2} instead.
\end{proof}

\section{Characterizing gradient-$\cS_p$ for Coxeter groups}	\label{section:semi-groups-word-length}
	
In this section we will consider the case of Coxeter groups.  For any Coxeter group the word length defines a proper length function that is conditionally of  negative type  \cite{bozejkoINFINITECOXETERGROUPS1988} (see also \cite[2.22]{Tits}). Therefore it determines a QMS of Fourier multipliers. The aim of this section is to find characterizations of when this specific QMS  is gradient-$\cS_p$.

Throughout Sections \ref{subsection:describing-support-of-gamma} -- \ref{Sect=GraphsWithParity} we give an almost characterization of gradient-$\cS_p$ in  terms of the Coxeter diagram. In particular we give sufficient conditions for gradient-$\cS_p$ that are easy to verify in Corollary \ref{corollary:gradient-Sp-Coxeter-groups} and Corollary \ref{corollary:no-labels-two-iff}. We also argue that these conditions are necessary for a large class of Coxeter groups.

In Section \ref{Sect=SmallAtInfinity}  we show that gradient-$\calS_p$  is equivalent to smallness at infinity of the Coxeter group. More precisely, a certain natural compactification of the Coxeter group that was considered in \cite{CapraceLecureux}, \cite{klisseTopologicalBoundariesConnected2020} (see also \cite{KlisseSimplicity}), \cite{LamThomas} is small at infinity.  This result can be understood directly after Section \ref{subsection:describing-support-of-gamma}.


\subsection{Preliminaries on Coxeter groups}
	Consider a finite set $S=\{s_1,..,s_n\}$ and a symmetric matrix $M = (m_{ij})_{1 \leq i,j \leq n}$ with $m_{i,j} \in \NN \cup \{\infty\}$ satisfying $m_{i,i} =1$ and $m_{i,j } \geq 2$ whenever $i\not=j$. Occasionally we write $m_{s_i,s_j}$ for $m_{i,j}$; this notation is convenient when considering $m_{s,t}$ without referring to the indices of the generators $s, t \in S$.

 We shall write $W=\langle S|M\rangle$ for the group freely generated by the set $S$ subject to the relation  $(s_is_j)^{m_{i,j}} = e$, where $e$ denotes the identity element of $W$. We call $W=\langle S|M\rangle$ a   {\it finite rank Coxeter system}. We sometimes simply say {\it Coxeter system} as we assume that they are all of finite rank.  A group that can be represented in such way is called a {\it Coxeter group}. Generally, a Coxeter group can be represented by different  pairs $S,M$.
 The group is called {\it right-angled} if moreover   $m_{i,j} \in \{1,2,\infty\}$ for all $1 \leq i, j \leq n$.
Throughout this entire section $W=\langle S|M\rangle$ is a general finitely generated Coxeter system.

When we deal with Coxeter groups we shall usually denote the elements of $W$ with boldface letters and the generators in $S$ with normal letters. This makes the exposition more clear.
Let $\ww \in W$. We say that an expression $w_1 \ldots w_n$ with  $w_i \in S$ is a reduced expression for $\ww$ if $\ww = w_1 \ldots w_n$ and this decomposition is of minimal length.   The minimal length is called the {\it word length} and which we denote by $\vert \ww \vert = n$.  We also set
 \[
 \psi_{S}:W\to \ZZ_{\geq 0}: \ww \mapsto \vert \ww \vert.
 \]

 \begin{theorem}[See \cite{bozejkoINFINITECOXETERGROUPS1988}]
For any Coxeter group $\psi_{S}:W\to \ZZ_{\geq 0}$  is conditionally of negative type.
 \end{theorem}

 Therefore by Theorem \ref{Thm=Schoenberg} there exists a QMS of Fourier multipliers on $\calL(W)$ associated with the word length function $\psi_S$.  The aim of the current Section \ref{section:semi-groups-word-length} is to describe when     this  QMS       has gradient-$\cS_p$.  Recall that by Lemmas \ref{lemma:reduction-to-set-generating-the-group} and \ref{lemma:reduction-to-set-generating-the-group2} we  must  thus investigate  for generators $u,w\in S$ when precisely $\gamma_{u,w}^{\psi_S}$ is finite rank where  $\gamma_{u,w}^{\psi_S}$  was defined in \eqref{Eqn=GammaFunction}.


	 \subsection{Describing support of the function $\gamma_{u,w}^{\psi_{S}}$} \label{subsection:describing-support-of-gamma}
The aim of this subsection is to describe the support of   $\gamma_{u,w}^{\psi_{S}}$ explicitly. In fact in anticipation of Section \ref{Sect=Weights} we will give this description for more general length functions $\psi$. Let $\mathds{1}( \cdot )$ be the indicator function which equals 1 if the statement within brackets is true.

	\begin{lemma}\label{lemma:shifting-identity}
		Let $W= \langle S|M\rangle$ be a Coxeter group. Suppose $\psi: W \rightarrow \mathbb{R}$ is conditionally of  negative type satisfying $\psi(\ww) = \psi(w_1) + ... + \psi(w_k)$ whenever $\ww = w_1...w_k$ is a reduced expression. Then for  $u,w\in S$ and
$\mathbf{v}\in W$ we have that
\[
|\gamma_{u,w}^{\psi}(\mathbf{v})|=2\psi(u)\mathds{1}(u\vv=\vv w)=2\psi(w)\mathds{1}(u\vv=\vv w).
\]
		\begin{proof}
			We first note that, since we have $u^2=w^2=e$ as they are generators, we have that
			$$\gamma_{u,w}^{\psi}(\mathbf{v}) = \gamma_{u,w}^{\psi}(u\mathbf{v}w) = -\gamma_{u,w}^{\psi}(u\mathbf{v}) = -\gamma_{u,w}^{\psi}(\mathbf{v}w).$$
			When $\vv$ is fixed, we can let $\mathbf{z}\in \{\mathbf{v},u\mathbf{v},\mathbf{v}w,u\mathbf{v}w\}$ be such that $|\mathbf{z}|=\min \{|\mathbf{v}|,|u\vv|,|\vv w|,|u\mathbf{v}w|\}$. Then we have $|\gamma_{u,w}^{\psi}(\mathbf{z})|= |\gamma_{u,w}^{\psi}(\mathbf{v})|$. Furthermore, because $|\zz|$ is minimal we have $|u\zz| = |\zz w|= |\zz|+1$. Thus, if $\zz = z_1....z_k$ is a reduced expression for $\zz$ we have that $uz_1...z_k$ and $z_1....z_kw$ are reduced expressions for $u\zz$ respectively $\zz w$. Therefore, $\psi(u\zz) = \psi(u)+\psi(\zz)$ and $\psi(\zz w)=\psi(\zz) + \psi(w)$. Hence
\[			
\begin{split}
			\gamma_{u,w}^{\psi}(\zz) &= \psi(u\zz w) + \psi(\zz) - \psi(u\zz) - \psi(\zz w)\\
			&=\psi(u\zz w) - \psi(\zz) - \psi(u) - \psi(w).
			\end{split}
\]
Now, since $|u\zz|=|\zz|+1$ we either have that $|u\zz w| = |\zz|+2$ or $|u\zz w| = |\zz|$. We shall consider these two separate cases, from which the result will follow.
			
			In the first case we have that $uz_1....z_kw$ is reduced so that $\psi(u\zz w) = \psi(u)+\psi(\zz)+\psi(w)$ and therefore $|\gamma_{u,w}^{\psi}(\vv)|=|\gamma_{u,w}^{\psi}(\zz)| = 0$. We note that in this case also $u\vv\not=\vv w$. Namely, $u\vv = \vv w$ would imply $u\zz = \zz w$ and hence $u\zz w = \zz$, which contradicts that $|u\zz w| = |\zz| +2$.

			 In the second case we have that $uz_1....z_kw$ is not reduced.
			Therefore, by the exchange condition (see \cite[Theorem 3.3.4.]{davisGeometryTopologyCoxeter2008}) and the fact that $|u\zz w|=|\zz|<|\zz w|$ we have that $uz_1....z_kw$ is equal to $z_1...z_{i-1}z_{i+1}..z_kw$ for some index $1\leq i\leq k$, or that $uz_1....z_kw = z_1....z_k$. Now in the former case we also have that
			$u\zz = z_1...z_{i-1}z_{i+1}...z_k$ so that $|u\zz|<|\zz|$ which is a contradiction. In this case we must thus have that $u\zz w = \zz$ and hence $u\zz=\zz w$. This then implies that $\psi(u\zz w)=\psi(\zz)$ and $\psi(u) = \psi(u\zz) - \psi(\zz) = \psi(\zz w) -\psi(\zz) = \psi(w)$. In this case we thus obtain that
			\[
			\gamma_{u,w}^{\psi}(\zz)=\psi(u\zz w) - \psi(\zz) - \psi(u) - \psi(w)
			= -2\psi(u) = -2\psi(w)
			\]
which shows that $|\gamma_{u,w}^{\psi}(\vv)| = |\gamma_{u,w}^{\psi}(\zz)| = 2\psi(u) = 2\psi(w)$ in this case.
			
			The result now follows from these cases. Namely, either we have that $|\gamma_{u,w}^{\psi}(\vv)|=0$ and that $\vv$ does not satisfy  $u\vv = \vv w$, or we have that $|\gamma_{u,w}^{\psi}(\vv)| = 2\psi(u) = 2\psi(w)$ and that $\vv$ does satisfy $u\vv = \vv w$.
			This thus shows us that $|\gamma_{u,w}^{\psi}(\vv)| = 2\psi(u)\mathds{1}(u\vv =\vv w) = 2\psi(w)\mathds{1}(u\vv = \vv w)$.\\
		\end{proof}
	\end{lemma}

\subsection{A characterization in terms of Coxeter diagrams}

	We note that for the word length $\psi_{S}$ we have $\psi_{S}(s)>0$ for all generators $s\in S$. Now by Lemma \ref{lemma:shifting-identity}, in order to see when $\gamma_{u,w}^{\psi_{S}}$ is finite-rank, we have to know what kind of words $\vv\in W$ have the property that $u\vv=\vv w$. For this we introduce some notation.\\
	
	For distinct $i,j\in \{1,...,|S|\}$ we will, whenever the label $m_{i,j}$ is finite, denote $k_{i,j} = \floor{\frac{m_{i,j}}{2}}\geq 1$. Now if $m_{i,j}$ is
	even, then $m_{i,j} = 2k_{i,j}$  and we set $\mathbf{r}_{i,j} = s_i(s_js_i)^{k_{i,j}-1}$. If $m_{i,j}$ is odd, then $m_{i,j} = 2k_{i,j} + 1$ and we set $\mathbf{r}_{i,j} = (s_is_j)^{k_{i,j}}$.
	Furthermore we set
	\begin{align}
	a_{i,j} &= s_i & b_{i,j} &= \begin{cases}
	s_i &	m_{i,j} \text{ even}\\
	s_j &	m_{i,j} \text{ odd}
	\end{cases}&
	c_{i,j} &= s_j & d_{i,j} &= \begin{cases}
	s_j & m_{i,j}	\text{ even}\\
	s_i	& m_{i,j}	\text{ odd}
	\end{cases}.
	\end{align}
	 Then $a_{i,j}$ and $b_{i,j}$ are respectively the first and last letter of the word $\mathbf{r}_{i,j}$. Furthermore when $m_{i,j}$ is even we have
	\[
c_{i,j}\mathbf{r}_{i,j} = s_{j}s_i(s_js_i)^{k_{i,j}-1} = (s_js_i)^{k_{i,j}} = (s_is_j)^{k_{i,j}} = \mathbf{r}_{i,j}s_j = \mathbf{r}_{i,j}d_{i,j},
\]
 and when $m_{i,j}$ is odd we have
	
\[
c_{i,j}\mathbf{r}_{i,j} = s_{j}(s_is_j)^{k_{i,j}} = s_i(s_js_i)^{k_{i,j}} = \mathbf{r}_{i,j}s_i = \mathbf{r}_{i,j}d_{i,j}.
\]
 Thus in either case $c_{i,j}\mathbf{r}_{i,j} = \mathbf{r}_{i,j}d_{i,j}$.\\

	For given generators $u,w\in S$ we will now check for what kind of words $\vv\in W$ with $|\vv|\leq |u\vv|,|\vv w|$ we have that $u\vv=\vv w$. In Proposition \ref{proposition:support-gamma} we then give a precise description of the support of $\gamma_{u,w}^{\psi_S}$.
	
	\begin{lemma}\label{lemma:word-sequence-expression}
		For generators $u,w\in S$ and a word $\mathbf{v}\in W$ with $|\mathbf{v}|\leq |u\vv|,|\vv w|$ we have $u\mathbf{v}=\mathbf{v}w$ if and only if $\mathbf{v}$ can be written in the reduced form $\mathbf{v} = \mathbf{r}_{i_1,j_1}.....\mathbf{r}_{i_k,i_k}$ so that $u = c_{i_1,j_1}$ and $w = d_{i_k,j_k}$ and so that for $l=1,...,k-1$ we have that $c_{i_{l+1},j_{l+1}} = d_{i_l,j_l}$ and $a_{i_{l+1},j_{l+1}}\not\in \{s_{i_l},s_{j_l}\}$ and
		 $b_{i_l,j_l}\not\in \{s_{i_{l+1}},s_{j_{l+1}}\}$.
	\end{lemma}
		\begin{proof}
			First, suppose that $\mathbf{v}$ can be written in the given form $\mathbf{v} =\mathbf{r}_{i_1,j_1}.....\mathbf{r}_{i_k,i_k}$ with the given conditions on $c_{i_l,j_l}$ and $d_{i_l,j_l}$. Then since we have $c_{i_l,j_l}\mathbf{r}_{i_l,j_l} = \mathbf{r}_{i_l,j_l}d_{i_l,j_l} = \mathbf{r}_{i_l,j_l}c_{i_{l+1},j_{l+1}}$ for $l=1,...,k-1$, and since $u = c_{i_1,j_1}$ and $w =d_{i_k,j_k}$ we have $u\mathbf{v}=\mathbf{v}w$, which shows the
			`if' direction.\\
			
			We now prove the opposite direction. First note that the statement holds for $\vv =e$ as this can be written as the empty word.
		 	We now prove by induction on $n$ that for $\vv\in W$ with $|\vv|\geq 1$ and $|\vv|\leq n$ and $|\vv|\leq |u\vv|,|\vv w|$ and $u\vv=\vv w$ for some $u,w\in S$,	 	
		 	we can write $\vv$ in the given form. Note first that the statement holds for $n=0$, since then no such $\vv\in W$ exists. Thus, assume that the statement holds for $n-1$, we prove the statement for $n$. Let $u,w\in S$ and $\mathbf{v}\in W$ be with $|\mathbf{v}|=n$ and $|u\mathbf{v}|=|\mathbf{v}w|=|\mathbf{v}|+1$ and $u\vv=\vv w$. 	
			Let $v_1 \ldots v_n$ be a reduced expression for $\mathbf{v}$. Then the expression $u v_1 \ldots v_n$ and $v_1 \ldots v_n w$ are reduced expressions for $u\mathbf{v}=\mathbf{v}w$. In particular we have $u\not=v_1$. Set $m:=m_{u,v_1}$.
			Now, since $u\mathbf{v}$ and $\mathbf{v}w$ are equal and $u\not=v_1$, we can as in the proof of
			\cite[theorem 3.4.2(ii)]{davisGeometryTopologyCoxeter2008} find
			a reduced expression $y_1 \ldots y_{n+1}$ for $u\mathbf{v}$ with
			$n\geq m-1$ so that $y_1 \ldots y_m = uv_1uv_1 \ldots u$ whenever $m$ is odd, and $y_1 \ldots y_m = uv_1 \ldots uv_1$ whenever $m$ is even. This is to say that if we let $i_0,j_0\in \{1,...,|S|\}$ be such that $v_1 = s_{i_0}$ and $u = s_{j_0}$, then we have that $\mathbf{r}_{i_0,j_0} =y_2 \ldots y_m$ and $c_{i_0,j_0}=s_{j_0}=u$. Note that by the proof of \cite[theorem 3.4.2(ii)]{davisGeometryTopologyCoxeter2008} we have in particular that $m<\infty$.
			Now moreover, since $y_1=u$ we have that  $y_2 \ldots y_{n+1}w$ is an expression for $\mathbf{v}w$, and this expression is reduced since $|\mathbf{v}w|=n+1$. \\
			
			Now suppose that $m = n+1$, then $\mathbf{v} = \mathbf{r}_{i_0,j_0}$ and $i_0\not=j_0$ since $u\not=v_1$. Now, we have $u = s_{j_0} =c_{i_0,j_0}$ and furthermore, since $\mathbf{r}_{i_0,j_0}d_{i_0,j_0} = c_{i_0,j_0}\mathbf{r}_{i_0,i_0} = u\mathbf{v} =\mathbf{v}w = \mathbf{r}_{i_0,j_0}w$, also $w=d_{i_0,j_0}$. Thus in this case we can write $\mathbf{v}$ in the given form.\\

			Now suppose $m<n+1$ and define $\mathbf{v}'= y_{m+1} \ldots y_{n+1}$ and $u' = d_{i_0,j_0}$ and $w' =w$. Note that since $u =s_{j_0} = c_{i_0,j_0}$ and $u' = d_{i_0,j_0}$ we have
			
			$$\mathbf{r}_{i_0,j_0}u'\mathbf{v}' = u\mathbf{r}_{i_0,j_0}\mathbf{v}'
			= u\mathbf{v} = \mathbf{v}w = \mathbf{r}_{i_0,j_0}\mathbf{v}'w'.$$
			Therefore $u'\mathbf{v}' = \mathbf{v}'w'$. Moreover $|u'\mathbf{v}'| = |\mathbf{v}'w'| = |\mathbf{v}'| +1$ since $y_{m+1} \ldots y_{n+1} w$ is a reduced expression for $\vv' w$. Now, since also $|\mathbf{v}'|\geq 1$ and $|\mathbf{v}'|\leq n-1$ we have by the induction hypothesis that there is a reduced expression $\mathbf{v}' = \mathbf{r}_{i_1,j_1}....\mathbf{r}_{i_k,j_k}$ for some indices $i_l,j_l\in \{1,...,|S|\}$ with $i_l\not=j_l$ so that $u' = c_{i_1,j_1}$ and $w' = d_{i_k,j_k}$ and so that for $l=1, \ldots ,k-1$ we have that $c_{i_{l+1},j_{l+1}} = d_{i_l,j_l}$ and
			$a_{i_{l+1},j_{l+1}}\not\in \{s_{i_l},s_{j_l}\}$ and
			$b_{i_l,j_l}\not\in \{s_{i_{l+1}},s_{j_{l+1}}\}$.
			Hence we can write $\mathbf{v} = \mathbf{r}_{i_0,j_0}\mathbf{v}' = \mathbf{r}_{i_0,j_0} \ldots \mathbf{r}_{i_k,j_k}$.
			We also have $u = s_{j_0} = c_{i_0,j_0}$ and $w = w' = d_{i_k,j_k}$ and $d_{i_0,j_0} = u' = c_{i_1,j_1}$. Furthermore, since $|\vv| = n = (m-1) + (n-m + 1) = |\mathbf{r}_{i_0,j_0}| + |\mathbf{v}'|$, and since the expression for $\mathbf{v}'$ is reduced we thus have that the expression for $\mathbf{v}$ is also reduced. Now suppose that $b_{i_0,j_0} \in \{s_{i_1},s_{j_1}\}$. We note that $b_{i_0,j_0}\not=d_{i_0,j_0}=c_{i_1,j_1}\not=a_{i_1,j_1}$. Now as also $c_{i_1,j_1},a_{i_1,j_1}\in \{s_{i_1},s_{j_1}\}$ we obtain that $a_{i_1,j_1} =b_{i_0,j_0}$.
			However as $\rr_{i_0,j_0}$ ends with $b_{i_0,j_0}$ and as $\rr_{i_1,j_1}$ starts with $a_{i_1,j_1}$ we then obtain that $\rr_{i_0,j_0}\rr_{i_1,j_1}$ is not a reduced expression. This contradicts the fact that the expression for $\vv$ is reduced.
			Likewise, if $a_{i_1,j_1}\in \{s_{i_0},s_{j_0}\}$ we have
			because of the fact that $a_{i_1,j_1}\not=c_{i_1,j_1}=d_{i_0,j_0}\not=b_{i_0,j_0}$ and $d_{i_0,j_0},b_{i_0,j_0} \in \{s_{i_0},s_{j_0}\}$ that $a_{i_1,j_1} = b_{i_0,j_0}$. This then shows that $\rr_{i_0,j_0}\rr_{i_1,j_1}$ is not a reduced expression, which contradicts the fact that the expression for $\vv$ is reduced.
			This proves the lemma.
		\end{proof}

	\begin{proposition}\label{proposition:support-gamma}
		Let $u,w\in S$. Then we have $\zz\in \supp(\gamma_{u,w}^{\psi_S})$ if and only if $\zz \in \{\vv,u\vv,\vv w,u\vv w\}$, where $\vv$ is a word as in \cref{lemma:word-sequence-expression}.
\end{proposition}
		\begin{proof} It is clear that if $\zz\in \{\vv,u\vv,\vv w,u\vv w\}$ where $\vv$ is of the form of Lemma \ref{lemma:word-sequence-expression}, that we then have that $u\zz =\zz w$, and hence by Lemma \ref{lemma:shifting-identity} that $\psi_{u,w}^{\psi_{S}}(\zz)\not=0$.
		For the other direction we suppose that $\zz\in \supp(\gamma_{u,w}^{\psi_S})$. Then we have that $u\zz = \zz w$ holds by \cref{lemma:shifting-identity}. Now there is a $\vv \in \{\zz, u\zz,\zz w,u\zz w\}$ such that $|\vv|\leq |u\vv|,|\vv w|$. This word $\vv$ moreover satisfies $u\vv = \vv w$ as we had $u\zz = \zz w$. Now, this means that $\vv$ can be written in an expression as in Lemma \ref{lemma:word-sequence-expression}. Last, we note that $\zz\in \{\vv,u\vv,\vv w,u\vv w\}$, which finishes the proof.
		\end{proof}
	
	\subsection{Parity paths in Coxeter diagram}
	In Proposition \ref{proposition:support-gamma} we showed precisely for what kind of words $\vv\in W$ we have $\vv\in \supp(\gamma_{u,w}^{\psi_S})$. The question is now whether this support is finite or infinite. It follows from the proposition that the support is finite if and only if there exist only finitely many words $\vv\in W$ that can be written in the form $\vv = \rr_{i_1,j_1}....\rr_{i_k,j_k}$ with the condition from Lemma \ref{lemma:word-sequence-expression}. To answer the question on whether this is the case, we shall identify these expressions with certain walks in a graph. The following defines essentially the Coxeter diagram with the difference that in a Coxeter diagram the edges that are labeled with $m_{i,j} = 2$ are deleted and those labeled with $m_{i,j} = \infty$ are added. Recall that a graph is simplicial if it contains no double edges and no edges from a point to itself. \\
	
\begin{definition}\label{Dfn=Graph}
	We will let $\Graph_S(W)=(V,E)$ be the complete simplicial graph with vertex set $V=S$ and labels $m_{i,j}$ on the edges $\{s_i,s_j\}\in E$.
\end{definition}

\begin{definition}\label{Dfn=Parity}
	Let $k\geq 1$ and $i_l,j_l\in \{1,..,|S|\}$ for $l=1,...,k$. Let $P=(s_{j_1},s_{i_1},s_{j_2},.....,s_{j_{k}},s_{i_{k}})$ be a  walk in the $\Graph_S(W)$, which has even length.
	We will say that $P$ is a \textit{parity path} if the edges of $P$ have finite labels, and if (1)  $i_{l}\not=j_{l}$ for all $l$; (2)  for $l=1,..,k-1$ we have $s_{j_{l+1}} = d_{i_{l},j_l}$ and (3) $i_{l+1}\not\in \{i_{l},j_{l}\}$.
	We will moreover call the parity path $P$ a \textit{cyclic parity path} if the path
	$\overline{P}:= (s_{j_1},s_{i_1},....,s_{j_k},s_{i_k},s_{j_1},s_{i_1})$ is a parity path.\\
\end{definition}	

	The intuition for a parity path is that if you walk an edge with odd label, you have to stay there for one turn and then continue your walk over a different edge than you came from. Furthermore, when you walk an edge with an even label you have to return directly over the same edge, and then continue your walk using another edge. Note that in both cases you may still use same edges as before at a later point in your walk. A cyclic parity path is defined so that walking the same path any number of times in a row gives you a parity path.\\

 We state the following definition.
	\begin{definition}	
			An elementary M-operation on a word $v_1....v_k$ is one of the following operations
			\begin{enumerate}
				\item Delete a subword of the form $s_is_i$.
				\item Replace an alternating subword of the form $s_is_js_is_j...$ of length $m_{i,j}$ by the alternating word $s_js_is_js_i....$ of the same length.
			\end{enumerate}
			A word is called M-reduced if it cannot be shortened by elementary M-operations.
		\end{definition}
	
	We shall now show in the following two Theorems that the gradient-$\mathcal{S}_p$ property of the semi-group $(\Phi_t)_{t\geq 0}$ on $\calL(W)$ associated to the word length $\psi_{S}$, is almost equivalent with the non-existence of cyclic parity paths in $\Graph_S(W)$.

	\begin{theorem}\label{theorem:parity-path-implies-not-gradient-Sp}
	Let $W = \langle S|M\rangle$ be a Coxeter system. Suppose there is a cyclic parity path $$P=(s_{j_1},s_{i_1},s_{j_2},...,s_{j_k},s_{i_k})$$ in $\Graph_S(W)$ in which the labels $m_{i_l,j_l},m_{i_l,i_{l+1}}$, $m_{j_l,i_{l+1}}$ are all unequal to $2$.
	Then the semi-group $(\Phi_{t})_{t\geq 0}$ associated to the word length $\psi_{S}$ is not gradient-$\mathcal{S}_p$ for any $p\in [1,\infty]$.\\
	\begin{proof}
		Suppose the assumptions hold. Then we have that there exists a parity path of the form $\overline{P} =(s_{j_1},s_{i_1},s_{j_2},...,s_{j_k},s_{i_k},s_{j_{k+1}},s_{i_{k+1}})$ where $s_{i_1} = s_{i_{k+1}}$ and $s_{j_1} = s_{j_{k+1}}$.
		We will denote $\mathbf{v}_1 = \mathbf{r}_{i_1,j_1}...\mathbf{r}_{i_k,j_k}$.
		We note that by the definition of a parity path we have $d_{i_l,j_l} = s_{j_{l+1}} = c_{i_{l+1},j_{l+1}}$ for $l=1,..,k-1$ and $d_{i_k,j_k}=s_{j_{k+1}}=s_{j_{1}} = c_{i_{1},j_1}$. We now define $u = c_{i_1,j_1}=d_{i_k,j_k}$. Now we thus have $u\mathbf{v}_1 = \mathbf{v}_1u$.
		This means by Lemma \ref{lemma:shifting-identity} that $\gamma_{u,u}^{\psi_{S}}(\vv_1)\not=0$. We show that $\psi_{S}(\mathbf{v}_1)\geq k$. To see this, note that $a_{i_{l+1},j_{l+1}}=s_{i_{l+1}} \not\in \{s_{i_{l}},s_{j_{l}}\}$ by the definition of the parity path. Furthermore, since $b_{i_{l},j_{l}}\not=d_{i_{l},j_{l}}=c_{i_{l+1},j_{l+1}}$ and $b_{i_{l},j_{l}}\not= a_{i_{l+1},j_{l+1}}$ (as $a_{i_{l+1},j_{l+1}}\not\in \{s_{i_{l}},s_{j_{l}}\}\ni b_{i_{l},j_{l}}$) and $a_{i_{l+1},j_{l+1}}=s_{i_{l+1}}\not=s_{j_{l+1}}=c_{i_{l+1},j_{l+1}}$ we have that $b_{i_{l},j_{l}}\not\in\{a_{i_{l+1},j_{l+1}},c_{i_{l+1},j_{l+1}}\}= \{s_{i_{l+1}},s_{j_{l+1}}\}$. Now, since there are no labels $m_{i_l,j_l}$ equal to $2$ we have that the sub-words $\mathbf{r}_{i_l,j_l}$ contain both elements $s_{i_l}$ and $s_{j_l}$. This means, since $a_{i_{l+1},j_{l+1}}\not\in \{s_{i_l},s_{j_l}\}$ and $b_{i_l,j_l}\not\in \{s_{i_{l+1}},s_{j_{l+1}}\}$, that the only sub-words of $\mathbf{v}_1$ of the form $s_is_js_i \ldots s_{i} s_j$ or $s_is_js_i \ldots s_j s_i$ are the sub-words of $\mathbf{r}_{i_l,j_l}$ for some $l=1,...,k$, and the words
		$b_{i_l,j_l}a_{i_{l+1},j_{l+1}}$  for $l=1,..,k-1$.
		  For an alternating subword $\xx$ of $\rr_{i,j}$ for some $i,j$ we have that $\xx$ is an alternating sequence of $s_i$'s and $s_j$'s and further
			\[|  \xx |\leq |\rr_{i,j}| \leq m_{i,j} -1.
\]
		Furthermore, for a word $s_is_j$ with $s_i=b_{i_l,j_l}$ and $s_j = a_{i_{l+1},j_{l+1}}$ for some $l=1, \ldots ,k-1$ (in which case we have $i\in \{i_l,j_l\}$ and $j = i_{l+1}$)	we have that
\[
|s_is_j| =2\leq \min\{m_{i_l,i_{l+1}},m_{j_l,i_{l+1}}\}-1\leq m_{i,j}-1.
\]

 Furthermore, there are no sub-words of $\vv_1$ of the form $s_i s_i$. This means that the expression for $\vv_1$ is $M$-reduced, and therefore, by \cite[Theorem 3.4.2]{davisGeometryTopologyCoxeter2008}, that the expression is reduced. This means that $\psi_{S}(\vv_1)\geq k$. Now, since we can create cyclic parity paths $P_n$ by walking over $P$ a $n$ number of times, we can create $\vv_n\in W$ with $\psi_{S}(\vv_n)\geq nk$ and $\gamma_{u,u}^{\psi_{S}}(\vv_n)\not=0$. Therefore $\gamma_{u,u}^{\psi_{S}}$ is not finite rank, and hence the semi-group $(\Phi_t)_{t\geq 0}$ is not gradient-$\calS_p$ for any $p\in [1,\infty]$.
	\end{proof}
	\end{theorem}

	\begin{theorem}\label{theorem:no-parity-path-implies-gradient-Sp}
	Let $W=\langle S|M\rangle$ be a Coxeter group. If there does not exist a cyclic parity path in $\Graph_S(W)$ then the semi-group $(\Phi_t)_{t\geq 0}$ associated to the word length $\psi_{S}$ is gradient-$\mathcal{S}_p$ for all $p\in[1,\infty]$.
	\end{theorem}
	\begin{proof}
		Suppose that $(\Phi_t)_{t\geq 0}$ is not gradient-$\mathcal{S}_p$ for some $p\in[1,\infty]$. We will show that a cyclic parity path exists. Namely, since the semi-group is not gradient-$\mathcal{S}_p$,
		there exist by Lemma \ref{lemma:reduction-to-set-generating-the-group2} generators $u,w\in S$ for which $\gamma_{u,w}^{\psi_{S}}$ is not finite rank. Set $m =\max\{m_{i,j}: 1\leq i,j\leq |S|\}\setminus\{\infty\}$. We can thus let $\zz\in \supp(\gamma_{u,w}^{\psi_S})$ be with $\psi_{S}(\zz)>m|S|^2+2$.
		Then by Proposition \ref{proposition:support-gamma} there is a $\vv\in \{\zz,u\zz,\zz w,u\zz w\}$ such that we can write $\vv$ in reduced form $\vv = \mathbf{r}_{i_1,j_1}....\mathbf{r}_{i_k,j_k}$ with the conditions as in Lemma \ref{lemma:word-sequence-expression}. Now define the path $P = (s_{j_1},s_{i_1},....,s_{j_k},s_{i_k})$.
		We show that this is a parity path. By the properties that we obtained from Lemma \ref{lemma:word-sequence-expression}, we have that $i_l\not=j_l$ and that $m_{i_l,j_l}<\infty$ for all $l$. Moreover $s_{j_{l+1}} = c_{i_{l+1},j_{l+1}} = d_{i_{l},j_l}$ and $s_{i_{l}} = a_{i_l,j_l}\not\in \{s_{i_{l+1}},s_{j_{l+1}}\}$. This shows that $P$ is a parity path.		
		Note furthermore that since $\psi_S(\vv)\geq \psi_S(\zz)-2> m|S|^2$, we have that $P$ has length $|P|=2k \geq 2\frac{\psi_S(\vv)}{m}>2|S|^2$. Therefore, there must exist indices $l<l'$ such that  $(s_{j_{l}},s_{i_l}) = (s_{j_{l'}},s_{i_{l'}})$. The sub-path $(s_{j_{l},s_{i_l}}, \ldots ,s_{j_{l'-1},j_{l'-1}})$ then is a cyclic parity path.\\		
	\end{proof}

	\subsection{Characterization of graphs that contain cyclic parity paths}\label{Sect=GraphsWithParity} In the previous subsection, in Theorem \ref{theorem:parity-path-implies-not-gradient-Sp} and Theorem \ref{theorem:no-parity-path-implies-gradient-Sp} we have shown that the gradient-$\calS_p$ property is almost equivalent to the non-existence of a cyclic parity path. We shall now characterize in Proposition \ref{proposition:characterization-parity-paths} precisely when a graph possesses a cyclic parity path. The content of this proposition is moreover visualized in Figure \ref{fig:graphs-with-and-without-cpp}. Thereafter we state two corollaries that follow from this proposition and from Theorem \ref{theorem:parity-path-implies-not-gradient-Sp} and Theorem \ref{theorem:no-parity-path-implies-gradient-Sp}. These corollaries give an `almost' complete characterization of the types of Coxeter systems for which the semi-group associated to $\psi_S$ is gradient-$\calS_p$.\\
	
	The following proposition shows exactly when a cyclic parity path $P$ in the graph $\Graph_S(W)$ exists. Recall that a forest is a union of trees. A graph is a tree if it has no loops/cycles.
	
\begin{figure}[!ht]
	  \textbf{Graphs with and without a cyclic parity path}\par\medskip
\centering
\subfloat[Graph with no cyclic parity path]{
	\begin{tikzpicture}[baseline]
\node[draw] at (2, 0)			(a1) {$s_1$};
\node[draw] at (1, 1.86) 	  	(a2) {$s_2$};
\node[draw] at (-1, 1.86)   	(a3) {$s_3$};
\node[draw] at (-2, 0)   		(a4) {$s_4$};
\node[draw] at (-1, -1.86)   	(a5) {$s_5$};
\node[draw] at (1, -1.86)  	 	(a6) {$s_6$};

\draw[blue,thick] (a1) -- node[midway,above,xshift=0.1cm] {\tiny $5$} (a2);
\draw[] (a1) -- node[pos = 0.15, above] {\tiny $\infty$} (a3);
\draw[] (a1) -- node[pos = 0.35, above] {\tiny $\infty$} (a4);
\draw[blue,thick] (a1) -- node[pos = 0.85,below] {\tiny $13$} (a5);
\draw[blue,thick] (a1) -- node[midway,below,xshift=0.1cm] {\tiny $3$} (a6);

\draw[orange,thick] (a2) -- node[midway,above] {\tiny $6$} (a3);
\draw[] (a2) -- node[pos = 0.15, above] {\tiny $\infty$} (a4);
\draw[] (a2) -- node[pos = 0.35, above,xshift=-0.1cm] {\tiny $\infty$} (a5);
\draw[] (a2) -- node[pos = 0.8, right] {\tiny $\infty$} (a6);

\draw[blue,thick] (a3) -- node[midway,above,xshift=-0.1cm] {\tiny $9$} (a4);
\draw[] (a3) -- node[pos = 0.2, left] {\tiny $\infty$} (a5);
\draw[] (a3) -- node[pos = 0.3, below,xshift=-0.1cm] {\tiny $\infty$} (a6);

\draw[] (a4) -- node[midway,below,xshift=-0.1cm] {\tiny $\infty$} (a5);
\draw[] (a4) -- node[pos = 0.15,below] {\tiny $\infty$} (a6);

\draw[] (a5) -- node[midway,below] {\tiny $\infty$} (a6);
\end{tikzpicture}
\label{fig:graph-without-cyclic-parity-path}}
\quad
\subfloat[Graph with a cyclic parity path]{
\begin{tikzpicture}[baseline]
\node[draw] at (2, 0)			(a1) {$s_1$};
\node[draw] at (1, 1.86) 	  	(a2) {$s_2$};
\node[draw] at (-1, 1.86)   	(a3) {$s_3$};
\node[draw] at (-2, 0)   		(a4) {$s_4$};
\node[draw] at (-1, -1.86)   	(a5) {$s_5$};
\node[draw] at (1, -1.86)  	 	(a6) {$s_6$};

\draw[blue,thick] (a1) -- node[midway,above,xshift=0.1cm] {\tiny $5$} (a2);
\draw[] (a1) -- node[pos = 0.15, above] {\tiny $\infty$} (a3);
\draw[] (a1) -- node[pos = 0.35, above] {\tiny $\infty$} (a4);
\draw[blue,thick] (a1) -- node[pos = 0.85,below] {\tiny $13$} (a5);
\draw[blue,thick] (a1) -- node[midway,below,xshift=0.1cm] {\tiny $3$} (a6);

\draw[orange,thick] (a2) -- node[midway,above] {\tiny $6$} (a3);
\draw[] (a2) -- node[pos = 0.15, above] {\tiny $\infty$} (a4);
\draw[] (a2) -- node[pos = 0.35, above,xshift=-0.1cm] {\tiny $\infty$} (a5);
\draw[] (a2) -- node[pos = 0.8, right] {\tiny $\infty$} (a6);

\draw[blue,thick] (a3) -- node[midway,above,xshift=-0.1cm] {\tiny $9$} (a4);
\draw[] (a3) -- node[pos = 0.2, left] {\tiny $\infty$} (a5);
\draw[] (a3) -- node[pos = 0.3, below,xshift=-0.1cm] {\tiny $\infty$} (a6);

\draw[orange,thick] (a4) -- node[midway,below,xshift=-0.1cm] {\tiny $4$} (a5);
\draw[] (a4) -- node[pos = 0.15,below] {\tiny $\infty$} (a6);

\draw[] (a5) -- node[midway,below] {\tiny $\infty$} (a6);
\end{tikzpicture}\label{fig:graph-with-cyclic-parity-path}}
\quad
\subfloat[Graph with a cyclic parity path]{
	\begin{tikzpicture}[baseline]
	\node[draw] at (2, 0)			(a1) {$s_1$};
	\node[draw] at (1, 1.86) 	  	(a2) {$s_2$};
	\node[draw] at (-1, 1.86)   	(a3) {$s_3$};
	\node[draw] at (-2, 0)   		(a4) {$s_4$};
	\node[draw] at (-1, -1.86)   	(a5) {$s_5$};
	\node[draw] at (1, -1.86)  	 	(a6) {$s_6$};

	\draw[blue,thick] (a1) -- node[midway,above,xshift=0.1cm] {\tiny $5$} (a2);
	\draw[] (a1) -- node[pos = 0.15, above] {\tiny $\infty$} (a3);
	\draw[] (a1) -- node[pos = 0.35, above] {\tiny $\infty$} (a4);
	\draw[blue,thick] (a1) -- node[pos = 0.85,below] {\tiny $13$} (a5);
	\draw[blue,thick] (a1) -- node[midway,below,xshift=0.1cm] {\tiny $3$} (a6);
	
	\draw[orange,thick] (a2) -- node[midway,above] {\tiny $6$} (a3);
	\draw[] (a2) -- node[pos = 0.15, above] {\tiny $\infty$} (a4);
	\draw[] (a2) -- node[pos = 0.35, above,xshift=-0.1cm] {\tiny $\infty$} (a5);
	\draw[] (a2) -- node[pos = 0.8, right] {\tiny $\infty$} (a6);
	
	\draw[blue,thick] (a3) -- node[midway,above,xshift=-0.1cm] {\tiny $9$} (a4);
	\draw[] (a3) -- node[pos = 0.2, left] {\tiny $\infty$} (a5);
	\draw[] (a3) --node[pos = 0.3, below,xshift=-0.1cm] {\tiny $\infty$} (a6);
	
	\draw[] (a4) -- node[midway,below,xshift=-0.1cm] {\tiny $\infty$} (a5);
	\draw[] (a4) -- node[pos = 0.15,below] {\tiny $\infty$} (a6);
	
	\draw[blue,thick] (a5) -- node[midway,below] {\tiny $5$} (a6);
	\end{tikzpicture}\label{fig:graph-with-cyclic-parity-path2}}

\caption{The graph $\Graph_{S}(W)$ is denoted for three different Coxeter systems $W = \langle S|M\rangle$ with $|S|=6$. In each of the graphs the label $m_{i,j}$ is shown on the edge $\{s_i,s_j\}$. We colored the edge orange when the label is even, we colored it blue when the label is odd, and we colored it black when the label is infinity. The relations we imposed on the generators are almost the same in the three cases. They only differ on the edges $\{s_4,s_5\}$ and $\{s_5,s_6\}$.
The graph in (A) satisfies the assumptions of Proposition \ref{proposition:characterization-parity-paths} and hence does not contain a cyclic parity path.
The graph in (B) does not satisfy the assumptions of the proposition as for the connected component $C=\{s_3,s_4\}$ of $(V,E_1)$ there are two distinct edges $\{s_2,s_3\}$ and $\{s_4,s_5\}$ with even label and with (at least) one endpoint in $C$. Therefore the graph contains a cyclic parity path. One is given by
$P=(s_3,s_2,s_3,s_4,s_4,s_5,s_4,s_3)$ (another cyclic parity path uses the node $s_1$).
The graph in (C) does also not satisfy the assumptions of the proposition as it contains a cycle with odd labels. Here a cyclic parity path is given by $P=(s_1,s_5,s_5,s_6,s_6,s_1)$ (another cyclic parity path is obtained by walking in reverse order).}
\label{fig:graphs-with-and-without-cpp}
\end{figure}
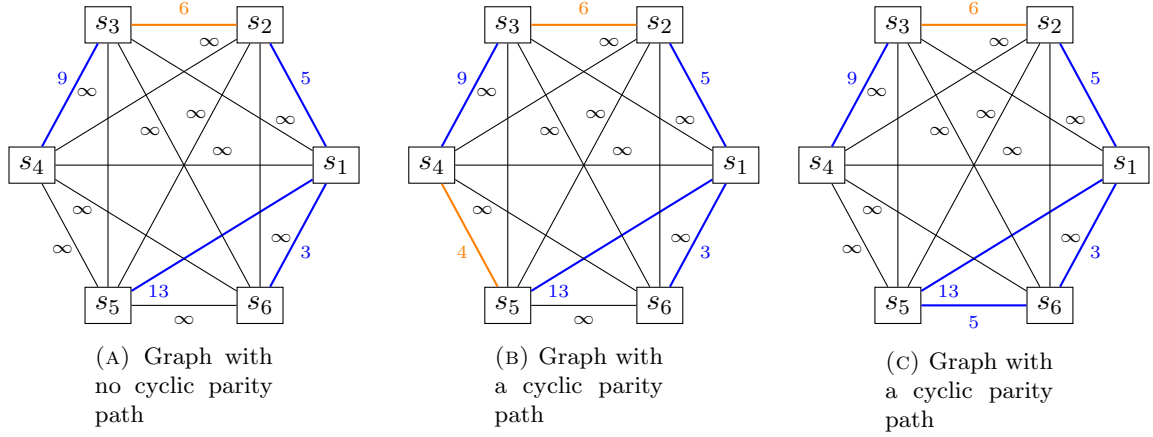

	\begin{proposition}\label{proposition:characterization-parity-paths}
		Let us denote $V=S$ and $E_0=\{\{i,j\}: m_{i,j}\in 2\NN\}$ and $E_1 =\{\{i,j\}: m_{i,j} \in 2\NN+1\}$. Then there does not exist a cyclic parity path $P$ in $\Graph_S(W)$ if and only if $(V,E_1)$ is a forest, and for every connected component $C$ of $(V,E_1)$ there is at most one edge $\{t,r\}\in E_0$ with $t\in C$ and $r\not\in C$, and for every connected component $C$ of $(V,E_1)$ there is no edge $\{t,t'\}\in E_0$ with $t,t'\in C$.
	\end{proposition}
		\begin{proof}
				First suppose that $(V,E_1)$ is not a forest. Then we can find a cycle $Q = (s_{j_1},s_{j_2}, \ldots ,s_{j_k},s_{j_1})$ in $(V,E_1)$. Now, since all edges are odd, this means that $$P = (s_{j_1},s_{j_2},s_{j_2},s_{j_3},s_{j_3},....,s_{j_{k}},s_{j_k},s_{j_1})$$ is a cyclic parity path. Indeed, if we denote $j_{k+1} := j_1$ and $j_{k+2}:=j_2$, then $j_l\not=j_{l+1}$ for $l=1,\ldots ,k$ and we have $s_{j_{l+1}} = d_{j_{l+1},j_{l}}$ and $j_{l+2}\not\in \{j_{l+1},j_l\}$, which shows all conditions hold.
			
			Now suppose that there is a connected component $C$ of $(V,E_1)$ for which there are two distinct edges $\{t_1,r_1\},\{t_2,r_2\}\in E_0$ with $t_1,t_2\in C$ and $r_1,r_2\not\in C$.
			If $t_1 = t_2$ then $r_1\not=r_2$ and a cyclic parity path is given by
			$P = (t_1,r_1,t_1,r_2)$.
			In the case that $t_1$ and $t_2$ are distinct there is a simple path $Q = (t_1,s_{j_1},...,s_{j_k},t_2)$ in $(V,E_1)$ from $t_1$ to $t_2$. The path $$P = (t_1,s_{j_1},s_{j_1},s_{j_2},s_{j_2},...,s_{j_k},s_{j_k},t_2,t_2,r_2,t_2,s_{j_k},s_{j_k},s_{j_{k-1}},s_{j_{k-1}},...,s_{j_1},s_{j_1},t_1,t_1,r_1)$$then is a cyclic parity path.
			Indeed, just as the previous case we have that
			the paths
			$$P_1:=(t_1,s_{j_1},s_{j_1},s_{j_2},s_{j_2}, \ldots , s_{j_k},s_{j_k},t_2)$$
			and
			$$P_2:=(t_2,s_{j_k},s_{j_k},s_{j_{k-1}},s_{j_{k-1}},\ldots, s_{j_1},s_{j_1}, t_1)$$
			are parity paths, since they are obtained from a simple path in $(V,E_1)$. We then only have to check that in the middle and at the start/end of the path $P$ the conditions are satisfied. For the middle, we see that indeed $r_2\not\in \{s_{j_k},t_2\}$ as the label of the edge between $t_2$ and $r_2$ is even. Furthermore, since $P_1$ is a parity path we have that $s_{j_k}\not=t_2$. Thus also $s_{j_k}\not\in \{t_2,r_2\}$. Furthermore, if we let $i,j$ be such that $t_2=s_j$, $r_2=s_i$, then since $m_{j_k,j}$ is odd, we have that $t_2 = d_{j,j_k}$ and since $m_{i,j}$ is even we have
			$t_2 = d_{i,j}$. This shows all conditions in the middle. The conditions at the start/end hold by symmetry. Thus $P$ is indeed a cyclic parity path.
			
			Now, suppose that there is a connected component $C$ of $(V,E_1)$ for which there exists an edge $\{t,t'\}\in E_0$ with $t,t'\in C$. Then we can, similar to what we just did, obtain a cyclic parity path by taking $t_1 =t$ and $t_2=t'$ and $r_1 = t'$ and $r_2 = t$.\\

			We now prove the other direction. Thus, suppose that $(V,E_1)$ is a forest and that for every connected component $C$ there is at most edge $\{t,r\}\in E_0$ with $t\in C$ and $r\in V$, and that for every connected component there is no edge $\{t,t'\}\in E_0$ with $t,t'\in C$. Suppose there exists a cyclic parity path $P=(s_{j_1},s_{i_1},...,s_{j_k},s_{i_k})$ in $(V,E_0\cup E_1)$, we show that this gives a contradiction. Namely, first suppose that $P$ only has odd edges. Then we have $s_{j_{l+1}}=d_{i_l,j_l} = s_{i_l}$ for $l=1,..,k-1$ and $s_{j_1} = d_{i_k,j_k} =s_{i_k}$, and thus $P = (s_{i_k},s_{i_1},s_{i_1},s_{i_2},s_{i_2},...,s_{i_{k-1}},s_{i_k})$. However, since also $i_{l+1}\not\in \{i_{l},j_l\} = \{i_l,i_{l-1}\}$, this means that  $Q = (s_{i_1},s_{i_2},....,s_{i_k},s_{i_1})$ is a cycle in $(V,E_1)$. But this is not possible since $(V,E_1)$ is a forest, which gives the contradiction. We thus assume that there is an index $l$ such that the label $m_{i_l,j_l}$ is even. By choosing the starting point of $P$ as $j_l$ instead of $j_1$, we can assume that $m_{i_1,j_1}$ is even. Now in that case we have $s_{j_2} = d_{i_1,j_1} = s_{j_1}$. We must moreover have $i_{2}\not\in \{i_1,j_1\}$ as $P$ is a parity path.
			Now as the edges $\{i_1,j_1\}$ and $\{i_2,j_2\}$ are thus distinct, and share an endpoint, we obtain that
			 $m_{i_{2},j_{2}}$ is odd. This means that $j_{3}=d_{i_2,j_2} = i_2\not=j_2$.
			Now the sub-path $ (s_{j_2},s_{i_{2}},\ldots, s_{j_k},s_{i_k},s_{j_1},s_{i_1})$ is also a parity path. Denote $j_{k+1}=j_1$ and $i_{k+1}=i_1$ and let $3<k'\leq k+1$ be the smallest index such that $s_{j_{k'}} = s_{j_2}$. Note that such $k'$ exists since $s_{j_{k+1}}=s_{j_1}=s_{j_2}$. Then the sub-path $P':=(s_{j_2},s_{i_2},...,s_{j_{k'}},s_{i_{k'}})$ is a parity path, and the labels $m_{i_l,j_l}$ for $l=2,\ldots,k'-1$ are odd since $s_{j_2}$ is the only vertex in its connected component in $(V,E_1)$ that is connected by an edge in $E_0$.
			Thus, just like in the previous case we have that $P' := (s_{i_{k'}},s_{i_2},s_{i_2},s_{i_3},\ldots,s_{i_{k'-1}},s_{i_{k'}})$. Now this means that the path $Q = (s_{i_{k'}},s_{i_2},s_{i_3},\ldots ,s_{i_{k'}})$ contains a cycle, which is a contradiction with the fact that $(V,E_1)$ is a forest. This proves the lemma.
		\end{proof}

We now state two corollaries that directly follow from Theorem \ref{theorem:parity-path-implies-not-gradient-Sp}, Theorem \ref{theorem:no-parity-path-implies-gradient-Sp} and Theorem \ref{proposition:characterization-parity-paths}.

\begin{corollary}\label{corollary:gradient-Sp-Coxeter-groups}
		Let $W=\langle S|M\rangle$ be a Coxeter system and fix $p\in [1,\infty]$. Let us denote $E_0=\{(i,j): m_{i,j}\in 2\NN\}$ and $E_1 =\{(i,j): m_{i,j} \in 2\NN+1\}$. Then the semi-group $(\Phi_t)_{t\geq 0}$ on $\calL(W)$ associated to the word length $\psi_{S}$ is gradient-$\calS_p$ if $(S,E_1)$ is a forest, and if for every connected component $C$ of $(S,E_1)$ there is at most one edge $\{t,r\}\in E_0$ with  $t\in C$ and $r\not\in C$ and no edge
		$\{t,t'\}\in E_0$ with $t,t'\in C$.
\end{corollary}

\begin{corollary}\label{corollary:no-labels-two-iff}
	Let $W=\langle S|M\rangle$ be a Coxeter system satisfying $m_{i,j} \not= 2$ for all $i,j$. Fix $p\in [1,\infty]$. Let us denote $E_0=\{(i,j): m_{i,j} \in  2\NN\}$ and $E_1 =\{(i,j): m_{i,j}  \in 2\NN+1\}$. Then the semi-group $(\Phi_t)_{t\geq 0}$ on $\calL(W)$ associated to the word length $\psi_{S}$ is gradient-$\calS_p$ if and only if $(S,E_1)$ is a forest, and for every connected component $C$ of $(S,E_1)$ there is at most one edge $\{t,r\}\in E_0$ with  $t\in C$ and $r\not\in C$ and no edge
	$\{t,t'\}\in E_0$ with $t,t'\in C$.
\end{corollary}

We would also like to point out the  following result from \cite[Example 5.1]{Brady}. It follows that  the Coxeter groups are in some cases actually equal. In such cases we have obtained the gradient-$\calS_p$ property for multiple quantum Markov semi-groups.

\begin{proposition}
	Let $W_i = \langle S_i|M_i\rangle$ be Coxeter systems for $i=1,2$ such that $\Graph_{S_i}(W)$ has no edges of even label, and such that the edges of odd label form a tree. Then if $\Graph_{S_1}(W_2)$ has the same set of labels as $\Graph_{S_2}(W_2)$ (counting multiplicities), then the Coxeter groups are equal, that is $W_1 = W_2$.
\end{proposition}

\subsection{Smallness at infinity} \label{Sect=SmallAtInfinity}
We recall the construction of a natural compactification and boundary associated with a finite rank Coxeter group. We base ourselves mostly on the very general construction from \cite{klisseTopologicalBoundariesConnected2020} but in the case of Coxeter groups this boundary was also considered in \cite{CapraceLecureux}, \cite{LamThomas}. In \cite{klisseTopologicalBoundariesConnected2020} then smallness at infinity was studied as well as its connection to the Gromov boundary, which generally is different from the construction below.

Let $W=\langle S|M\rangle$ be a finite rank Coxeter system and let $G$ be its Cayley graph which has vertex set $W$ and   $\ww, \vv \in W$ are connected by an edge if and only if $\ww =  \vv s$ for some $s \in S$. We see $G$ as a rooted graph with $e \in W$ the root. We say that $\ww \leq \vv$ if there exists a geodesic (=shortest path) from $e$ to $\vv$ passing through $\ww$.
An infinite geodesic path is a sequence $\alpha = (\alpha_i)_{i \in \mathbb{N}}$  such that: (1)  $\alpha_i \in W$, (2)  $\alpha_i$ and $\alpha_{i+1}$ have distance 1 in the Cayley graph,  (3) $(\alpha_i)_{i = 0, \ldots, n}$ is the shortest path (geodesic) from $\alpha_0$ to $\alpha_n$ for every $n$. For every $\ww \in W$ we have either $\ww \leq \alpha_i$ for all large enough $i$ or $\ww \not \leq \alpha_i$ for all large enough $i$. We write $\ww \leq  \alpha$ in the former case and $\ww \not \leq \alpha$ in the latter case. We define an equivalence relation $\sim$ by saying that for two infinite geodesics $\alpha$ and $\beta$ we have $\alpha \sim \beta$ if for all $\ww \in W$ both implications $\ww \leq \alpha$ $\Leftrightarrow$ $\ww \leq \beta$ hold. Let $\partial (W, S)$ be the set of infinite geodesics modulo $\sim$. Define $\overline{(W,S)} = W \cup \partial (W,S)$. We equip $\overline{(W,S)}$ with the topology generated by the subbase consisting of
\[
\mathcal{U}_\ww := \left\{  \alpha \in  \overline{(W,S)}  :  \ww  \leq  \alpha   \right\}, \qquad
\mathcal{U}_\ww^c := \left\{  \alpha \in  \overline{(W,S)}  :  \ww \not \leq  \alpha   \right\},
\]
with $\ww \in W$. Then $\overline{(W,S)}$ contains $W$ as an open dense subset and the left translation action of $W$ on $W$ extends to a continuous action on $\overline{(W,S)}$ (see \cite{klisseTopologicalBoundariesConnected2020}). This means that $\overline{(W,S)}$ is a compactification of $W$ in the sense of \cite[Definition 5.3.15]{brownAlgebrasFiniteDimensionalApproximations2008} and $\partial(W,S)$ is the boundary. We now recall the following definition from \cite[Definition 5.3.15]{brownAlgebrasFiniteDimensionalApproximations2008}.

\begin{definition}
We will say that a finite rank Coxeter system $(W,S)$ is {\it small at infinity} if the compactification  $\overline{(W,S)}$ is small at infinity.  This means that for every sequence $(x_i)_{i \in \mathbb{N}} \in W$ converging to a boundary point $z \in \partial(W,S)$ and for every $\ww \in W$ we have that $x_i \ww \rightarrow z$.
\end{definition}

The following is the main theorem of this subsection. The authors are indebted to Mario Klisse for noting the connections in this theorem as well as its proof.

	\begin{theorem}\label{theorem:reflections}
	Let $W=\langle S|M\rangle$ be a Coxeter system. Fix $p\in [1,\infty]$. The following are equivalent:
	\begin{enumerate}
		\item The QMS $(\Phi_t)_{t\geq 0}$ associated with the word length $\psi_{S}$ is gradient-$\calS_p$ on $\calL(W)$.
		\item For all $u,w\in S$ the set $\{\vv\in W: u\vv = \vv w\}$ is finite.
		\item For all $s\in S$ the set $\{\vv\in W:s\vv = \vv s\}$ is finite.
		\item The Coxeter  system $W=\langle S|M\rangle$ is small at infinity.
\end{enumerate}
	\end{theorem}
	\begin{proof}
		 (1) is equivalent to saying that for all $u,v \in S$ we have that $\gamma_{u,v}^{\psi_S}$ has compact support by Lemma \ref{lemma:reduction-to-set-generating-the-group2}. By Lemma \ref{lemma:shifting-identity} this is equivalent to (2).
		The equivalence between (3) and (4) was proven in \cite[Theorem 0.3]{klisseTopologicalBoundariesConnected2020}.
		The implication $(2)\implies (3)$ is immediate.


Now assume (4). We shall prove that (2) holds by contradiction.  		
So suppose that  $\#\{ \vv :  u \vv= \vv w\} = \infty$ for some $u,w\in S$.
Choose a sequence $(\vv_i)_i$ in $\{\vv :  u \vv= \vv w\}$
		which has increasing word length. By the compactness of the compactification $\overline{(W,S)}$ \cite[Proposition 2.8]{klisseTopologicalBoundariesConnected2020} this implies that (by possibly going over to a subsequence) the sequence $(\vv_i)_i$ converges to a boundary point $z$. Now, by the smallness at infinity and the assumption that $u \vv_i = \vv_i w$   we have that 		$z = \lim_i \vv_i w = \lim_i u \vv_i = u \cdot z$. 	We have either $u \leq z$ or $u \not \leq z$ but not both  in the partial order from \cite[Lemma 2.2]{klisseTopologicalBoundariesConnected2020}. Further, $u \not \leq z$ iff $u \leq u \cdot z = z$ which yields a contradiction.
	\end{proof}

\begin{remark}\label{Rmk=SmallAtInftyHyperbolic}
We refer to \cite[Theorem 0.3]{klisseTopologicalBoundariesConnected2020} for yet another statement that is equivalent to the statements in Theorem \ref{theorem:reflections}. A consequence of   \cite[Theorem 0.3]{klisseTopologicalBoundariesConnected2020} is that Coxeter groups that are small at infinity are word hyperbolic. Conversely, not every word hyperbolic Coxeter group is small at infinity.   The simplest example is probably the Coxeter group generated by $S = \{ s_1, s_2, s_3, s_4\}$ where $m_{i,j} = 2$ if $\vert i - j \vert = 1$ and $m_{i,j} = \infty$ otherwise. We thus see that not for every hyperbolic Coxeter group we have the  gradient-$\cS_p$ property for the QMS associated with the word length. However, in Section \ref{Sect=SolidTensor} we show that using tensoring we may still use our methods for such Coxeter groups.
\end{remark}

\begin{remark}\label{Rmk=StrongSolidityReturns}
It is known that every discrete hyperbolic group is strongly solid by combining results in \cite{HigsonGuentner} (to get AO$^+$ using amenable actions on the Gromov boundary), \cite{OzawaHyperbolic} (for weak amenability, see \cite{fendlerWeakAmenabilityCoxeter2002}, \cite{Janus} for general Coxeter groups) and \cite{PopaVaesCrelle} (for Theorem \ref{Thm=AOimpliesStronglySolid}). Condition AO$^+$ may also be obtained by Theorem \ref{Thm:AO} for the Coxeter groups that admit a QMS with gradient-$\calS_p$. However, Remark \ref{Rmk=SmallAtInftyHyperbolic} shows that this covers a smaller class than   \cite{HigsonGuentner} and so our methods -- for now at least -- do not improve on existing methods concerning strong solidity questions.

There are still two large benefits of the results in this section. Firstly, given a Coxeter system  $W=\langle S|M\rangle$ it is not directly clear whether it is small at infinity. A combination of Theorem \ref{theorem:reflections} and Corollaries \ref{corollary:gradient-Sp-Coxeter-groups} and \ref{corollary:no-labels-two-iff} gives in many cases an easy way to see whether a Coxeter group is small at infinity. Secondly, for now we may not improve on current strong solidity results but in Section \ref{Sect=SolidTensor} we show that using the tensor methods of Section  \ref{Sect=Coefficients} we may prove strong solidity for all hyperbolic right-angled Coxeter groups. This gives an alternative path to the method of \cite{HigsonGuentner} (still not outweighing known results). In Section \ref{Sect=HeckeStronglySolid} this alternative path also gives strong solidity results for Hecke von Neumann algebras. Here we really improve on existing results as the methods of \cite{HigsonGuentner} can only be applied in a limited way, see \cite[Theorem 3.15 and  Corollary 3.17]{klisseTopologicalBoundariesConnected2020}.
\end{remark}

\begin{remark}
 By Theorem \ref{theorem:reflections} (see \cite[Theorem 0.3]{klisseTopologicalBoundariesConnected2020}) smallness at infinity or gradient-$\cS_p$ can be characterized in terms of the finiteness of the centralizers of the generators. Such centralizers can be analyzed using the methods from \cite{Allcock}, \cite{Brink}.
\end{remark}

\section{Gradient-$\mathcal{S}_p$ semi-groups associated to weighted word lengths on Coxeter groups}\label{Sect=Weights}
\label{section:semi-groups-weighted-word-length}
In this section we will consider proper length functions  on Coxeter groups  that are conditionally of negative type and are different from the standard word length. We can then consider the quantum Markov semi-groups associated to these other functions, and study the gradient-$\calS_p$ property of these semi-groups. We show that  these other semi-groups may have the gradient-$\calS_p$ properties in cases where the semi-group associated to the word length $\psi_S$ fails to be gradient-$\calS_p$. For $p\in [1,\infty]$ this gives us new examples of Coxeter groups $W$ for which there exist a gradient-$\calS_p$ quantum Markov semi-group on $\calL(W)$.
These results will turn out to be crucial in Section \ref{Sect=SolidTensor}.

 \subsection{Weighted word lengths}\label{subsection:weighted-word-length-conditionally-negative}
For   non-negative weights $\xx=(x_1,...x_{|S|})$ we consider, if existent, the function $\psi
_{\xx}: W \rightarrow \mathbb{R}$  by taking the word length with respect to the weights $\xx$ on the generators (see below). These functions are conditionally of negative definite type as follows for instance as a special case of \cite[Theorem 1.1]{BozejkoSpeicherMathAnn}. Here we give another purely group theoretical proof.

Fix again a (finite rank) Coxeter group $W = \langle S | M \rangle$. Recall that the graph  $\Graph_S(W)$ was defined in Definition \ref{Dfn=Graph}.
Let $\Graph_S'(W)$ be the subgraph of $\Graph_S(W)$ that has vertex set $S$ and edge set  $E=\{(s_i,s_j): 3\leq m_{i,j} := m_{s_i,s_j} <\infty\}$. Then let  $\mathcal{C}_i$ be the connected component in $\Graph_S'(W)$ that contains $s_i$.

\begin{lemma}\label{lemma:weighted-word-lenght-conditionally-negative}
	Let $W=\langle S|M\rangle$ be a Coxeter group. Then if $\xx\in [0,\infty)^{|S|}$ is such that $x_i = x_j$ whenever $\mathcal{C}_i=\mathcal{C}_j$, then the function
\[
\psi_{\xx}:W\to [0,\infty),
\]
 given for a word $\ww = w_1....w_k$ in reduced expression by $\psi_{\xx}(\ww) = \sum_{i=1}^{|S|}x_i|\{l:w_l=s_i\}|$ is well-defined and is conditionally of negative type.
	\begin{proof}
		Let $\nn=(n_1,...,n_{|S|})\in \NN^{|S|}$ be such that $n_i=n_j$ whenever $\mathcal{C}_i=\mathcal{C}_j$. We will construct a new Coxeter group $\widetilde{W}_{\nn} = \langle S_{\nn}| M_{\nn}\rangle$ as follows. We denote
		$S_{\nn} = \{s_{i,k}: 1\leq i\leq |S|, 1\leq k\leq n_i\}$ for the set of letters. We then define $M_{\nn}:S_{\nn}\to \NN\cup\{\infty\}$ as:\\
		\[
m_{\nn,  s_{i,k},s_{j,l} } = \begin{cases}
		m_{  s_i,s_j } & \mathcal{C}_i = \mathcal{C}_j \text{ and }k=l\\
		2 & \mathcal{C}_i=\mathcal{C}_j \text{ and } k\not=l\\
		m_{s_i,s_j} & \mathcal{C}_i\not=\mathcal{C}_j
		\end{cases}.
\]
		We set the Coxeter group $\widetilde{W}_{\nn} = \langle S_{\nn}| M_{\nn}\rangle$.   We now define a homomorphism $\varphi_{\nn}:W\to \widetilde{W}_{\nn}$ given for generators by $\varphi_{\nn}(s_i) = s_{i,1}s_{i,2}....s_{i,n_i}$. We note that
		$\varphi_{\nn}(s_i)^2 = s_{i,1}...s_{i,n_i}s_{i,1}...s_{i,n_i} = s_{i,1}^2...s_{i,n_i}^2 = e$.
		Furthermore, when $\mathcal{C}_i=\mathcal{C}_j$ we have that $n_i=n_j$ and
		$(\varphi_{\nn}(s_i)\varphi_{\nn}(s_j))^{m} = (s_{i,1}....s_{i,n_i}s_{j,1}....s_{j,n_j})^{m} = (s_{i,1}s_{j,1})^m(s_{i,2},s_{j,2})^m....(s_{i,n_i}s_{j,n_j})^m$. This means that in this case $(\varphi_{\nn}(s_i)\varphi_{\nn}(s_j))^{m_{s_{i},s_j} }=e$. If $\mathcal{C}_i\not=\mathcal{C}_j$ then either $m_{s_i,s_j} =2$ or $m_{s_i,s_j} =\infty$. If $m_{s_i,s_j} =2$ then also $\varphi_{\nn}(s_i)\varphi_{\nn}(s_j) = s_{i,1}...s_{i,n_i}s_{j,1}...s_{j,n_j} = s_{j,1}...s_{j,n_j}s_{i,1}...s_{i,n_i} = \varphi_{\nn}(s_i)\varphi_{\nn}(s_j)$ holds. Therefore, we can extend $\varphi_{\nn}$ to words $\ww=w_1....w_k\in W$ by defining $\varphi_{\nn}(\ww)=\varphi_{\nn}(w_1)...\varphi_{\nn}(w_k)$. By what we just showed, this map is well-defined. Furthermore, from the definition it follows that this map is a homomorphism. Moreover, we note that if $\ww=w_1...w_k\in W$ is a reduced expression, then $\varphi_{\nn}(\ww)= \varphi_{\nn}(w_1)...\varphi_{\nn}(w_k)$ is also a reduced expression. This means in particular that $\varphi_{\nn}$ is injective. Furthermore, if we denote $\widetilde{\psi}_{\nn}$ for the word length on $\widetilde{W}_{\nn}$, then we have that for a word $\ww=w_1....w_k\in W$ written in a reduced expression that
		\[
		\widetilde{\psi}_{\nn}\circ \varphi_{\nn}(\ww)  = \sum_{l=1}^k\widetilde{\psi}_{\nn}(\varphi_{\nn}(w_l))
		 =\sum_{i=1}^{|S|}\widetilde{\psi}_{\nn}(\varphi_{\nn}(s_i))|\{l:w_l=s_i\}|
		 =\sum_{i=1}^{|S|}n_i|\{l:w_l=s_i\}|.
		\]

	Now fix $\xx\in [0,\infty)^{|S|}$ with $x_i=x_j$ whenever $\mathcal{C}_i=\mathcal{C}_j$. For $m\in\NN$ define $\nn_m\in \NN^{|S|}$ by $(\nn_m)_i = \ceil{m x_i}+1\in\NN$.
Now, for $\ww\in W$ with reduced expression $\ww =w_1...w_k$ we have
\[
\begin{split}
 & \left|\frac{1}{m}\widetilde{\psi}_{\nn_m}\circ \varphi_{\nn_m}(\ww)- \sum_{i=1}^{|S|}x_i|\{l:w_l=s_i\}|\right|
 \leq \sum_{i=1}^{|S|}|\frac{(\nn_m)_i}{m}-x_i|\cdot|\{l:w_l=s_i\}| \\
= & \sum_{i=1}^{|S|}\frac{|\ceil{mx_i} +1 -mx_i|}{m}|\{l:w_l=s_i\}|
 \leq \sum_{i=1}^{|S|}\frac{2}{m}|\{l:w_l=s_i\}|
 \leq \frac{2|\ww|}{m},
\end{split}
\]
 and hence $\frac{1}{m}\widetilde{\psi}_{\nn_m}\circ \varphi_{\nn_m}(\ww)\to \sum_{i=1}^{|S|}x_i|\{l:w_l=s_i\}|$ as $m\to \infty$. This shows in particular that $\psi_{\xx}$ is well defined. Now, since $\frac{1}{m}\widetilde{\psi}_{\nn_m}\circ \varphi_{\nn_m}\to \psi_{\xx}$ point-wise. Since $\frac{1}{m}\widetilde{\psi}_{\nn_m}\circ \varphi_{\nn_m}$ is conditionally of negative type we have by \cite[Proposition C.2.4(ii)]{bekkaKazhdanProperty2008} that $\psi_{\xx}$ is conditionally of negative type.
	\end{proof}
\end{lemma}


\begin{remark}\label{remark:weights-right-angled-Coxeter-group}
	By Lemma \ref{lemma:weighted-word-lenght-conditionally-negative} in the case of a right-angled Coxeter group $W=\langle S|M\rangle$ we have that every weight $\xx \in [0,\infty)^{|S|}$ defines a function that is  conditionally  of negative type.
\end{remark}
	
\begin{remark}For a general Coxeter group $W=\langle S|M\rangle$ and arbitrary non-negative weights $\xx\in [0,\infty)^{|S|}$ the weighted word length is not well-defined. Indeed, if $s_i,s_j\in S$ are such that $m_{s_i,s_j}$ is odd, then for $k_{i,j}:= \floor{\frac{1}{2}  m_{s_i,s_j} }$ we have that
$(s_is_j)^{k_{i,j}}s_i$ and $s_j(s_is_j)^{k_{i,j}}$ are two reduced expressions for the same word, but the values of $|\{l: w_l=s_i\}|$ and $|\{l:w_l=s_j\}|$ depend on the choice of the reduced expressions.
\end{remark}
We shall now turn to examine when a weighted word length is proper.
Fix again a Coxeter system $W=\langle S|M\rangle$. Let $\mathcal{I} \subseteq S$ be a subset of the generators such that for $i=1,\ldots,|S|$ either $\mathcal{C}_i\subseteq \mathcal{I}$ or $\mathcal{C}_i\cap \mathcal{I}=\emptyset$. We set
\[
\psi_{\mathcal{I}} = \psi_{\xx}, \qquad \textrm{ with } \xx\in [0,\infty)^{|S|} \textrm{ defined by } \xx(i)=\chi_{\mathcal{I}}(i),
\]
where $\chi_{\mathcal{I}}$ is the indicator function on $\mathcal{I}$. Then by Lemma \ref{lemma:weighted-word-lenght-conditionally-negative} we have that $\psi_{\mathcal{I}}: W \rightarrow \mathbb{R}$  is a well-defined function that is conditionally of negative type.
 We give the following characterization on when the function $\psi_{\calI}$ is  proper.

\begin{proposition}\label{Prop=ProperWeight}
	The function $\psi_{\mathcal{I}}$ is proper if and only if the elements $S\setminus \mathcal{I}$ generate a finite subgroup.
	\begin{proof}
		Indeed, if the generated group $H$ is infinite, then $\psi_{\calI}$ is not proper as $\psi_{\calI}|_{H} =0$. On the other hand, if the generated group $H$ contains $N<\infty$ elements, then for a reduced expression $\ww = w_1....w_k\in W$ we can not have that $w_l,w_{l+1},...w_{l+N}\in S\setminus \mathcal{I}$ for some $1\leq l\leq k-N$ as the expressions $w_l, w_lw_{l+1},w_lw_{l+1}w_{l+2},..$ would all be distinct elements in $H$.  This thus implies that $\psi_{\mathcal{I}}(\ww)> \frac{|\ww|}{N+1}-1$ which shows that $\psi_{\mathcal{I}}$ is proper in this case.
	\end{proof}
\end{proposition}

\subsection{Gradient-$\mathcal{S}_p$ property with respect to weighted word lengths on right-angled Coxeter groups}
\label{subsection:weighted-word-length-gradient-Sp}

In this subsection we shall consider a right-angled Coxeter group $W=\langle S|M\rangle$. Recall that right-angled means that $m_{s,t} \in \{ 2, \infty\}$ for all $s,t \in S, s \not = t$ so either $s,t$ are free or they commute.  By Remark \ref{remark:weights-right-angled-Coxeter-group} it follows that for any $\xx\in [0,\infty)^{|S|}$ we have that $\psi_{\xx}: W \rightarrow \mathbb{R}$ is well-defined and conditionally of negative definite type.   We note also that   $\psi_{\xx}(\ww)=\psi_{\xx}(w_1) + .. + \psi_{\xx}(w_k)$ when $\ww = w_1...w_k$ is a reduced expression. Therefore by Lemma \ref{lemma:shifting-identity} we have that $\gamma_{u,w}^{\psi_{\xx}}(\vv) \not=0$ for $u,w\in S$ and $\vv\in W$ if and only if $u\vv =\vv w$ and $\psi_{\xx}(u)>0$.

\begin{theorem}\label{lemma:gradient-Sp-weighted-word-length-right-angled-Coxeter-group}
	Let $W=\langle S|M\rangle$ be a right-angled Coxeter group. Let $\xx\in [0,\infty)^{|S|}$ and $p\in [1,\infty]$. Suppose the function $\psi_{\xx}$ is proper. Then, the semi-group $(\Phi_t)_{t\geq 0}$ induced by $\psi_{\xx}$ is gradient-$\mathcal{S}_p$ if and only if there do not exist generators $r,s,t\in S$ with $m_{r,s} = m_{r,t} =2, m_{s,t} = \infty$ and $\psi_{\xx}(r)>0$.
	\begin{proof}
		Suppose that $(\Phi_t)_{t\geq 0}$ is not gradient-$\mathcal{S}_p$ for some $p\in[1,\infty]$. We will show the generators with the given properties exist. Namely, there are generators $u,w$ for which $\gamma_{u,w}^{\psi_{\xx}}$ is not finite rank.  We can thus let $\mathbf{v}\in W$ with $|\vv|>|S|+1$ be such that $\gamma_{u,w}^{\psi_{\xx}}(\vv)\not=0$.
		Then $u\vv = \vv w$ and $\psi_{\xx}(u),\psi_{\xx}(w)>0$ by Lemma \ref{lemma:shifting-identity}. We note moreover that  by \cite[Lemma 3.3.3]{davisGeometryTopologyCoxeter2008} we have that $u=w$ because these elements are conjugate and the Coxeter group is right-angled.
		We can now let $\zz \in \{\vv,u\vv,\vv w,u\vv w\}$ be such that $|\zz|\leq |u\zz|,|\zz w|$. Then the equality $u\mathbf{z}=\mathbf{z}w$ also holds. Therefore, we can write $\zz$ in reduced form $\zz = \mathbf{r}_{i_1,j_1}....\mathbf{r}_{i_k,j_k}$ with the conditions as in Lemma \ref{lemma:word-sequence-expression}.
		Now, as $m_{ i_l, j_l } := m_{ s_{i_l},s_{j_l} }    <\infty$ we must have $ m_{ s_{i_l},s_{j_l} }=2$ for $l=1,...,k$. Hence $\zz = s_{i_1}s_{i_2}...s_{i_k}$. Furthermore $s_{j_{l+1}}=s_{j_l}$ for $l=1,..,k-1$ since $m_{ s_{i_l},s_{j_l} }$ is even. We define $r= s_{j_1}$. Then $r = c_{i_1,j_1} = u$ so that $\psi_{\xx}(r)>0$. Furthermore, since $k=|\zz|\geq |\vv|-1>|S|$ there exist indices $l<l'$ such that $m_{s_{i_l},s_{i_{l'}}} = \infty$. We then set $s = s_{i_l}$ and $t = s_{i_{l'}}$.
		Then $m_{s,r } =m_{s_{i_l},s_{j_l}}=2$ and likewise $m_{t,r} =2$. This shows that all stated properties hold for $r,s,t$.
		
		For the other direction, suppose that there exist $r,s,t\in S$ with $m_{r,s} = m_{r,t} =2$ and $m_{s,t} = \infty$ and $\psi_{\xx}(r)>0$. Define the words $\vv_n = (st)^n$. Then we have $|\vv_n|=2n$ and hence $\{\vv_n\}_{n\geq 1}$ are all distinct. Moreover, we have $r\vv_n = \vv_nr$ and $\psi_{\xx}(r)>0$. This means by Lemma \ref{lemma:shifting-identity} that $\gamma_{r,r}^{\psi_{\xx}}(\vv_n)=\psi_{\xx}(r)>0$ for $n\geq 1$.
		Thus the semi-group $(\Phi_t)_{t\geq 0}$ is not gradient-$\mathcal{S}_p$.
	\end{proof}
\end{theorem}

\section{Strong solidity for hyperbolic right-angled Coxeter groups}\label{Sect=SolidTensor}

We conclude this paper with two applications  that combines all the techniques that we have developed so far. This section contains the first application.
 We prove that any right-angled hyperbolic Coxeter group has a strongly solid group von Neumann algebra. This result was surely known before; it follows for instance from \cite{PopaVaesCrelle}. Nevertheless we present our alternative proof to demonstrate the techniques that we have established in this paper. For the rest of this section fix a (finite rank) {\it right-angled} Coxeter system $W = \langle S | M \rangle$. We shall use the following characterisation of word hyperbolicity.

\begin{theorem}[See \cite{davisGeometryTopologyCoxeter2008}]\label{Thm=HyperbolicRightAngled}
Let $W = \langle S | M \rangle$ be a right-angled Coxeter system.  The following are equivalent:
\begin{enumerate}
\item $W = \langle S | M \rangle$ is word hyperbolic.
\item\label{Item=Hypberolic2} There do not exist four distinct elements $s,t,u,v \in S$ such that $m_{s,t} = m_{u,v} = \infty$ and $m_{s, u} = m_{s,v} = m_{t,u} = m_{t,v} = 2$.
\end{enumerate}
\end{theorem}

Our aim is to prove the following. The proof is based on Proposition \ref{Prop=RACGquasi} and Lemma \ref{lemma:product-finite-support} which we prove at the end.

\begin{theorem}\label{Thm=StrongSolidRightAngled}
Let $W = \langle S | M \rangle$ be a word hyperbolic right-angled Coxeter group. Then $\calL(W)$ satisfies AO$^+$ and is strongly solid.
\end{theorem}

\begin{proof}
Let $\Cliq(S | M)$ be the set of subsets  $\mathcal{I} \subseteq S$ that generate a finite Coxeter subgroup $W_{\mathcal{I}}$ of $W = \langle S | M \rangle$.  These are precisely the subsets  $\mathcal{I} \subseteq S$ such that for every $s,t \in \mathcal{I}$ we have $m_{s,t} = 2$. We call $\Cliq(S | M)$ the set of {\it cliques}. They could also be seen as the cliques in a natural graph that is associated with the graph product decomposition of a right-angled Coxeter group, where a clique is defined as a complete subgraph, see \cite{GreenThesis}, \cite{caspersGraphProductsOperator2017}. We shall not need this graph product decomposition here except for the fact that it explains the terminology.

For $\mathcal{I} \in \Cliq(S | M)$ the function $\psi_{S \backslash \mathcal{I}}$ is proper (see Proposition \ref{Prop=ProperWeight}) and conditionally of negative  type (see Lemma \ref{lemma:weighted-word-lenght-conditionally-negative}). We may therefore consider the QMS $\Phi_{\mathcal{I}}$ associated with $\psi_{S \backslash \mathcal{I}}$, the associated gradient $\mathbb{C}[W]$ bimodule $H_{\mathcal{I}}$ and the Riesz transform $R_\mathcal{I}: \ell_2(\Gamma) \rightarrow H_{\mathcal{I}}$. The Riesz transform $R_{\mathcal{I}}$ is then a partial isometry with a finite dimensional kernel spanned by $\delta_u, u \in W_{\mathcal{I}}$. $R_{\mathcal{I}}$ is almost bimodular by Corollary \ref{Cor=RieszBimodular}. We now consider the $\otimes_\Gamma$ tensor product of bimodules with $\Gamma = W$   over all $\mathcal{I} \in \Cliq(S | M)$ as was defined in Section \ref{Sect=TensoringBimodules},
\begin{equation}\label{Eqn=ConvolveBimodule}
H_W =   \bigotimes_{\mathcal{I} \in \Cliq(S | M), \Gamma} H_{\mathcal{I}}.
\end{equation}
We note that the order in which the tensor products are taken is not relevant for our analysis.
Consider the convolution product of Riesz transforms
\[
R_W = \ast_{\mathcal{I} \in \Cliq(S | M)}  R_{\mathcal{I}}: \ell_2(\Gamma) \rightarrow H_W.
\]
By Lemma \ref{Lem:almostbimod} and Lemma \ref{Lem=PartialIsoConvolution} we see that $R_W$ is an almost bimodular partial isometry whose kernel is spanned by all vectors $\delta_u$ where $u \in W_{\mathcal{I}}$ for some $\mathcal{I} \in \Cliq(S | M)$. In particular the kernel of $R_W$ is finite dimensional. Let $L \subseteq H_W$ be the smallest $\mathbb{C}[W]$ subbimodule containing the image of $R_W$. Then $R_W: \ell_2(W) \rightarrow L$ is still an almost bimodular partial isometry with finite dimensional kernel.

Recall that $C_r^\ast(W)$ is locally reflexive and $\calL(W)$ has the weak-$\ast$ completely bounded approximation property as $W$ is weakly amenable (see \cite{fendlerWeakAmenabilityCoxeter2002}, \cite{Janus}). It then follows from Theorem \ref{Thm=RieszImpliesAO} that if $L$ is quasi-contained in the coarse bimodule of $W$ then $\calL(W)$ satisfies AO$^+$. Consequently, $\calL(W)$ is strongly solid by Theorem \ref{Thm=AOimpliesStronglySolid}. The proof that  $H_W$ is quasi-contained in the coarse bimodule of $W$ is given in Proposition \ref{Prop=RACGquasi} below.
\end{proof}

\begin{proposition}\label{Prop=RACGquasi}
The $\mathbb{C}[W]$ bimodule $L$ defined in the proof of Theorem \ref{Thm=StrongSolidRightAngled}  is quasi-contained in the coarse bimodule of the word hyperbolic right-angled  Coxeter group $W$.
\end{proposition}
\begin{proof}
We shall prove that a cyclic set of coefficients is in $\cS_2$ so that the proposition follows from  Lemma \ref{lemma:reduction-to-cyclic-subset}.
Let us denote $H_{00}\subseteq L$ for the sets of all the vectors
\[
\xi_{\vv}:=  (  \ast_{\mathcal{I} \in \Cliq(S | M)}   R_{\mathcal{I}})  (\delta_{\vv}) =   \bigotimes_{  \mathcal{I} \in \Cliq(S | M)  }  \lambda_{\vv} \otimes_{\nabla_{\mathcal{I}}}  \delta_e, \qquad \vv\in W.
\]
Here we used the symbol $\otimes_{\nabla_{\mathcal{I}}}$ to denote elements in the gradient bimodule constructed from $\mathcal{I}$.  By construction of $L$ we have that $H_{00}$ is cyclic for $L$.
For $\xi_{\uu},\xi_{\ww}\in H_{00}$ we now inspect the coefficient $T_{\xi_{\uu},\xi_{\ww}}$. We have for $\uu, \vv, \ww \in W, y \in \mathbb{C}[W]$,
\[
\begin{split}
\tau(T_{\xi_{\ww},\xi_{\uu}}(\lambda_\vv)y) &=
\langle \lambda_{\vv}\cdot \xi_{\ww} \cdot y, \xi_{\uu}\rangle \\
&= \prod_{\calI  \in \Cliq(S | M)   }\langle \lambda_{\vv}\cdot (\lambda_{\ww}\otimes_{\calI} \delta_e)\cdot y, \lambda_{\uu}\otimes_{\calI} \delta_e\rangle_{\calI}\\
&= \prod_{\calI \in \Cliq(S | M)  }\langle \Psi^{\lambda_{\uu^{-1}},\lambda_{\ww}}_{\mathcal{I}}(\lambda_{\vv})\delta_e y,\delta_e\rangle\\
&= \prod_{\calI \in \Cliq(S | M)  }\gamma_{\uu^{-1},\ww}^{\psi_{S\setminus \calI}}(\vv)\langle \lambda_{\uu^{-1}\vv\ww}\delta_e y,\delta_e\rangle.
\end{split}
\]
Define the function
\begin{align}\label{eq:gamma-product}
\widetilde{\gamma}_{\uu,\ww}(\vv) = \prod_{\calI\in \Cliq(S | M) )}\gamma_{\uu,\ww}^{\psi_{S\setminus \calI}}(\vv).
\end{align}
 Then, if $\widetilde{\gamma}_{\uu^{-1},\ww}(\vv) = 0$ we have that $\tau(T_{\xi_{\ww},\xi_{\uu}}(\lambda_{\vv})y) = 0$ for all $y\in \CC[W]$ and consequently  $T_{\xi_{\ww},\xi_{\uu}}(\lambda_{\vv})=0$.  We thus have that $T_{\xi_{\ww},\xi_{\uu}}$ is finite rank whenever $\widetilde{\gamma}_{\uu^{-1},\ww}$ has finite support. In Lemma \ref{lemma:product-finite-support} we shall show that the function $\widetilde{\gamma}_{\uu,\ww}$ has  finite rank for all $\uu,\ww\in W$ so that we conclude the proof.
 \end{proof}

In order to prove Lemma \ref{lemma:product-finite-support} rigorously we shall introduce some notation here. A tuple
$(w_1,\ldots,w_k)$ with $w_i\in S$ will be call \textit{reduced} if the expression $w_1 \ldots w_k$ is reduced. Furthermore, we will call the tuple \textit{semi-reduced} whenever $|w_1 \ldots w_k| + |\{l: w_l=e\}|  = k$. We will say that a pair $(i,j)$ with $i<j$ \textit{collapses} for a tuple $(w_1,\ldots,w_k)$ whenever $w_i = w_j\not=e$ and the elements $\{w_l: i\leq l\leq j\}$ pair-wise commute. In that case we will call the tuple $(w_1,\ldots,w_{i-1},e,w_{i+1},\ldots,w_{j-1},e,w_{j+1},\ldots,w_k)$ the \textit{tuple obtained from $(w_1,\ldots ,w_k)$ by collapsing on the pair $(i,j)$}. We note that the word $w_1 \ldots w_k$ corresponding to $(w_1,\ldots ,w_k)$ is in fact the same as the word $w_1 \ldots w_{i-1}ew_{i+1} \ldots w_{j-1}ew_{j+1} \ldots w_k$ corresponding to the collapsed tuple. The notation that we introduced here is convenient because it keeps indices aligned correctly.
We also note that a tuple $(w_1,\ldots,w_k)$ is semi-reduced if and only if we cannot collapse on any pair $(i,j)$. Hence, for a general tuple we can obtain a semi-reduced tuple by subsequently collapsing on pairs $(i_1,j_1),\ldots,(i_q,j_q)$.

\begin{lemma}\label{lemma:product-finite-support}
	For a right-angled word hyperbolic Coxeter group $W$ we have that for $\uu,\ww\in W$ the function $\widetilde{\gamma}_{\uu,\ww}: W\to \RR$ defined in  \eqref{eq:gamma-product}   has finite support.
\end{lemma}
	\begin{proof}
Let $\uu=u_1 \ldots u_{n_1},\vv=v_1 \ldots v_{n_2}, \ww=w_1 \ldots w_{n_3}\in W$ written in reduced expression.
 We will moreover assume that $|\vv|>|\uu| +|\ww| + |S| +2$. We will show that for such words we have  $\widetilde{\gamma}_{\uu,\ww}(\vv)=0$. This then shows that $\widetilde{\gamma}_{\uu,\ww}$ has finite support.\\

Let $(u_1', \ldots ,u_{n_1}',v_1', \ldots ,v_{n_2}')$ be the semi-reduced tuple obtained by subsequently collapsing the tuple $(u_1, \ldots ,u_{n_1},v_1, \ldots ,v_{n_2})$ on pairs $(i_1',j_1'),\ldots ,(i_{q_1}',j_{q_1}')$. Then we must have $i_l'\leq n_1$ and $j_l'>n_1$ since the expressions for $\uu$ and $\vv$ were reduced. Also $|\uu\vv| =|\uu| + |\vv| - 2q_1$ and more generally for a weight $\xx\in [0,\infty)^{|S|}$ we have
$$\psi_{\xx}(\uu\vv) = \psi_{\xx}(\uu) + \psi_{\xx}(\vv) - 2\sum_{l=1}^{q_1}\psi_{\xx}(u_{i_l'}).$$

Likewise let $(v_1'', \ldots ,v_{n_2}'',w_1'', \ldots ,w_{n_3}'')$ be the semi-reduced tuple obtained by subsequently collapsing the tuple $(v_1, \ldots ,v_{n_2},w_1, \ldots ,w_{n_3})$ on pairs $(i_1'',j_1''),\ldots ,(i_{q_2}'',j_{q_2}'')$. Then we must have $i_l''\leq n_2$ and $j_l''>n_2$ since the expressions for $\vv$ and $\ww$ were reduced. Also $|\vv\ww| =|\vv| + |\ww| - 2q_2$ and more generally for a weight $\xx\in [0,\infty)^{|S|}$ we have
$$\psi_{\xx}(\vv\ww) =\psi_{\xx}(\vv) + \psi_{\xx}(\ww) -  2\sum_{l=1}^{q_2}\psi_{\xx}(w_{j_l''-n_2}).$$

Let us denote
\[
\calJ = \{v_j: j\in \{1,\ldots,n_2\}\setminus (\{j_1'-n_1,\ldots,j_{q_1}'-n_1\}\cup\{i_1'',\ldots,i_{q_2}''\})\}.
\]
 Now since $n_2=|\vv|>|\uu|+|\ww|+|S|+2\geq q_1+q_2+|S|+2$ we have that $|\mathcal{J}|\geq |S|+2$. Hence, there are two elements $g_1,g_2\in \mathcal{J}$ that do not mutually commute. Now, if $s_1,s_2\in S$ commute with all elements in $\mathcal{J}$ , then $s_1,s_2$ commute with both $g_1$ and $g_2$ so that by the word hyperbolicity of $W$ (see \eqref{Item=Hypberolic2} of  Theorem \ref{Thm=HyperbolicRightAngled}) we must have that also $s_1$ commutes with $s_2$. We now let $\calI_0\subseteq S$ be the set of all generators that commute with all elements in $\mathcal{J}$. Then by what we just mentioned we have that the elements in $\calI_0$ pair-wise commute, i.e. $\calI_0\in   \Cliq(S | M)$.\\

Now, for $i=1,\ldots,n_1$ let us set $\widetilde{u_i} = u_i'$ and for $i=1,\ldots,n_3$ set $\widetilde{w_i} = w_i''$. Furthermore, for $i=1, \ldots ,n_2$ set $\widetilde{v_i}=e$ whenever either $v_i'=e$ or $v_i''=e$ but not both, and set $\widetilde{v_i}=v_i$ otherwise. Let us also denote $\widetilde{\uu} = \widetilde{u_1} \ldots \widetilde{u_{n_2}}$, $\widetilde{\vv}=\widetilde{v_1}...\widetilde{v_{n_2}}$ and $\widetilde{\ww} = \widetilde{w_1} \ldots \widetilde{w_{n_3}}$.

 We claim that then we have
$\widetilde{\uu}\widetilde{\vv}\widetilde{\ww} = \uu\vv\ww$.
Indeed, we have that $\uu\vv\ww = \uu v_1''\ldots v_{n_2}''w_1'' \ldots w_{n_3}''$. Now we can collapse
$(u_1,\ldots,u_{n_1},v_{1}'',\ldots,v_{n_2}'',w_{1}'',\ldots,w_{n_3}'')$ subsequently on the pairs $(i_l',j_l')$ for $l=1,\ldots,q_1$ except when $v_{j_l'-n_1}'' \not= v_{j_l'-n_1}$ for some $1\leq l\leq q_1$, in which case $v_{j_l'-n_1}''=e$.
If this is the case then $j_l'-n_1 = i_{k_l}''$ for some $k_l\in \{1,\ldots,q_2\}$. In particular it follows that in this case $u_{i_l'}=v_{j_l'-n_1}=v_{i_{k_l}''}=w_{j_{k_l}''-n_2}$ and that this element commutes with all elements in $\calJ$. Therefore $u_{i_l'}\in \calI_0$. We can then simply interchange the elements at index $i_l'$ (which is $u_{i_l'}$) and the element at index $j_{l}'$ (which is $v_{j_{l}'-n_1}''=e$). This manipulation does not change the word, and allows us to continue collapsing on the remaining pairs. Once we are done collapsing on all pairs we have obtained the tuple $(\widetilde{u}_1,\ldots,\widetilde{u_{n_1}},\widetilde{v_1},\ldots, \widetilde{v_{n_2}},\widetilde{w_1},...,\widetilde{w_{n_3}})$. This thus shows us that $\uu\vv\ww = \widetilde{\uu}\widetilde{\vv}\widetilde{\ww}$. It also shows us that $\widetilde{v_{j_l'-n_1}}\in \{e\}\cup \calI_0$ for $l=1,\ldots,q_2$. Note that also by definition $\widetilde{u_{i_l'}}=e$ for $l=1,\ldots,q_1$ and $\widetilde{w_{j_l''-n_2}}=e$ for $l=1,\ldots,q_2$. Therefore we also have that $\psi_{S\setminus \calI_0}(\widetilde{u_{i_l'}})=
\psi_{S\setminus \calI_0}(e)=0$ for $l=1,\ldots,q_1$ and likewise $\psi_{S\setminus \calI_0}(\widetilde{w_{j_l''-n_2}}) =0$ for $l=1,\ldots,q_2$. Furthermore $\psi_{S\setminus \calI_0}(\widetilde{v_{j_l'-n_1}}) = 0$ for $l=1,\ldots,q_1$ and $\psi_{S\setminus \calI_0}(\widetilde{v_{i_l''}}) = 0$ for $l=1,\ldots,q_2$.

Also, if we can collapse $(\widetilde{u_1},\ldots, \widetilde{u_{n_1}},\widetilde{v_1},\ldots, \widetilde{v_{n_2}},\widetilde{w_1},\ldots,\widetilde{w_{n_3}})$ on some pair $(i,j)$ then we must have $i\leq n_1$ and $j>n_1+n_2$. Indeed otherwise we have that either $(u_1',\ldots,u_{n_1}',v_1',\ldots,v_{n_2}')$ or $(v_1'',\ldots,v_{n_2}'',w_1',\ldots,w_{n_3}'')$ is not semi-reduced, which is a contradiction. Now let $(i_1,j_1),\ldots,(i_q,j_q)$ be the pairs on which we can subsequently collapse $(\widetilde{u_1},\ldots, \widetilde{u_{n_1}},\widetilde{v_1},...\widetilde{v_{n_2}},\widetilde{w_1},\ldots,\widetilde{w_{n_3}})$ to obtain a semi-reduced tuple. Then we thus must have $i_l\leq n_1$ and $j_l>n_1+n_2$. This thus implies that for $l=1,\ldots,q$ we have that $\widetilde{u_{i_l}} = \widetilde{w_{j_l}}$ commutes with the elements from $\mathcal{J}$. Therefore we have $\{\widetilde{u_{i_l}}: l=1,\ldots,q\} = \{   \widetilde{w_{j_l   } } :l=1,\ldots,q\}\subseteq \calI_0$.

Now, we have that
\[
\begin{split}
\psi_{S\setminus \calI_0}(\uu\vv\ww) &=	\psi_{S\setminus \calI_0}(\uu) +	\psi_{S\setminus \calI_0}(\vv)
+\psi_{S\setminus \calI_0}(\ww) \\ &-2\left[\sum_{l=1}^{q_1}\psi_{S\setminus \calI_0}(u_{i_l'})
+ \sum_{l=1}^{q_2}\psi_{S\setminus \calI_0}(w_{i_l''-n_2})
+ \sum_{l=1}^{q} \psi_{S\setminus\calI_0}(\widetilde{u_{i_l}})\right]\\
&= \psi_{S\setminus \calI_0}(\uu\vv) + \psi_{S\setminus \calI_0}(\vv\ww) - \psi_{S\setminus \calI_0}(\vv) 	+ 2\sum_{l=1}^{q} \psi_{S\setminus\calI_0}(\widetilde{u_{i_l}})\\
&= \psi_{S\setminus \calI_0}(\uu\vv) + \psi_{S\setminus \calI_0}(\vv\ww) - \psi_{S\setminus \calI_0}(\vv).
\end{split}
\]
This shows that $\gamma^{\psi_{S\setminus \calI_0}}_{\uu,\ww}(\vv) = 0$.
Therefore, as $\calI_0\in \Cliq(S | M)$ we obtain that $\widetilde{\gamma}_{\uu,\ww}(\vv)=0$.  Now as this holds for every $\vv\in W$ with $|\vv|>|\uu|+|\ww| + |S| +2$, we obtain that $\widetilde{\gamma}_{\uu,\ww}$ has finite support.
	\end{proof}

\section{Application B: Strong solidity of Hecke von Neumann algebras}\label{Sect=HeckeStronglySolid}

In this final section we obtain strong solidity results for Hecke von Neumann algebras. These are $q$-deformations of the group (von Neumann) algebra of a Coxeter group. If $q=1$ we retrieve the classical case of a group (von Neumann) algebra of a Coxeter group.

For the Hecke deformations our methods turn out to improve on existing strong solidity results. In \cite[Theorem 0.7]{klisseTopologicalBoundariesConnected2020} it was shown that for  Coxeter groups that are small at infinity, their Hecke von Neumann algebras satisfy the condition AO$^+$. If  such Hecke von Neumann algebras have the weak-$\ast$ completely bounded approximation property then they are strongly solid by \cite[Theorem A]{isonoExamplesFactorsWhich2015} (this is a generalisation of Theorem \ref{Thm=AOimpliesStronglySolid} from \cite{PopaVaesCrelle}). The weak-$\ast$ completely bounded approximation property was proved in \cite{caspersAbsenceCartanSubalgebras2020} for Hecke von Neumann algebras  of right-angled Coxeter groups; outside the right-angled case this is an open problem. Therefore right-angled Coxeter groups that are small at infinity have strongly solid Hecke von Neumann algebras.  It was proved in \cite{klisseTopologicalBoundariesConnected2020} that such   Coxeter groups are in fact free products of   abelian Coxeter groups; hence this result is somewhat more limited than one would hope for.

It is natural to ask whether these strong solidity results for Hecke von Neumann algebras apply to more general word hyperbolic Coxeter groups. In the group case ($q=1$)  this is exactly Theorem \ref{Thm=StrongSolidRightAngled}. However, the results from \cite{klisseTopologicalBoundariesConnected2020} and in particular  \cite[Corollary 3.17]{klisseTopologicalBoundariesConnected2020}  show that it is hard to extend current methods beyond free products of abelian Coxeter groups. A typical right-angled word hyperbolic Coxeter group that was not covered before this paper is given by
\begin{equation} \label{Eqn=FourPointsOnALine}
\langle \{ s_1, s_2, s_3, s_4 \} | M = (m_{i,j})_{i,j} \rangle \qquad \textrm{ with } m_{i,j} = 2 \textrm{ if } \vert i - j \vert = 2 \textrm{  and  } m_{i,j} = \infty \textrm{ otherwise}.
\end{equation}
In this section we prove that also the Hecke deformations of this Coxeter group satisfy AO$^+$ and are strongly solid. The precise statement is contained in Theorem \ref{Thm=HeckeStronglySolid}.

\subsection{Definition of Hecke algebras}\label{subsection:definition-Hecke-algebra}
Fix a (not necessarily right-angled, finite rank) Coxeter system $W = \langle S | M \rangle$.
 Let $q = (q_s)_{s\in S}$ with $q_s>0$ for $s\in S$ and such that $q_s = q_t$ whenever $s,t\in S$ are conjugate in $W$. In this text we shall call such tuples \textit{Hecke tuples}. Moreover, we will denote $p_s(q) = \frac{q_s-1}{\sqrt{q_s}}$.
We can as in \cite[Theorem 19.1.1]{davisGeometryTopologyCoxeter2008}
define for $s\in S$ the operators $T_s^{(q)}: \ell_2(W)\to \ell_2(W)$ given by
$$T_s^{(q)}(\delta_{\mathbf{w}}) = \begin{cases}
\delta_{s\mathbf{w}} & |s\mathbf{w}|> |\mathbf{w}|\\
\delta_{s\mathbf{w}} + p_s(q)\delta_{\mathbf{w}} & |s\mathbf{w}|< |\mathbf{w}|
\end{cases}.$$
For these operators we have
\[
\begin{split}
\langle T_s^{(q)}(\delta_{\mathbf{w}}),\delta_{\mathbf{z}}\rangle
&=
\langle \delta_{s\mathbf{w}},\delta_{\mathbf{z}}\rangle +
\langle p_s(q)\delta_{\mathbf{w}} ,\delta_{\mathbf{z}}\rangle\mathds{1}(|s\mathbf{w}|<|\mathbf{w}|)\\
&= 	\langle \delta_{\mathbf{w}},\delta_{s\mathbf{z}}\rangle +
\langle \delta_{\mathbf{w}}  , p_s(q)\delta_{\mathbf{z}}\rangle\mathds{1}(|s\mathbf{z}|<|\mathbf{z}|)\\
&=\langle \delta_{\mathbf{w}},T_s^{(q)}(\delta_{\mathbf{z}})\rangle
\end{split}
\]
that is $(T_{s}^{(q)})^* = T_s^{(q)}$.

Now, for a word $\mathbf{w}\in W$ with a reduced expression $\mathbf{w}=w_1 \ldots w_k$ we can set
$T_{\mathbf{w}}^{(q)} = T_{w_1}^{(q)}.....T_{w_k}^{(q)}$, which is well-defined by \cite[Theorem 19.1.1]{davisGeometryTopologyCoxeter2008}.
We note that we have $(T_{\mathbf{w}}^{(q)})^* = T_{\mathbf{w}^{-1}}^{(q)}$ and $T_{\mathbf{w}}^{(q)}(\delta_e) = \delta_{\mathbf{w}}$.
Furthermore for $s\in S$ and $\mathbf{w}\in W$ they satisfy the equations
\[
\begin{split}
T_s^{(q)}T_{\mathbf{w}}^{(q)} &= T_{s\mathbf{w}}^{(q)} + p_s(q)T_{\mathbf{w}}^{(q)}\mathds{1}(|s\mathbf{w}|<|\mathbf{w}|),\\
T_{\mathbf{w}}^{(q)}T_s^{(q)} &= T_{\mathbf{w}s}^{(q)} + p_s(q)T_{\mathbf{w}}^{(q)}\mathds{1}(|\mathbf{w}s|<|\mathbf{w}|).
\end{split}
\]
Note that the first equation holds by the proof of \cite[Theorem 19.1.1]{davisGeometryTopologyCoxeter2008}, and the second equation follows by taking the adjoint on both sides.

We will denote $\CC_q[W]$ for the $*$-algebra spanned by the linear basis $\{T_\mathbf{w}^{(q)} : \mathbf{w}\in W\}$.	
We furthermore denote $C_{r, q}^*(W)\subseteq B(\ell_2(W))$ for the reduced C$^\ast$-algebra obtained by taking the norm closure of $\CC_q[W]$. Finally, we define the Hecke von Neumann algebra $\mathcal{N}_q(W)$ (or simply $\mathcal{N}_q$) as the strong closure of $C_{r,q}^*(W)$. We equip the von Neumann algebra with the faithful finite trace $\tau(x) = \langle x\delta_e,\delta_e\rangle$. We note here that when $q=(q_s)_{s\in S}$ is taken as $q_s=1$ for $s\in S$, then $(\mathcal{N}_q,\tau)$ is simply the group von Neumann algebra $\calL(W)$ with canonical trace $\tau$. The group von Neumann algebra is thus a special case of a Hecke algebra.

\subsection{Coefficients for gradient bimodules of Hecke algebras}\label{subsection:Hecke-gradient-Sp-property}
We freely use the notation of Section \ref{subsection:definition-Hecke-algebra}. In particular we fix the tuple $q=(q_s)_{s\in S}$. We will simply write $T_{\mathbf{w}}$  instead of $T_\mathbf{w}^{(q)}$ and $p_s$ instead of $p_s(q)$.
   We let $\psi: W \rightarrow \mathbb{R}$ be proper and conditionally of negative type.  Define
\[
\Delta_\psi := \Delta_{\psi}^{(q)}:  \CC_q[W] \rightarrow \CC_q[W]: T_\mathbf{w} \mapsto \psi(\mathbf{w})T_\mathbf{w},
\]
and for $t \geq 0$,
\begin{equation}\label{Eqn=PhiHecke}
\Phi_t := \Phi_t^{(q)}:  \CC_q[W] \rightarrow \CC_q[W]: T_\mathbf{w} \mapsto \exp(-t \psi(\mathbf{w}) ) T_\mathbf{w}.
\end{equation}
We will now work under the following assumption.

\begin{assumption}\label{assumption:extension-ucp-Hecke-algebras}
  $\Phi_t, t \geq 0$  extends to a normal unital completely positive map   $\calN_q(W) \rightarrow \calN_q(W)$.
\end{assumption}

The main point of the assumption is the complete positivity of $\Phi_t$; the unitality is automatic since $\psi(e) = e$ and also the existence of a normal extension can usually be proved using a standard argument once one knows that $\Phi_t$ is bounded (see the final paragraph of the proof of \cite[Theorem 4.13]{caspersAbsenceCartanSubalgebras2020}).

The assumption holds in case $q = 1$ by Sch\"onberg's theorem and in case $W$ is right-angled by combining \cite[Corollary 3.4, Proposition 3.7]{caspersAbsenceCartanSubalgebras2020} and  \cite[Proposition 2.30]{caspersGraphProductsOperator2017}. Note that if the assumption holds then  $\calN_q(W)$ satisfies the Haagerup property since $\psi$ is proper. In general we do not know whether Assumption \ref{assumption:extension-ucp-Hecke-algebras} holds. In fact, it is not even known whether $\calN_q(W)$ has the Haagerup property unless $W$ is right-angled (see \cite[Section 3]{caspersAbsenceCartanSubalgebras2020}) or $q = 1$ (see \cite{bozejkoINFINITECOXETERGROUPS1988}).

It is standard  to check that if Assumption \ref{assumption:extension-ucp-Hecke-algebras} holds then $\Phi = (\Phi_t)_{t \geq 0}$ is a symmetric quantum Markov semi-group. For the continuity property note that $\Phi_t$ is a contractive semi-group on $L_2(\calN_q, \tau)$ and then use that on the unit ball of $\calN_q$ the strong topology equals the   $L_2(\calN_q, \tau)$-topology.

  We shall now investigate the gradient-$\cS_p, p \in [1, \infty]$ property for $\Phi$ with respect to the $\sigma$-weak dense subalgebra $\calA := \CC_q[W]$ of $\calN_q(W)$.  The set $\mathcal{A}_0:= \{T_s :  s\in S\}$ forms a self-adjoint set that generates the $\ast$-algebra $\mathcal{A}$. Therefore by Lemma \ref{lemma:reduction-gradient-Sp-to-generators}  in order to check the gradient-$\mathcal{S}_p$ property for $\Phi$ we only have to check that $\Psi^{T_u,T_w}$ given in Definition \ref{Dfn=GradientSp} is in $\mathcal{S}_p = \mathcal{S}_p(L_2(\calN_q(W)))$ for generators $u,w\in S$. To check  this  we shall make some calculations to obtain a simplified expression for $\Psi^{T_u,T_w}$.

Fix $u,w\in S$ and let $\mathbf{v}\in W$. We have by the multiplication rules that
\[
\begin{split}
T_u(T_\mathbf{v}T_w)	=& T_uT_{\mathbf{v}w} + T_uT_{\mathbf{v}}p_w\mathds{1}(|\mathbf{v}w|<|\mathbf{v}|)\\
= & T_{u\mathbf{v}w}+ p_uT_{\mathbf{v}w}\mathds{1}(|u\mathbf{v}w|<|\mathbf{v}w|)\\
 & +(T_{u\mathbf{v}}
+p_uT_{\mathbf{v}}\mathds{1}(|u\mathbf{v}|<|\mathbf{v}|))p_w\mathbf{1}(|\mathbf{v}w|<|\mathbf{v}|).
\end{split}
\]
We can now make the following calculations
\[
\begin{split}
\Delta_{\psi}(T_uT_\mathbf{v}T_w)	= & \psi(u\mathbf{v}w)T_{u\mathbf{v}w} + \psi(\mathbf{v}w)p_uT_{\mathbf{v}w}\mathds{1}(|u\mathbf{v}w|<|\mathbf{v}w|)\\ &+\psi(u\mathbf{v})T_{u\mathbf{v}}p_w\mathds{1}(|\mathbf{v}w|<|\mathbf{v}|)
+\psi(\mathbf{v})p_uT_{\mathbf{v}}p_w\mathds{1}(|u\mathbf{v}|<|\mathbf{v}|)\mathbf{1}(|\mathbf{v}w|<|\mathbf{v}|),\\
T_u\Delta_{\psi}(T_\mathbf{v})T_w	=& \psi(\mathbf{v})(T_{u\mathbf{v}w} + p_uT_{\mathbf{v}w}\mathds{1}(|u\mathbf{v}w|<|\mathbf{v}w|))\\
&+\psi(\mathbf{v})(T_{u\mathbf{v}}
+p_uT_{\mathbf{v}}\mathds{1}(|u\mathbf{v}|<|\mathbf{v}|))p_w\mathbf{1}(|\mathbf{v}w|<|\mathbf{v}|),\\
T_u\Delta_{\psi}(T_\mathbf{v}T_w) &= \psi(\mathbf{v}w)T_uT_{\mathbf{v}w} + \psi(\mathbf{v})T_uT_{\mathbf{v}}p_w\mathds{1}(|\mathbf{v}w|<|\mathbf{v}|),\\
= & \psi(\mathbf{v}w)(T_{u\mathbf{v}w} + p_uT_{\mathbf{v}w}
\mathds{1}(|u\mathbf{v}w|<|\mathbf{v}w|))\\
&+ \psi(\mathbf{v})(T_{u\mathbf{v}} + p_uT_\mathbf{v}\mathds{1}(|u\mathbf{v}|<|\mathbf{v}|))p_w\mathds{1}(|\mathbf{v}w|<|\mathbf{v}|),\\
\Delta_{\psi}(T_uT_\mathbf{v})T_w &= \psi(u\mathbf{v})T_{u\mathbf{v}}T_w + \psi(\mathbf{v})p_uT_{\mathbf{v}}T_w\mathds{1}(|u\mathbf{v}|<|\mathbf{v}|)\\
= & \psi(u\mathbf{v})(T_{u\mathbf{v}w} + T_{u\mathbf{v}}p_w
\mathds{1}(|u\mathbf{v}w|<|u\mathbf{v}|))\\
&+ \psi(\mathbf{v})p_u(T_{\mathbf{v}w} + T_\mathbf{v}p_w\mathds{1}(|\mathbf{v}w|<|\mathbf{v}|))\mathds{1}(|u\mathbf{v}|<|\mathbf{v}|).
\end{split}
\]
Let $\psi_{S}$ be again the word length function on $W$.  Now by collecting all previous terms we get
\[
\begin{split}
\Psi^{T_u,T_w} (T_{\mathbf{v}}) = & \Delta_{\psi}(T_uT_\mathbf{v}T_w) + T_u\Delta_{\psi}(T_\mathbf{v})T_w  - T_u\Delta_{\psi}(T_\mathbf{v}T_w)- \Delta_{\psi}(T_uT_\mathbf{v})T_w\\
= & (\psi(u\mathbf{v}w) + \psi(\mathbf{v}) -\psi(\mathbf{v}w) - \psi(u\mathbf{v}))T_{u\mathbf{v}w}\\
&+[(\psi(u\mathbf{v}) + \psi(\mathbf{v}) - \psi(\mathbf{v}))\mathds{1}(|\mathbf{v}w|<|\mathbf{v}|)
-\psi(u\mathbf{v})\mathds{1}(|u\mathbf{v}w|<|u\mathbf{v}|)]
T_{u\mathbf{v}}p_w\\
&+[(\psi(\mathbf{v}w) + \psi(\mathbf{v}) - \psi(\mathbf{v}w))\mathds{1}(|u\mathbf{v}w|<|\mathbf{v}w|)
-\psi(\mathbf{v})\mathds{1}(|u\mathbf{v}|<|\mathbf{v}|)]
p_uT_{\mathbf{v}w}\\
&+(\psi(\mathbf{v}) + \psi(\mathbf{v}) - \psi(\mathbf{v}) - \psi(\mathbf{v})) \mathds{1}(|u\mathbf{v}|<|\mathbf{v}|)\mathds{1}(|\mathbf{v}w|<|\mathbf{v}|)p_uT_{\mathbf{v}}p_w\\
= & \gamma_{u,w}^{\psi}(\mathbf{v})T_{u\mathbf{v}w}\\
&+ \psi(u\mathbf{v})(\mathds{1}(|\mathbf{v}w|<|\mathbf{v}|) - \mathds{1}(|u\mathbf{v}w|<|u\mathbf{v}|))T_{u\mathbf{v}}p_w \\
&+\psi(\mathbf{v})(\mathds{1}(|u\mathbf{v}w|<|\mathbf{v}w|) - \mathds{1}(|u\mathbf{v}|<|\mathbf{v}|)) p_uT_{\mathbf{v}w}\\
= & \gamma_{u,w}^{\psi}(\mathbf{v})T_{u\mathbf{v}w}\\
&+ \psi(u\mathbf{v})\left(\frac{|\vv| - |\vv w|+1}{2} -\frac{|u\vv| - |u\vv w| +1}{2}\right)T_{u\mathbf{v}}p_w \\
&+\psi(\mathbf{v})\left(\frac{|\vv w| - |u\vv w|+1}{2} -\frac{|\vv| - |u\vv| +1}{2}\right)p_uT_{\mathbf{v}w}\\
= & \gamma_{u,w}^{\psi}(\mathbf{v})T_{u\mathbf{v}w}
+ \frac{1}{2}\left(|u\vv w| + |\vv| - |\vv w| -|u\vv|\right) (\psi(u\mathbf{v})T_{u\mathbf{v}}p_w -\psi(\vv)p_uT_{\vv w})\\
= & \gamma_{u,w}^{\psi}(\mathbf{v})T_{u\mathbf{v}w}
+ \frac{1}{2}\gamma_{u,w}^{\psi_{S}}(\vv) (\psi(u\mathbf{v})T_{u\mathbf{v}}p_w -\psi(\vv)p_uT_{\vv w}).
\end{split}
\]
Now when  $u\vv \not=\vv w$ we have by Lemma \ref{lemma:shifting-identity} that $\gamma_{u,w}^{\psi_{S}}(\vv)=0$.
When  $u\mathbf{v} = \mathbf{v}w$  we have  $|\frac{1}{2}\gamma_{u,w}^{\psi_S}(\vv)|=\psi_{S}(u)=1$. In this case the elements $u$ and $w$ are also conjugate and therefore $p_u=p_w$.
Combining these facts we obtain the simplified formula
\begin{equation}
\Psi^{T_u,T_w}(T_{\vv}) = \gamma_{u,w}^{\psi}(\vv)T_{u\vv w} + \frac{1}{2}\gamma_{u,w}^{\psi_{S}}(\vv)(\psi(u\vv)-\psi(\vv))T_{u\vv}p_w.
\end{equation}
We will proceed under the further assumption that $\psi$ is a length function.
\begin{assumption}
We shall assume from this point that the proper, conditionally of negative type function $\psi: W \rightarrow \mathbb{R}$ is also a length function.
\end{assumption}

Using the fact that $\{T_{\vv}\}_{\vv\in W}$ is an orthonormal basis for $L_2(\calN_q(W),\tau)$ we obtain that for the $\calS_2$-norm of $\Psi^{T_u,T_w}$ we have the following bound
\begin{equation}\label{Eqn=HeckeEstimateNew}
\begin{split}
\|\Psi^{T_u,T_w} \|_{\mathcal{S}_2}^2 = & \sum_{\vv\in W}\langle \Psi^{T_u,T_w} (T_{\vv}),\Psi^{T_u,T_w} (T_{\vv})\rangle\\
= &\sum_{\vv\in W}\left[|\gamma_{u,w}^{\psi}(\vv)|^2 + \frac{1}{4}|\gamma_{u,w}^{\psi_{S}}(\vv)|^2|\psi(u\vv)-\psi(\vv)|^2|p_u|^2\right]\\
\leq &\|\gamma_{u,w}^{\psi}\|_{\ell_2(W)}^2 + \frac{1}{4}|\psi(u)|^2p_u^2\|\gamma_{u,w}^{\psi_{S}}\|_{\ell_2(W)}^2.
\end{split}
\end{equation}
We are then thus interested in functions $\psi$ for which this bound is finite for all $u,w\in S$.\\

\begin{theorem}\label{Thm=SpHecke}
Let $W = \langle S | M \rangle$ be a right angled Coxeter group and let $q = (q_s)_{s\in S}$ with $q_s >0$. Assume that the elements in
\begin{equation}\label{Eqn=HeckeClique}
\mathcal{I} := \{ r \in S : \exists s,t \in S \textrm{ such that } m_{r,s} = m_{r,t} = 2 \textrm{ and } m_{s,t} = \infty \}
\end{equation}
commute, i.e. $\mathcal{I} \in \Cliq(S \mid M)$. Suppose that   $\psi := \psi_{S \backslash \mathcal{I}}$ satisfies Assumption \ref{assumption:extension-ucp-Hecke-algebras}.
 Then the QMS on $\calN_q(W)$ determined by \eqref{Eqn=PhiHecke} associated with $\psi_{S \backslash \mathcal{I}}$ is gradient-$\cS_2$.
\end{theorem}
\begin{proof}
We have shown already in Theorem \ref{lemma:gradient-Sp-weighted-word-length-right-angled-Coxeter-group} that for $u,w\in S$ we have that $\gamma_{u,w}^{\psi_{S \backslash \calI}}\in \ell_2(\Gamma)$.
 Now if $u \in \mathcal{I}$ then $\psi_{ S \backslash \mathcal{I} }  (u) = 0$ and hence by \eqref{Eqn=HeckeEstimateNew},
 \[
 \|\Psi^{T_u,T_w} \|_{\mathcal{S}_2}^2 \leq \|\gamma_{u,w}^{\psi_{ S \backslash \mathcal{I} } }\|_{\ell_2(W)}^2 < \infty.
 \]
 If $u \in S \backslash \mathcal{I}$ then $\psi_{ S \backslash \mathcal{I} }  (u) = 1$
 and therefore by Lemma \ref{lemma:shifting-identity} we have
\[
|\gamma_{u,w}^{\psi_{S \backslash \calI}}(\vv)| = 2\psi_{ S \backslash  \mathcal{I} }(u)\mathds{1}(u\vv = \vv w) = 2\psi_{S }(u)\mathds{1}(u\vv = \vv w) = |\gamma_{u,w}^{\psi_{S}}(\vv)|.
\]
This means that in this case $\gamma_{u,w}^{\psi_{S \backslash \calI}} = \gamma_{u,w}^{\psi_{S }}\in \ell_2(\Gamma)$. We conclude from \eqref{Eqn=HeckeEstimateNew} that
\[
\|\Psi^{T_u,T_w} \|_{\mathcal{S}_2}^2\leq\|\gamma_{u,w}^{\psi_{S \backslash \calI}}\|_{\ell_2(\Gamma)}^2 + \frac{1}{4}  p_u^2
\cdot\|\gamma_{u,w}^{\psi_{S}}\|_{\ell_2(\Gamma)}^2<\infty.
\]
\end{proof}

\begin{theorem}\label{Thm=HeckeStronglySolid}
Let $W = \langle S | M \rangle$ be a right angled Coxeter group and let $q = (q_s)_{s\in S}$ with $q_s >0$. Assume that   \eqref{Eqn=HeckeClique} is contained in $\Cliq(S \mid M)$. Then $\calN_q(W)$ satisfies AO$^+$ and is strongly solid.
\end{theorem}
\begin{proof}
Theorem \ref{Thm=SpHecke} shows that the QMS $\Phi$ on $\calN_q(W)$ associated with the length function $\psi_{S \backslash \mathcal{I}}$ is gradient-$\cS_2$ . Therefore by Theorem \ref{Thm=GradientCoefficients}  we see that a dense set of coefficients of the associated gradient bimodule $L_2(\calN_q(W), \tau)_\nabla$ is in $\cS_2$.
Note that Theorem \ref{Thm=GradientCoefficients} is stated only for groups, but a straightforward adaptation of the computations in the proof yields the same result for Hecke algebras.
  Hence the gradient bimodule is quasi-contained in the coarse bimodule of $\calN_q(W)$ by \cite[Theorem 3.9]{caspersL2CohomologyDerivationsQuantum2021} (see also Proposition \ref{proposition:quasicontainment}). The Riesz transform  is then an isometry $R_\Phi:  L_2(\calN_q(W), \tau) \rightarrow L_2(\calN_q(W), \tau)_\nabla$. The kernel of $R_\Phi$ is given by the space spanned by the vectors $T_\ww$ with $\ww$ in the (finite) group generated by $\mathcal{I}$.
 Essentially in the same way as in the group case ($q = 1$) one checks that $\Phi$ is filtered with subexponential growth. Therefore by Theorem \ref{Thm=AlmostBimodularRiesz}  we see that  $R_\Phi$ is almost bimodular. By \cite[Theorem 6.1]{caspersGraphProductKhintchine2021} $C_{r,q}^\ast(W)$ is exact and hence locally reflexive \cite{brownAlgebrasFiniteDimensionalApproximations2008}.  We may now invoke Theorem \cite[Proposition 5.2]{caspersL2CohomologyDerivationsQuantum2021} (see also Theorem \ref{Thm=RieszImpliesAO}) to conclude that $\calN_q(W)$ satisfies AO$^+$. By \cite[Theorem A]{caspersAbsenceCartanSubalgebras2020} $\calN_q(W)$ satisfies the weak-$\ast$ completely bounded approximation property. Hence \cite[Theorem A]{isonoExamplesFactorsWhich2015}  (see also Theorem \ref{Thm=AOimpliesStronglySolid}) shows that $\calN_q(W)$ is strongly solid.
\end{proof}

\begin{remark}
The strong solidity result of Theorem \ref{Thm=HeckeStronglySolid} can also be proved by combining the results in this paper with the methods of \cite{caspersGradientFormsStrong2021}, \cite{OzawaPopaAJM}, \cite{Peterson} without using condition AO$^+$.
\end{remark}

\begin{remark}
The set \eqref{Eqn=HeckeClique} can be understood as all elements in $S$ that belong to exactly one maximal clique.
\end{remark}

\section{Open problems}\label{Sect=Problems}

 We list two natural problems which we believe are open.

 \begin{problem}
Consider a Coxeter system $W = \langle S | M \rangle$ and $q = (q_s)_{s\in S}$ with $q_s>0, s\in S$   such that $q_s = q_t$ whenever $s,t\in S$ are conjugate in $W$. Does the Hecke von Neumann algebra $\calN_q(W)$ have the Haagerup property and/or the weak-$\ast$ completely bounded approximation property? An affirmative answer for both properties is known in case  $q_s =1$ for all $s \in S$ \cite{bozejkoINFINITECOXETERGROUPS1988}, \cite{fendlerWeakAmenabilityCoxeter2002}, \cite{Janus} or in case $W = \langle S | M \rangle$ is right-angled \cite{caspersAbsenceCartanSubalgebras2020}. For all other cases these properties are open. In particular we do not know in which generality Assumption \ref{assumption:extension-ucp-Hecke-algebras} holds  for $\psi = \psi_S$ the (unweighted) word length function.
 \end{problem}

\begin{problem}
For $W = \langle S | M \rangle$ a right-angled word hyperbolic Coxeter system and $q = (q_s)_{s\in S}$ with $q_s>0, s\in S$. Is $\calL_q(W)$ strongly solid? The cases obtained in Theorem \ref{Thm=HeckeStronglySolid} are word hyperbolic but do not exhaust all word hyperbolic right-angled Coxeter groups. In case $q_s = 1, s \in S$ the tensor product techniques from Section \ref{Sect=SolidTensor} allows one to improve the results of Section \ref{Sect=HeckeStronglySolid} to all  word hyperbolic right-angled Coxeter groups. However, such  tensor products of bimodules are unavailable unless $q_s =1, s\in S$ by the absence of a suitable comultiplication for Hecke algebras.
\end{problem}

\end{document}